\documentclass[english,a4paper]{smfart}
\usepackage{smfmulti}
\usepackage{macros}

\DeclareMathOperator{\Br}{Br}
\DeclareMathOperator{\coker}{coker}
\DeclareMathOperator{\Div}{Div}
\DeclareMathOperator{\divi}{div}
\DeclareMathOperator{\Gal}{Gal}
\DeclareMathOperator{\Hom}{Hom}
\DeclareMathOperator{\Id}{Id}
\DeclareMathOperator{\im}{im}
\DeclareMathOperator{\NS}{NS}
\DeclareMathOperator{\Pic}{Pic}
\DeclareMathOperator{\pr}{pr}
\DeclareMathOperator{\sgn}{sgn}
\DeclareMathOperator{\Spec}{Spec}
\DeclareMathOperator{\Sym}{Sym}
\DeclareMathOperator{\Val}{Val}
\DeclareMathOperator{\Vol}{Vol}

\newcommand{\etale}{_\textinmath{\'et}}
\newcommand{\eff}{_\textinmath{eff}}

\newcommand{\zerofour}{{\{0,\dots,4\}}}
\newcommand{\onefour}{{\{1,2,3,4\}}}
\newcommand{\DDelta}{{\boldsymbol\Delta}}
\newcommand{\badp}{{\matheusm S}}
\newcommand{\NN}{\mathbf Z_{\geq 0}}
\newcommand{\Sp}{{\matheusm S_p}}
\newcommand{\primes}{{\mathcal P}}
\newcommand{\Ceff}{{C\eff}}
\newcommand{\mset}{,\ }
\newcommand{\bmset}{,\ }

\newcommand{\ZZ}{{\mathbf Z}}
\newcommand{\QQ}{{\mathbf Q}}
\newcommand{\KK}{{\mathbf K}}
\newcommand{\RR}{{\mathbf R}}
\newcommand{\Ga}{{\mathcal G}}
\newcommand{\Adeles}{{\boldsymbol A}}
\newcommand{\ci}{{\mathsl i}}

\newcommand{\Aff}{{\mathbf A}}
\newcommand{\Proj}{{\mathbf P}}
\newcommand{\torsQ}{{\matheusm T}}
\newcommand{\torsZ}{{\mathcal T}}
\newcommand{\torscZ}{{\mathcal Y}}
\newcommand{\torsiQ}{{\matheusm T}_\inter}
\newcommand{\torsiZ}{{\mathcal T}_\inter}
\newcommand{\torsicZ}{{\mathcal Y}_\inter}
\newcommand{\TNS}{T_{\NS}}
\newcommand{\Tspl}{T_{\inter}}
\newcommand{\Gm}{{\mathbf G}_m}
\newcommand{\GmZ}{{\mathbf G}_{m,\ZZ}}
\newcommand{\GmQ}{{\mathbf G}_{m,\QQ}}
\newcommand{\RProj}{{\matheusm P}}

\newcommand{\CU}{{\mathcal U}}

\newcommand{\mm}{{\boldsymbol{m}}}
\newcommand{\fa}{\mathfrak{a}}
\newcommand{\faa}{{\boldsymbol{\mathfrak a}}}
\newcommand{\fb}{\mathfrak{b}}
\newcommand{\fbb}{{\boldsymbol{\mathfrak b}}}
\newcommand{\bb}{{\boldsymbol{\mathfrak b}}}

\newcommand{\ee}{{\boldsymbol{e}}}
\newcommand{\ff}{{\boldsymbol{f}}}
\newcommand{\nn}{{\boldsymbol{n}}}
\newcommand{\zz}{{\boldsymbol{z}}}

\newcommand{\rr}{{\boldsymbol{r}}}
\newcommand{\sss}{{\boldsymbol{s}}}
\newcommand{\llambda}{{\boldsymbol{\lambda}}}

\newcommand{\formon}[1]{{{\breve\omega}_{#1}}}
\newcommand{\measure}{\boldsymbol\omega}

\newcommand{\card}{\sharp}
\newcommand{\meas}{{\boldsymbol\omega}}
\newcommand{\tors}{\textinmath{tors}}
\newcommand{\inter}{\textinmath{spl}}
\newcommand{\oone}{{\boldsymbol 1}}
\newcommand{\dual}{^\vee}
\newcommand{\sha}{{\cyrille X}}
\let\mod\bmod
\newcommand{\rd}{\,\textinmath{d}}

\newcommand{\lab}{\label}
\newcommand{\ma}{\mathbf}

\newcommand{\fD}{{\mathfrak{D}}}
\newcommand{\mcal}{\mathcal}

\newcommand{\sfg}{\mathsf{\Gamma}}
\newcommand{\sfl}{\mathsf{\Lambda}}

\newcommand{\D}{\Delta}
\newcommand{\ZZp}{\mathbf{Z}_{>0}}
\newcommand{\UU}{\mathcal{U}}
\newcommand{\ve}{\varepsilon}
\renewcommand{\rho}{\varrho}

\newcommand{\mnu}{\boldsymbol{\nu}}
\newcommand{\mal}{\boldsymbol{\lambda}}
\newcommand{\mmu}{\boldsymbol{\mu}}
\newcommand{\bd}{\mathbf{d}}
\newcommand{\la}{\lambda}

\newcommand{\N}{\mathrm{N}}

\date{\today}
\title[Manin's conjecture for Ch\^atelet surfaces]{On Manin's
conjecture\\for a family of Ch\^atelet surfaces}
\author{R\'egis de la Bret\`eche}
\address{Institut de Math\'ematiques de Jussieu,  UMR 7586
Case 7012, Universit\'e Paris 7 -- Denis Diderot
2, place Jussieu, F-75251 Paris cedex 05, France}
\email{breteche@math.jussieu.fr}
\author{Tim Browning}
\address{School of Mathematics,  
University of Bristol, Bristol BS8 1TW, England}
\email{t.d.browning@bristol.ac.uk}
\author{Emmanuel Peyre}
\address{Institut Fourier\\
UFR de Math\'ematiques, UMR 5582\\
Universit\'e de Grenoble I et CNRS\\
BP 74\\ 38402 Saint-Martin d'H\`eres CEDEX\\ France}
\email{Emmanuel\!.Peyre@ujf-grenoble.fr}
\urladdr{http://www-fourier.ujf-grenoble.fr/\~{}peyre}
\begin{document}
\begin{abstract}
The Manin conjecture is established for 
Ch\^atelet surfaces over $\QQ$ arising as minimal proper smooth models of 
the surface 
\[Y^2+Z^2=f(X)\]
in $\Aff_\QQ^3$, 
where $f\in \ZZ[X]$ is a totally reducible polynomial of
degree $3$ without repeated roots.  These surfaces do not satisfy weak
approximation. 
\end{abstract}
\ifx\altabstract\undefined\else
\begin{altabstract}
Nous d\'emontrons la conjecture de Manin pour les
surfaces de Ch\^atelet sur $\QQ$ obtenues comme mod\`ele minimal propre
et lisse de la surface affine d'\'equation
\[Y^2+Z^2=f(X)\]
o\`u $f\in \ZZ[X]$ est un polyn\^ome scind\'e \`a racines distinctes.
Ces surfaces ne satisfont pas l'appro\-xi\-ma\-tion faible.
\end{altabstract}
\fi
\maketitle
\tableofcontents
\section{Introduction}

The purpose of this paper is to prove Manin's conjecture
about points of bounded height for a family of Ch\^atelet surfaces
over $\QQ$. These surfaces have been considered by F.~Ch\^atelet
in \cite{chatelet:points} and \cite{chatelet:points2},
by V.~A.~Iskovskikh~\cite{iskovskih:hasse},
by D.~Coray and M.~A. Tsfasman \cite{coraytsfasman:delpezzo},
and by {J.-L.} Colliot-Th\'el\`ene, J.-J. Sansuc, and P. Swinnerton-Dyer
in \cite{ctssd:chatelet1} and \cite{ctssd:chatelet2}, among others.

The surfaces considered here are smooth proper models
of the affine surfaces given in $\Aff^3_\QQ$ by an equation
of the form
\begin{equation*}
Y^2+Z^2=X(a_3X+b_3)(a_4X+b_4),
\end{equation*}
for suitable $a_3$, $b_3$, $a_4$, $b_4\in\ZZ$.

It is important to note that the surfaces we consider
do not satisfy weak approximation, 
the lack of which
is explained by the Brauer-Manin obstruction, as described in
\cite{ctssd:chatelet1} and \cite{ctssd:chatelet2}.
Up to now, the only cases for which Manin's principle was
proven despite weak approximation not holding
were obtained using harmonic analysis and required the
action of an algebraic group on the variety with
an open orbit. The method used in this paper is completely
different. Following ideas of P.~Salberger \cite{salberger:tamagawa},
we use versal torsors introduced by
Colliot-Th\'el\`ene and Sansuc in \cite{cts:predescente2},
\cite{cts:descente1}, and \cite{cts:descente2}
to estimate the number of rational points of bounded
height on the surface.

This paper is organised as follows: in section~\ref{section.geometry},
we recall some facts about the geometry of the surfaces. In 
section~\ref{section.result}, we define the height and state
our main result. Section~\ref{section.torsors} contains
the description of the versal torsors we use. In 
section~\ref{section.jumpingup}, we describe the lifting of rational
points to the versal torsors. This lifting reduces the initial
problem to the estimation of some arithmetic sums denoted by
$\CU(T)$. The following sections contain
the key analytical tools used in the proof.
In section~\ref{section.bound} we give a uniform upper bound
for $\CU(T)$ and in section~\ref{section.estimate}
an asymptotic formula for it.
The last section is devoted to an interpretation of the leading
constant.

Let us fix some notation for the remainder of this text.

\begin{notaconv}
If $k$ is a field, we denote by $\overline k$ an algebraic closure
of $k$. For any variety $X$ over $k$ and any $k$-algebra $A$, we denote
by $X_A$ the product $X\times_{\Spec(k)}\Spec(A)$ and by $X(A)$
the set $\Hom_{\Spec(k)}(\Spec(A),X)$. We also put $\overline X=X_{\overline k}$.
The cohomological Brauer group of $X$ is defined as 
$\Br(X)=H^2\etale(X,\mathbf G_m)$, where $\mathbf G_m$ denotes
the multiplicative group.
The projective space of dimension $n$
over $A$ is denoted by $\Proj^n_A$ and the
affine space by $\Aff^n_A$. For any $(x_0,\dots,x_n)\in k^{n+1}\setminus\{0\}$
we denote by $(x_0:\dots:x_n)$ its image in $\Proj^n(k)$.
\end{notaconv}
\section{A family of Ch\^atelet surfaces}
\label{section.geometry}

Let us fix $a_1$, $a_2$, $a_3$, $a_4$, $b_1$, $b_2$,
$b_3, b_4\in\ZZ$ such that 
\[\Delta_{i,j}=\left\vert\begin{matrix}
a_i&a_j\\
b_i&b_j
\end{matrix}\right\vert\neq 0\]
for any $i,j\in\onefour$ with $i\neq j$.
We then consider the linear forms $L_i$ defined
by $L_i(U,V)=a_iU+b_iV$ for $i\in\onefour$
and define the hypersurface $S_1$ of $\Proj^2_\QQ\times\Aff^1_\QQ$
given by the equation
\begin{equation*}
X^2+Y^2=T^2\prod_{i=1}^4L_i(U,1)
\end{equation*}
and the hypersurface $S_2$ given by the equation
\begin{equation*}
{X'}^2+{Y'}^2={T'}^2\prod_{i=1}^4L_i(1,V).
\end{equation*}
Let $U_1$ be the open subset of $S_1$ defined by $U\neq 0$
and $U_2$ be the open subset of $S_2$ defined by $V\neq 0$.
The map $\Phi:U_1\to U_2$ which maps $((X:Y:T),U)$
onto $((X:Y:U^2T),1/U)$ is an isomorphism and we define $S$
as the surface obtained by glueing $S_1$ to $S_2$ using the isomorphism
$\Phi$. The surface $S$ is a smooth projective surface and is
a particular case of a Ch\^atelet surface. The geometry of such surfaces
has been described by J.-L. Colliot-Th\'el\`ene, J.-J. Sansuc
and P. Swinnerton-Dyer in \cite[\S7]{ctssd:chatelet2}. For
the sake of completeness, let us recall part of this description which will
be useful for the description of versal torsors. 

The maps $S_1\to\Proj^1_\QQ$ (resp. $S_2\to\Proj^1_\QQ$)
which maps $((X:Y:T),U)$ onto $(U:1)$
(resp. $((X':Y':T'),V)$ onto $(1:V)$) glue together
to give a conic fibration $\pi:S\to\Proj^1_\QQ$
with four degenerate fibres over the points given by
$P_i=(-b_i:a_i)\in\Proj^1(\QQ)$ for $i\in\onefour$.
In fact, the glueing of $\Proj^2_\QQ\times\Aff^1_\QQ$ to
$\Proj^2_\QQ\times\Aff^1_\QQ$ through the map
\begin{equation}\label{equ.glueing}
((X:Y:T),U)\mapsto ((X:Y:U^2T),1/U)
\end{equation}
gives the projective bundle\footnote{We define
here $\Proj(\mathcal O^2\oplus\mathcal O(-2))$ as the projective
bundle associated to the sheave of graded commutative algebras
$\underline{\Sym}(\mathcal O^2\oplus\mathcal O(2))$. In other
words the fibre over a point is given by the lines in the
fibre of the vector bundle and not by the hyperplanes.} 
$\RProj=\Proj(\mathcal O^2\oplus\mathcal O(-2))$
over $\Proj^1_\QQ$ and $S$ may be seen as a hypersurface
in that bundle.

Over $\QQ(\ci)$, if $\xi\in\{-\ci,\ci\}$, the map 
$\Aff_{\QQ(\ci)}\to {S_1}_{\QQ(\ci)}$ given by $u\mapsto ((\xi:1:0),U)$
extends to a section $\sigma_\xi$ of $\pi$. The surface
$S_{\QQ(\ci)}$ contains $10$ exceptional curves, that is irreducible
curves with negative self-intersection. Eight of them are given
in $S_{\QQ(\ci)}$ by the following equations
\begin{equation*}
D_j^\xi:\qquad L_j(\pi(P))=0\quad\text{and}\quad X-\xi Y=0
\end{equation*}
for $\xi\in\{-\ci,\ci\}$ and $j\in\onefour$; the last ones
correspond to the section $\sigma_\xi$ and are given by the\
equations
\begin{equation*}
E^\xi:\qquad T=0\quad\text{and}\quad X-\xi Y=0.
\end{equation*}
Here $X$, $Y$ and $T$ are seen as sections of 
$\mathcal O_{\RProj}(1)$.
Let us denote by $\Ga$ the Galois group of $\QQ(\ci)$
over $\QQ$ and by $z\mapsto \overline z$ the nontrivial element
in $\Ga$. Then we have
\[\overline{E^\xi}=E^{\overline\xi}\quad\text{and}\quad
\overline{D_j^\xi}=D_j^{\overline\xi}\]
for $\xi\in\{-\ci,\ci\}$ and $j\in\onefour$. We shall also
write $D_j^+$ (resp. $D_j^-$, $E^+$, $E^-$)
for $D_j^{\ci}$ (resp. $D_j^{-\ci}$, $E^{\ci}$, $E^{-\ci}$).
The intersection multiplicities of these divisors are given by
\[(E^\xi,E^\xi)=-2,\quad (D_j^\xi,D_j^\xi)=-1,\quad (D_j^\xi,D_j^{-\xi})=1,
\quad(E^\xi,D_j^\xi)=1,\]
where $\xi\in\{-\ci,\ci\}$, and $j\in\onefour$, all other 
intersection multiplicities
being equal to~$0$. These intersections are summarized in 
figure~\ref{fig.intersection}.
\begin{figure}[ht]
\[\begin{pspicture}(0,0)(10,6)
\psline(1,1)(9,1)
\rput(0.7,1){$E^-$}
\rput(8.7,1.2){$-2$}
\psline(1,5)(9,5)
\rput(0.7,5){$E^+$}
\rput(8.7,4.8){$-2$}
\psbezier(2,0)(2,1)(2,3)(1,4)
\rput(1.7,0){$D_1^-$}
\rput(2.1,2.1){-1}
\psbezier(4,0)(4,1)(4,3)(3,4)
\rput(3.7,0){$D_2^-$}
\rput(4.1,2.1){-1}
\psbezier(6,0)(6,1)(6,3)(5,4)
\rput(5.7,0){$D_3^-$}
\rput(6.1,2.1){-1}
\psbezier(8,0)(8,1)(8,3)(7,4)
\rput(7.7,0){$D_4^-$}
\rput(8.1,2.1){-1}
\psbezier(2,6)(2,5)(2,3)(1,2)
\rput(1.7,6){$D_1^+$}
\rput(2.1,3.9){-1}
\psbezier(4,6)(4,5)(4,3)(3,2)
\rput(3.7,6){$D_2^+$}
\rput(4.1,3.9){-1}
\psbezier(6,6)(6,5)(6,3)(5,2)
\rput(5.7,6){$D_3^+$}
\rput(6.1,3.9){-1}
\psbezier(8,6)(8,5)(8,3)(7,2)
\rput(7.7,6){$D_4^+$}
\rput(8.1,3.9){-1}
\end{pspicture}\]
\caption{Intersection multiplicities}
\label{fig.intersection}
\end{figure}
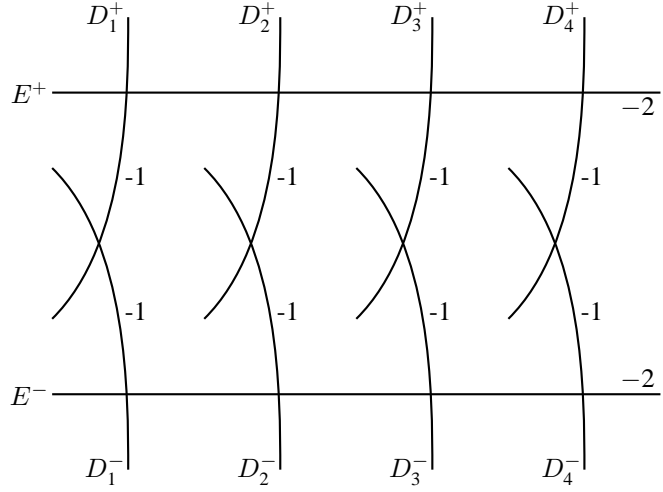
The geometric Picard group of $S$, that is $\Pic(\overline S)$,
is isomorphic to $\Pic(S_{\QQ(\ci)})$ and is generated by these exceptional
divisors with the relations
\begin{equation}\label{equ.pic.fibre}
[D_j^+]+[D_j^-]=[D_k^+]+[D_k^-]
\end{equation}
for $j,k\in\onefour$ and
\begin{equation}\label{equ.pic.can}
[E^+]+[D_j^+]+[D_k^+]=[E^-]+[D_l^-]+[D_m^-]
\end{equation}
whenever $\{j,k,l,m\}=\onefour$.
In particular, a basis of $\Pic(S_{\QQ(\ci)})$ is
given by the family 
\[([E^+],[D_1^+],[D_2^+],[D_3^+],[D_4^+],[D_1^-])\]
and the rank of the geometric Picard group of $S$ is equal
to $6$.  
Using the fact that $\Pic(S)=(\Pic(S_{\QQ(\ci)}))^{\Ga}$ it is 
easy to deduce that $\Pic(S)$ has rank $2$.

The class of the anticanonical line bundle is given by
\[\omega_S^{-1}=2E^++\sum_{j=1}^4D_j^+=2E^-+\sum_{j=1}^4D_j^-.\]
Indeed, by the adjunction formula, for any curve $C$ in $S$ of genus
$g$, one has the relation
$[C].([C]+\omega_S)=2g-2$. Therefore if $\xi\in\{-\ci,\ci\}$
and $j\in\onefour$,
\[[D_j^\xi].\omega_S^{-1}=1\quad\text{and}\quad[E^\xi].\omega_S^{-1}=0.\]
It is worthwhile noting
that $\omega_S^{-1}=
\mathcal O_{\RProj}(1)$.

\begin{lemma}
\label{lemma.sectionsomega}%
Using the trivialisation described by \eqref{equ.glueing},
the $5$-tuple of functions 
\[(T,UT,U^2T,X,Y)\] 
gives
a basis of $\Gamma(S,\omega_S^{-1})$.
\end{lemma}
\begin{proof}
Let $C$ be a generic divisor in $\vert\omega_S^{-1}\vert$. Then
$C$ is a smooth irreducible curve; let
$g_C$ be its genus. According to the adjunction formula,
we have that $2g_C-2=\omega_S.(\omega_S-\omega_S)=0$. Thus $g_C=1$.
The exact sequence of sheaves
\[0\longrightarrow\mathcal O_S\longrightarrow\omega_S^{-1}
\longrightarrow\omega_S^{-1}\otimes\mathcal O_C\longrightarrow 0\]
gives an exact sequence
\[0\longrightarrow H^0(S,\mathcal O_S)\longrightarrow
H^0(S,\omega_S^{-1})\longrightarrow H^0(C,{\omega_S^{-1}}_{\vert C})
\longrightarrow H^1(S,\mathcal O_S).\]
But $S$ is geometrically rational and $H^1(S,\mathcal O_S)=\{0\}$.
We get that
\[h^0(S,\omega_S^{-1})=1+h^0(C,{\omega_S^{-1}}_{\vert C}).\]
Let $D={\omega_S^{-1}}_{\vert C}$. We have that $\deg(D)=4$ and
$\deg(\omega_C-D)=-4$ since $\omega_C=0$.
Applying Riemann--Roch theorem to $C$, we get that
\[h^0(D)=\deg(D)+2g_C-2=4\]
and $h^0(S,\omega_S^{-1})=5$. Since the sections $T,UT,U^2T,X$ and $Y$
are linearly independent, and extend to a section
of $\mathcal O_{\RProj}(1)$, we get
a basis of $\Gamma(S,\omega_S^{-1})$.
\end{proof}
\begin{lemma}
\label{lemma.linearomega}%
The linear system $\vert\omega_S^{-1}\vert$ has no base point and the
basis given in lemma~\ref{lemma.sectionsomega} gives a morphism
from $S$ to $\Proj^4_\QQ$, the image of which is the surface $S'$
given by the system of equations
\begin{equation*}
\begin{cases}
X_0X_2-X_1^2=0\\
X_3^2+X_4^2=(aX_0+bX_1+cX_2)(a'X_0+b'X_1+c'X_2)
\end{cases}
\end{equation*}
where
\begin{align*}
a&=a_1a_2,&b&=a_1b_2+a_2b_1,&c&=b_1b_2,\\
a'&=a_3a_4,&b'&=a_3b_4+a_4b_3,&c'&=b_3b_4.
\end{align*}
The induced map $\psi:S\to S'$ is the blowing up
of the conjugate singular points of $S'$ given by
$P^\xi=(0:0:0:1:-\xi)$ with $\xi^2=-1$ and $\psi^{-1}(P^\xi)=E^{\xi}$.
\end{lemma}
\begin{proof}
This follows from the fact that the map from $S$ to $\Proj^4_\QQ$
induces the maps
\[((x:y:t),u)\longmapsto (t:ut:u^2t:x:y)\]
from $S_1$ to $\Proj^4_\QQ$ and
\[((x':y':t'),v)\longmapsto (v^2t':vt':t':x':y')\]
from $S_2$ to $\Proj^4_\QQ$.
\end{proof}
\begin{rema}
The surface $S'$ is an Iskovskikh surface
\cite{coraytsfasman:delpezzo}; 
it is a singular Del Pezzo surface of degree $4$ with a singularity
of type $2A_1$ and $\psi:S\to S'$ is a minimal resolution
of singularities for $S'$.
\end{rema}
We finish this section by a brief reminder of the description
of the Brauer group of $S$.
\begin{lemma}
\label{lem:brauer}%
The cokernel of the morphism from the Brauer group of
$\QQ$ to the Brauer group of $S$ is isomorphic to the Klein group
$(\ZZ/2\ZZ)^2$ and the image of the natural injective map
\[\Br(S)/\Br(\QQ)\longrightarrow \Br(\QQ(S))/\Br(\QQ)\]
is generated by the elements $(-1,L_j(U,V)/L_k(U,V))$
for $j,k\in\onefour$.
\end{lemma}
\begin{proof}
By \cite[lemma 6.3]{sansuc:brauer} and the fact that $\Pic(S)$
coincides with $\Pic(S_{\QQ(\ci)})^\Ga$, there is an exact sequence
\[0\longrightarrow \Br(\QQ)\longrightarrow\ker(\Br(S)\longrightarrow
\Br(\overline S))\longrightarrow H^1(\Gal(\overline\QQ/\QQ),
\Pic(\overline S))\longrightarrow 0.\]
Since $\overline S$ is rational and the Brauer group is a
birational invariant of smooth projective varieties, we get
that the cokernel of the morphism $\Br(\QQ)\to\Br(S)$ is isomorphic
to the cohomology group
$H^1(\Gal(\overline \QQ/\QQ),\Pic(\overline S))$. But
the group $H^1(\Gal(\overline \QQ/\QQ(\ci)),\Pic(\overline S))$
is trivial and we are reduced to computing the group
$H^1(\Ga,\Pic(S_{\QQ(\ci)}))$.
Since $\Ga$ is cyclic of order $2$, this cohomology
group coincides with the homology of the complex
{\CDat
\[\Pic(S_{\QQ(\ci)})@>\Id-\sigma>>
\Pic(S_{\QQ(\ci)})@>\Id+\sigma>>\Pic(S_{\QQ(\ci)})\]}%
where $\sigma$ denotes the complex conjugation.
By the description of the action of $\sigma$, the $\ZZ$-module
$\ker(\Id+\sigma)$ has a basis given by
\[([D_1^+]-[D_2^+],[D_2^+]-[D_3^+],[D_3^+]-[D_4^+],
[D_1^+]-[D_1^-]).\]
On the other hand, $\im(\Id-\sigma)$ is generated by
\begin{gather*}
[D_1^+]-[D_1^-],\qquad 2[D_2^+]-[D_1^+]-[D_1^-],\\
2[D_3^+]-[D_1^+]-[D_1^-]\quad
\text{and}\quad
2[D_4^+]-[D_1^+]-[D_1^-].
\end{gather*}
Thus the quotient is isomorphic to $(\ZZ/2\ZZ)^2$ and generated
by the classes of elements of the form $[D_j^+]-[D_k^+]$
with $j,k\in\onefour$.

It remains to describe the images of the classes in the Brauer
group of the function field $\QQ(S)$. But the isomorphism
\[H^1(\Gal(\overline\QQ/\QQ),\Pic(\overline S))\longrightarrow
\Br(S)/\Br(\QQ)\]
may be described as follows:
let us consider the exact sequence of $\Gal(\overline \QQ/\QQ)$-modules:
{\CDat
\[0\longrightarrow\overline\QQ^*\longrightarrow\overline \QQ(S)^*@>\divi>>
\Div(\overline S)\longrightarrow\Pic(\overline S)\longrightarrow 0\]}%
which yields two short exact sequences:
\[0\longrightarrow\overline\QQ^*\longrightarrow\overline \QQ(S)^*\longrightarrow
\overline\QQ(S)^*/\overline\QQ^*\longrightarrow 0\]
and
\[0\longrightarrow\overline \QQ(S)^*/\overline\QQ^*\longrightarrow
\Div(\overline S)\longrightarrow\Pic(\overline S)\longrightarrow 0.\]
Taking the corresponding cohomology long exact sequences
we get exact sequences
\[0\longrightarrow H^1(\Gal(\overline \QQ/\QQ),\Pic(\overline S))
\buildrel\partial\over\longrightarrow H^2(\Gal(\overline\QQ/\QQ),
\overline\QQ(S)^*/\overline\QQ^*)\]
and
\[0\longrightarrow\Br(\QQ)\longrightarrow \Br(\QQ(S))
\longrightarrow H^2(\Gal(\overline\QQ/\QQ),\overline\QQ(S)^*/\overline\QQ^*)
\longrightarrow 0\]
and using the natural injection $\Br(S)\to\Br(\QQ(S))$
we get an isomorphism from the image of $\partial$
to $\coker(\Br(\QQ)\to\Br(S))$. But if $D$ is a divisor
on $S$ such that its class $[D]$ belongs to $\ker(1+\sigma)$ and represents
$\alpha\in H^1(\Gal(\overline \QQ/\QQ),\Pic(\overline S))$
then
\[(1+\sigma)D\in\ker(\Div(\overline S)\to\Pic(\overline S))
\cap\Div(S).\]
Therefore $(1+\sigma)D=\divi(f)$ for a function $f$ in $\QQ(S)^*$
and $\partial(\alpha)$ coincides with the image of $(-1,f)$.
In our particular case, we get that
\[(1+\sigma)(D_j^+-D_k^+)=D_j^++D_j^--D_k^+-D_k^-
=\divi(L_j(U,V)/L_k(U,V))\]
which concludes the proof.
\end{proof}
\section{Points of bounded height}
\label{section.result}

Over $\overline\QQ$ or even $\QQ(\ci)$, the only geometrical
invariant of $S$ is the cross-ratio
\[\alpha=\frac{\left\vert
\begin{matrix}
a_3&a_1\\
b_3&b_1
\end{matrix}
\right\vert\left/\left\vert
\begin{matrix}
a_3&a_2\\
b_3&b_2
\end{matrix}
\right\vert\right.}{\left\vert
\begin{matrix}
a_4&a_1\\
b_4&b_1
\end{matrix}
\right\vert\left/\left\vert
\begin{matrix}
a_4&a_2\\
b_4&b_2
\end{matrix}
\right\vert\right.}\quad\in\QQ.\]
Indeed the automorphisms of $\Proj^1_\QQ$
sending the points $P_1$, $P_2$, $P_3$
onto $\infty=(0:1)$, $0=(1:0)$ and $1=(1:1)$
lifts to an isomorphism from $S$ to the Ch\^atelet
surface with an equation of the form
\[X^2+Y^2=\beta U(U-1)(U-\alpha)T^2\]
where $\beta\in\QQ$. Over $\QQ(\ci)$ we may further reduce
to the case where $\beta=1$.
In particular, without any loss of generality, we may assume
that
\begin{equation}\label{equ.pthreefour}
a_1=b_2=1\qquad\text{and}\qquad a_2=b_1=0.
\end{equation}

\begin{hypo}
From now on we assume the relations \eqref{equ.pthreefour},
that we have $\gcd(a_3,b_3)=\gcd(a_4,b_4)=1$, and that
$a_3b_3a_4b_4(a_3b_4-a_4b_3)\neq 0$.
\end{hypo}
\begin{nota}
\label{nota.height}%
Let $C=\sqrt{\prod_{j=1}^4(\vert a_j\vert+\vert b_j\vert)}$.
We equip the projective space $\Proj^4_\QQ$ with the exponential height
$H_4:\Proj^4(\QQ)\to\RR$
defined by
\begin{equation*}
H_4(x_0:x_1:x_2:x_3:x_4)=
\max\left(\vert x_0\vert,\vert x_1\vert,\vert x_2\vert,
\frac{\vert x_3\vert}C,\frac{\vert x_4\vert}C\right)
\end{equation*}
if $x_0,\dots,x_4$ are coprime integers. Using the morphism
$\psi:S\to S'$, we get a height $H=H_4\circ\psi$
which is associated to the anticanonical line bundle $\omega_S^{-1}$.
\par
We denote by $\Val(\QQ)$ the set of places of $\QQ$.
For any $v\in\Val(\QQ)$, $\QQ_v$ is the corresponding
completion of $\QQ$.
As explained in \cite[\S2]{peyre:fano}, such a height
enables us to define a Tamagawa measure $\meas_H$ on the adelic space
$S(\Adeles_\QQ)=\prod_{v\in\Val(\QQ)}S(\QQ_v)$.
We also consider the constant $\alpha(S)$ defined
in \cite[definition 2.4]{peyre:fano} which is equal to $1$
in our particular case and, following Batyrev and 
Tschinkel~\cite{batyrevtschinkel:toric}, 
we also put
\[\beta(S)=\sharp\bigl(\coker(\Br(\QQ)\to\Br(S))\bigr)=4,\]
by lemma \ref{lem:brauer}. We then set
\[C_H(S)=\alpha(S)\beta(S)\meas_H(S(\Adeles_\QQ)^{\Br})\]
where $S(\Adeles_\QQ)^{\Br}$ is 
the set of points in the adelic space for which the Brauer-Manin obstruction
to weak approximation is trivial.
\par
We are interested in the asymptotic behaviour
of the number of points of bounded height in $S(\QQ)$,
that is by the number
\[N_{S,H}(B)=\sharp\{\,P\in S(\QQ)\mset H(P)\leq B\,\}\]
for $B\in\RR$ with $B>1$.
\end{nota}
\begin{figure}[ht]
\drawing{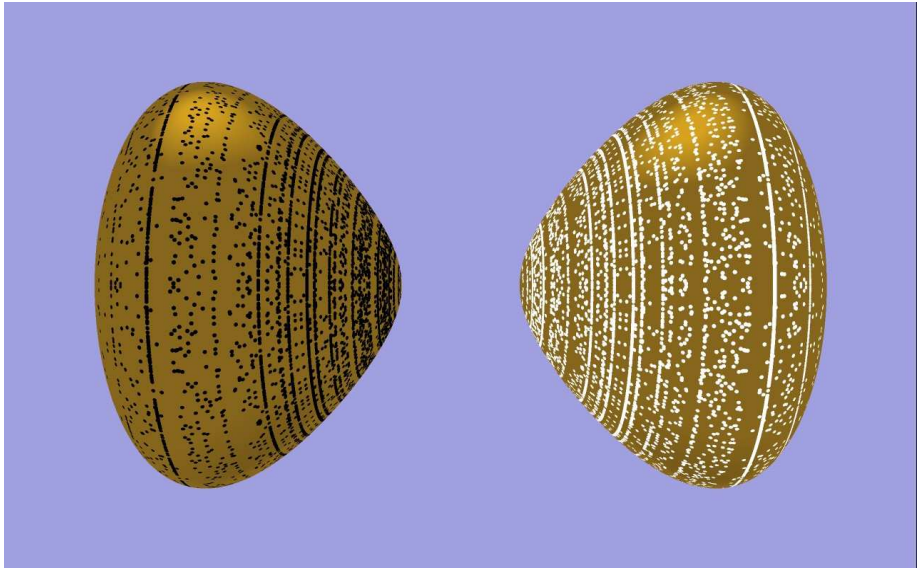}
\caption{Obstruction to weak approximation}\label{fig.zolidessin}
\end{figure}
As an illustration of our problem we have drawn in 
figure~\ref{fig.zolidessin}
the set of points
\[\{\,P\in S(\QQ)\mset H(P)\leq 2000\,\}\]
for the surface $S$ obtained with $a_2=b_1=0$,
$a_1=b_2=a_3=b_3=a_4=1$,
and $b_4=-1$.
The colour of a rational point $P=((y:z:t),u)$ is
black if $u/2^{v_2(u)}\equiv 1\mod 4$, white otherwise.
The fact that all black points are on one of the real
connected components of $S(\RR)$ may be explained by
the Brauer-Manin obstruction to weak approximation.
\par
We can now state the main result of this paper.
\begin{theo}\label{theo.main}
For any Ch\^atelet surface as above,
we have the asymptotic formula
\begin{equation}\tag{F}
N_{S,H}(B)= C_H(S)B\log(B) +O\big(B \log (B)^{0.972}\big).
\end{equation}
\end{theo}
\begin{rems}
(i) One may note that, as $S(\QQ)$ is dense in $S(\Adeles_\QQ)^{\Br}$
by \cite[theorem B]{ctssd:chatelet1},
this formula is compatible
with the empirical formula (F) described in
\cite[formule empirique 5.1]{peyre:lille}
which is a refinement of a conjecture of Batyrev and 
Manin~\cite{batyrevmanin:hauteur}.

(ii) Over $\RR$, the image of $S(\RR)$ on $\Proj^1(\RR)$ is
the union of two intervals defined by the conditions
$\prod_{j=1}^4L_j(U,V)>0$.
Therefore we may choose $j,k\in\onefour$ such that $j\neq k$
and the sign of $L_j(U,V)L_k(U,V)$ is not constant on $S(\RR)$.
The evaluation of the corresponding element $(-1,L_j(U,V)/L_k(U,V))\in\Br(S)$
(see lemma~\ref{lem:brauer}) is not constant on $S(\RR)$. Therefore
in all the cases we consider,
\[S(\Adeles_\QQ)^{\Br}\neq S(\Adeles_\QQ).\]
\end{rems}
\section{Description of versal torsors}
\label{section.torsors}

Versal torsors were first
introduced by J.-L. Colliot-Th\'el\`ene and J.-J. Sansuc
in \cite{cts:predescente2}, \cite{cts:descente1}
and \cite{cts:descente2} as a tool to prove that the Brauer--Manin
obstruction to the Hasse principle and weak approximation
is the only one. In their setting, it is sufficient
to construct a variety which is birational
over the ground field to the versal torsors.
Such a construction for Ch\^atelet surfaces has
been carried out in \cite[\S7]{ctssd:chatelet2}.

Our purpose, however, is slightly different:
we want to parametrise the points of $S(\QQ)$ using versal
torsors. Therefore we shall make the description of
\cite[\S7]{ctssd:chatelet2} slightly more precise
in the particular case we are considering and construct the
versal torsors with rational points as constructible subsets
of an affine space of dimension ten. Our construction is also
akin to the constructions based upon Cox rings.

We shall first introduce an intermediate versal torsor
which corresponds to the Picard group of $S$ over $\QQ$,
that is to the maximal split quotient of $\TNS$.
This intermediate torsor is easy to describe and shall
be useful in the parametrisation of the rational points.

\begin{defi}
Let $\torsiZ$ be the subscheme of $\Aff^5_\ZZ=\Spec(\ZZ[X,Y,T,U,V])$ 
defined by the equation
\begin{equation}
\label{equ.torsiz}%
X^2+Y^2=T^2 \prod_{j=1}^4L_j(U,V)
\end{equation}
and the conditions
\[(X,Y,T)\neq 0\quad\text{and}\quad(U,V)\neq 0.\]
The split algebraic torus $\Tspl=\GmZ^2$ acts on $\torsiZ$
via the morphism of tori
\[(\lambda,\mu)\mapsto(\lambda,\lambda,\mu^{-2}\lambda,\mu,\mu)\]
from $\GmZ^2$ to $\GmZ^5$ and the natural action of $\GmZ^5$
on $\Aff^5_\ZZ$. Let $\torsiQ$ be the variety $\torsZ_{\inter,\QQ}$.
We have an obvious morphism $\pi_\inter$ from $\torsiQ$ to
$S$ which may be described as follows: for any extension $\KK$ of $\QQ$
and any point $(x,y,t,u,v)$ of $\torsiQ(\KK)$,
if $v\neq 0$, then the point ${((x:y:tv^2),u/v)}$ belongs to ${S_1(\KK)\subset
S(\KK)}$. If $u\neq 0$ then the point
${((x:y:tu^2),v/u)}$ belongs to ${S_2(\KK)\subset S(\KK)}$
and the points obtained in $S(\KK)$ coincide if $uv\neq0$.
The morphism $\pi_\inter$ makes of $\torsiQ$ a $\Gm^2$-torsor
over $S$. 
\end{defi}

We now turn to the construction of the versal torsors.
\begin{nota}
We denote by $\DDelta$ the set of exceptional divisors
in $S_{\QQ(\ci)}$ and consider it as a $\Ga$-set. We then
consider the affine space $\Aff_\DDelta$ of dimension $10$ over $\QQ$ 
defined by
\[\Aff_\DDelta=\Spec\left((\QQ(\ci)[Z_\delta,\delta\in\DDelta])^\Ga\right)\]
where ${Z_\delta,\delta\in\Delta}$ are ten variables.
We also consider the algebraic torus
\[T_\Delta=\Spec\left((\QQ(\ci)[Z_\delta,Z^{-1}_\delta,
\delta\in\DDelta])^\Ga\right).\]
We shall also write $Z_k^\varepsilon$ (resp. $Z_0^\varepsilon$)
for $Z_{D_k^\varepsilon}$ (resp. $Z_{E^\varepsilon}$).
Let $\DDelta_\QQ$ be the set of $\Ga$-orbits in $\DDelta$.
We put $E=\{E^+,E^-\}$ and $D_j=\{D_j^+,D_j^-\}$ for $j\in\onefour$.
Then $\DDelta_\QQ=\{E,D_1,D_2,D_3,D_4\}$. 
For $\delta\in\DDelta_\QQ$, we may also write $\delta=\{\delta^+,\delta^-\}$
and we put
\[X_\delta=\frac 12(Z_{\delta^+}+Z_{\delta^-})\qquad\text{and}
\qquad Y_\delta=\frac 1{2\ci}(Z_{\delta^+}-Z_{\delta^-}).\]
Then
\[(\QQ(\ci)[Z_\delta,\delta\in\DDelta])^\Ga=
\QQ[X_\delta,Y_\delta,\delta\in\DDelta_\QQ].\]
\end{nota}

We now wish to construct for each isomorphism class of versal torsor
over $S$ with a rational point a representative
of this class in $\Aff_\DDelta$.
It follows from \cite[proposition~2]{cts:descente1} that the set
of isomorphism classes of such torsors is finite.
We first introduce a finite set which will be used to
parametrise this set of torsors.
\begin{nota}
Let $\badp$ be the set of primes $p$ such that $p\mid\prod_{1\leq j<k\leq 4}
\Delta_{j,k}$\footnote{Over $\ZZ/2\ZZ$,
one of the $\Delta_{j,k}$ has to be zero,
and so $2 \in\badp$.}. For any $j$ in $\onefour$, we put
\[\badp_j=\{\,p\in\badp\mset p\equiv 3\mod 4\quad
\text{and}\quad p\mid\prod_{k\neq j}
\Delta_{j,k}\,\}\]
and
\[\Sigma_j=\biggl\{(-1)^{\varepsilon_{-1}}\prod_{p\in\badp_j}p^{\varepsilon_p},
(\varepsilon_{-1},(\varepsilon_p)_{p\in\badp_j})
\in\{0,1\}\times\{0,1\}^{\badp_j}\biggr\}.\]
Finally, we define $\Sigma$ to be the set of $\mm=(m_j)_{1\leq j\leq 4}
\in \prod_{j=1}^4\Sigma_j$ such that the four integers are
relatively prime,
$m_1$ is positive and $\prod_{j=1}^4m_j$
is a square. For any $\mm\in\Sigma$, we denote by $\alpha_\mm$
the positive square root of $\prod_{j=1}^4m_j$.

Let $\mm$ belong to $\Sigma$. We denote by $\torsQ_\mm$
the constructible subset of $\Aff_\DDelta$ defined by the equations
\begin{equation}\label{equ.tors.closed}
\Delta_{j,k}m_lZ_l^+Z_l^-+\Delta_{k,l}m_jZ_j^+Z_j^-
+\Delta_{l,j}m_kZ_k^+Z_k^-=0
\end{equation}
if $1\leq j<k<l\leq 4$ and the inequalities
\begin{equation}\label{equ.tors.open}
(Z_{\delta_1},Z_{\delta_2})\neq(0,0)
\end{equation}
whenever $\delta_1\cap\delta_2=\emptyset$.
Note that these conditions are invariant under the action
of the Galois group $\Ga$. Thus $\torsQ_\mm$ is defined
over $\QQ$.

We then define a morphism $\pi_\mm:\torsQ_\mm\to S$.
In order to do this, it is enough to define a morphism
$\widehat \pi_\mm:\torsQ_\mm\to\torsiQ$ which is done
as follows: for any extension $\KK$ of $\QQ$ and any 
$\zz=(z_\delta)_{\delta\in\Delta}$ in $\torsQ_\mm(\KK)$, the
conditions~\eqref{equ.tors.closed} and \eqref{equ.tors.open}
ensure that there exists a pair $(u,v)\in \KK^2\setminus\{0\}$
such that
\begin{equation}\label{equ.tors.uv}
L_j(u,v)=m_jz_j^+z_j^-
\end{equation}
for $j\in\onefour$. Let $(x,y,t)\in \KK^3\setminus\{0\}$ be given
by the conditions
\begin{equation}\label{equ.tors.xyz}
\begin{cases}
x+\ci y=\alpha_\mm (z_0^+)^2\prod_{j=1}^4z_j^+,\\
x-\ci y=\alpha_\mm (z_0^-)^2\prod_{j=1}^4z_j^-,\\
t=z_0^+z_0^-.
\end{cases}
\end{equation}
Then we have the relation
\[x^2+y^2=t^2\prod_{j=1}^4L_j(u,v).\]
and $(x,y,t,u,v)$ belongs to $\torsiQ(\KK)$.
\par
It remains to describe the action of the torus
$\TNS$ associated to the $\Ga$-lattice $\Pic(\overline S)$
on $\torsQ_\mm$. The algebraic torus
$T_\DDelta$ corresponds to the $\Ga$-lattice $\ZZ^\DDelta$ and $T_\DDelta$
acts by multiplication of the coordinates on $\Aff_\DDelta$.
The natural surjective morphism of $\Ga$-lattices
\[-\pr:\ZZ^\DDelta\longrightarrow \Pic(\overline S)\]
induces an embedding of the algebraic torus $\TNS$ on
$T_\DDelta$.\footnote{There is some question of convention
in the definition of versal torsors which leads us to use
the opposite of the projection map.}
\par
The description of the kernel of the morphism $\pr$ (see~\eqref{equ.pic.fibre}
and~\eqref{equ.pic.can}) give the following
equations for $\TNS$:
\begin{equation}\label{equ.torus.fibre}
Z_j^+Z_j^-=Z_k^+Z_k^-
\end{equation}
for $j$, $k\in\onefour$ and
\begin{equation}\label{equ.torus.can}
Z_0^+Z_j^+Z_k^+=Z_0^-Z_l^-Z_m^-
\end{equation}
if $\{j,k,l,m\}=\onefour$. The equations~\eqref{equ.tors.closed}
are invariant under the action of $\TNS$ thanks to~\eqref{equ.torus.fibre}
as are the inequalities~\eqref{equ.tors.open}. Therefore
the action of $\TNS$ on $\Aff_\DDelta$ induces a natural action
of $\TNS$ on $\torsQ_\mm$. This description of $\TNS$ also
implies that $\pi_\mm$ is invariant under the action
of $\TNS$ on $\torsQ_\mm$. Indeed let $\KK$ be an extension of $\QQ$,
let $t$ belong to $\TNS(\KK)$ and $\zz$ to $\torsQ_\mm(\KK)$.
We put $\zz'=t\zz$. It follows from \eqref{equ.tors.uv}
and \eqref{equ.torus.fibre} that $\zz$ and $\zz'$ define
the same point $(u:v)\in\Proj^1(\KK)$ and from \eqref{equ.tors.xyz},
\eqref{equ.torus.fibre} and \eqref{equ.torus.can} that $\zz$
and $\zz'$ give the same point $(x:y:tv^2)$ (resp.
$(x:y:tu^2)$ in $\Proj^2(\KK)$).
\end{nota}
\begin{prop}
\label{prop.torsor}%
For any $\mm\in\Sigma$,
the variety $\torsQ_\mm$ equipped with the map
$\pi_\mm:\torsQ_\mm\to S$ and the above action of $\TNS$
is a versal torsor above $S$.
\end{prop}
\begin{proof}
First of all, we may note that for any extension $K$ of
$\QQ$, if $R\in\torsQ_\mm(K)$ then $\pi_\mm^{-1}(\pi_\mm(R))$
coincides with the orbit of $R$ under the action of $\TNS$.
Indeed if $R'\in\torsQ_\mm(K)$ satisfies $\pi_\mm(R')=\pi_\mm(R)$,
then there exists a unique $\zz\in T_\DDelta(\KK)$ such that $R'=\zz R$.
Let us write $\zz=(z_\delta)_{\delta\in\DDelta}$.
Using~\eqref{equ.tors.uv} and~\eqref{equ.tors.xyz}
and the description of the action of $\Gm(K)$ on $\torsiZ$,
we get that $z_i^+z_i^-=z_j^+z_j^-$ if $1\leq i<j\leq 4$ and
\[z_0^+z_0^-(z_k^+z_k^-)^2=(z_0^+)^2\prod_{j=1}^4z_j^+=(z_0^-)^2
\prod_{j=1}^4z_j^-.\]
for $k\in\onefour$. We deduce from these equations that $\zz\in\TNS(K)$.
\par
It is enough to prove the result over $\overline\QQ$.
By choosing square roots $\alpha_j$ of $m_j$ such that
$\prod_{j=1}^4\alpha_j=\alpha_\mm$, and using a change of
variable of the form ${Z_j^\varepsilon}'=\alpha_jZ_j^\varepsilon$
for $\varepsilon\in\{+1,-1\}$ and $j\in\onefour$ we may
assume that $\mm=(1,1,1,1)$.
Note that for any $\delta$ in $\DDelta$, the variety
$\pi_\mm^{-1}(E_\DDelta)$ is the subvariety of $\torsQ_\mm$
defined by $Z_\delta=0$. If $\varepsilon\in\{+1,-1\}$,
we consider the open subset
\[U_\varepsilon=S-E^{\varepsilon}-\bigcup_{j=1}^4E_j^\varepsilon\]
of $S$ and for $j\in\onefour$, we put
\[U_j=S-E^+-E^--\bigcup_{k\neq j}(E_k^+\cup E_k^-).\]
The open subsets $U_1,U_2,U_3,U_4,U_+$ and $U_-$ form
an open covering of $S$. If $\varepsilon \in\{+1,-1\}$,
we may consider that $X+\varepsilon \ci Y=1$ on $U_\varepsilon$
and we define a section $s^1_\varepsilon$ (resp. $s^2_\varepsilon$)
of $\pi_{\boldsymbol 1}$ over $U_\varepsilon\cap S_1$ (resp. $U_\varepsilon\cap S_2$)
by $Z_0^\varepsilon=Z_1^\varepsilon=Z_2^\varepsilon=Z_3^\varepsilon=Z_4^\varepsilon=1$,
$Z_0^{-\varepsilon}=t$ and $Z_j^{-\varepsilon}=L_j(U,1)$
(resp. $Z_j^{-\varepsilon}=L_j(1,V)$) for $j\in\onefour$.
Similarly, for $j\in\onefour$, fix $k,l,m$ so that $\{j,k,l,m\}=
\onefour$. On $U_j$, we may consider that $L_k(U,V)=1$
and $T=1$. We may then define a section $s_j$ of $\pi_{\boldsymbol 1}$
over $U_j$ by
$Z_k^+=Z_k^-=Z_0^+=Z_0^-=Z_l^+=Z_m^+=1$
and
\[Z_l^-=L_l(U,V),\quad Z_m^-=L_m(U,V),
\quad Z_j^+=\frac{X+\ci Y}{\prod_{r\neq j}Z_r^+}\quad
\text{and}\quad Z_j^-=\frac{X-\ci Y}{\prod_{r\neq j}Z_r^+}.\]
The conditions~\eqref{equ.tors.open} ensures that, for any point
$P\in\torsQ_{\boldsymbol 1}(\overline\QQ)$, 
the stabilizer of $P$ in $\TNS(\overline\QQ)$
is trivial. Using the action of $\TNS$ on $\torsQ_{\boldsymbol 1}$
we then get an equivariant isomorphism from $\TNS\times U$ to
$\pi_{\boldsymbol 1}^{-1}(U)$ for each open subset $U$ described above.
This proves that $\torsQ_\mm$ is a $\TNS$-torsor over~$S$.
\par
It remains to prove that the endomorphism of $\Pic(\overline S)$
defined by this torsor is the identity map. Let us first recall
how this endomorphism may be defined. If $L$ is a line bundle 
over~$\overline S$, then the class of~$L$ defines a morphism
of Galois lattices $\ZZ\to\Pic(\overline S)$ and therefore
a morphism of algebraic tori $\phi_L:\TNS\to\Gm$ and an action
of $\TNS$ on $\Gm$. The restricted product $\torsQ\times^{\TNS}\Gm$
is a $\Gm$-torsor over $\overline S$ which defines an element 
of $\Pic(\overline S)$.
For any $\delta$ in $\DDelta$, the function $Z_\delta$ on $\torsQ_\mm$
is invariant under the action of the kernel of the map 
${\phi_\delta:\TNS\to\Gm}$ defined by the class of $\delta$
in $\Pic(\overline S)$. Therefore this function defines
an antiequivariant map from $\torsQ_\mm\times^{\TNS}\Gm$ to $\Aff^1$ which
vanishes with multiplicity one over $\pi_\mm^{-1}(\delta)$.
Thus the endomorphism defined by $\torsQ_\mm$ on $\Pic(\overline S)$
sends the class of $\delta$ to itself for any $\delta\in\DDelta$.
This proves that $\torsQ_\mm$ is a versal torsor over $S$. 
\end{proof}

To conclude these constructions it remains to prove
that the set of rational points $S(\QQ)$ is the disjoint union
of the sets $\pi_\mm(\torsQ_\mm(\QQ))$ where $\mm$ runs over  the
set $\Sigma$.

\begin{lemma}
\label{lemma.lift.split}
For any $P\in S(\QQ)$, we have
\[\card(\pi_\inter^{-1}(P)\cap\torsiZ(\ZZ))=\card\Gm^2(\QQ)_{\tors}=2^2.\]
\end{lemma}
\begin{proof}
Let us start with a point
$P=((x_0:y_0:t_0),u_0)$ in $S_1(\QQ)$. We then have the relation
\[x_0^2+y_0^2=t_0^2\prod_{j=1}^4L_i(u_0,1)\]
We may write $u_0=u/v$ with $u,v\in\ZZ$ and $\gcd(u,v)=1$. Then
we may find an element $\lambda$ of $\QQ$ such that the rational numbers
$x=\lambda x_0$, $y=\lambda y_0$ and $t=\lambda t_0/v^2$ are coprime
integers and we have
\[x^2+y^2=t^2\prod_{j=1}^4L_j(u,v).\]
The same construction works for any point of $S_2(\QQ)$
and if $P$ belongs to $S_1(\QQ)\cap S_2(\QQ)$ the elements
of $\ZZ^5$ thus obtained coincide up to multiplication
of the first three or the last two coordinates by $-1$.
\end{proof}
\begin{rema}
Note that if we impose conditions like
\[t>0,\quad L_1(u,v)\geq 0\quad\text{and}\quad\prod_{j=2}^4L_j(u,v)\geq 0,\]
the lifting of $P$ is unique.
\end{rema}
\begin{prop}
\label{prop.uniquem}%
Let $P$ belong to $S(\QQ)$.
Then there exists a unique $\mm$ in $\Sigma$ such that
$P$ belongs to $\pi_\mm(\torsQ_\mm(\QQ))$.
\end{prop}
\begin{proof}
Let $Q=(x,y,t,u,v)\in\torsiZ(\ZZ)$ be such that $\pi_\inter(Q)=P$. 
Without loss of generality we may assume that $Q=(x,y,t,u,v)\in\ZZ^5$ 
is such that
\begin{equation}\label{equ.lifting.first}
\begin{cases}
x^2+y^2=t^2\prod_{j=1}^4L_j(u,v),\\
\gcd(x,y,t)=1,\ \gcd(u,v)=1,\\
t>0,\ L_1(u,v)\geq 0,\ \text{and}\ \prod_{j=2}^4L_j(u,v)\geq 0.\\
\end{cases}
\end{equation}
The fact that $t^2\prod_{j=1}^4L_j(u,v)$
is the sum of two squares implies that
\begin{equation}\label{equ.lifting.pos}
\prod_{j=1}^4L_j(u,v)\geq 0
\end{equation}
and, if $\prod_{j=1}^4L_j(u,v)\neq 0$,
for any prime $p$ congruent to $3$ modulo $4$
\begin{equation}\label{equ.lifting.cong}
\sum_{j=1}^4v_p(L_j(u,v))\equiv 0\ \text{mod}\ 2.
\end{equation}
Let $j$ belong to $\onefour$. 
If $L_j(u,v)\neq0 $, we denote by $\epsilon_j\in\{-1,+1\}$
the sign of $L_j(u,v)$ 
and by $\Sigma_j(Q)$ the set of prime numbers $p$
which are congruent to $3$ modulo $4$ and such that
$v_p(L_j(u,v))$ is odd. We then put
\[m_j=\epsilon_j\times\prod_{p\in\Sigma_j(Q)}p.\]
If $L_j(u,v)=0$ we define $m_j$ as the only integer in $\Sigma_j$
such that $\prod_{k=1}^4m_k$ is a square. By construction,
we have $m_j\mid L_j(u,v)$
and the quotient $L_j(u,v)/m_j$ is the sum of two squares.
\par
Let us now check that $\mm=(m_1,m_2,m_3,m_4)$ belongs to $\Sigma$.
According to~\eqref{equ.lifting.cong}, if a prime number belongs to
$\Sigma_j(Q)$ for some $j\in\onefour$, then there exists 
$k\in\onefour$ with $k\neq j$
such that $p\in\Sigma_k(Q)$.
In particular, $p$ divides both $L_j(u,v)$ and $L_k(u,v)$
as well as
\[\Delta_{j,k}u=b_kL_j(u,v)-b_jL_k(u,v)\]
and $\Delta_{j,k}v$. Since $\gcd(u,v)=1$, we get that $p\mid\Delta_{j,k}$.
This proves that $\mm\in\prod_{j=1}^4\Sigma_j$. But 
combining~\eqref{equ.lifting.pos}, \eqref{equ.lifting.cong} and
the definition of $\mm$ we get that $\prod_{j=1}^4m_j$ is a square.
If $d$ divides all the $m_j$, it divides $\gcd_{1\leq j<k\leq 4}(\Delta_{j,k})$
which is equal to $1$ since $\Delta_{1,2}=1$ under the
condition~\eqref{equ.pthreefour}. Finally $m_1>0$ since
$L_1(u,v)>0$ or $\prod_{j=2}^4L_j(u,v)>0$. Thus, $\mm$ belongs to $\Sigma$.
\par
We now wish to prove that $Q$ belongs to $\hat\pi_\mm(\torsQ_\mm(\QQ))$.
By construction of $\mm$, for any $j$ in $\onefour$, the integer
$L_j(u,v)/m_j$ is the sum of two squares. Moreover if $p$
is a prime number, congruent to $3$ modulo $4$, then $p$
generates a prime ideal of $\ZZ[\ci]$. From the 
relations~\eqref{equ.lifting.first}, if $p\mid t$, then
$p\mid (x+\ci y)(x-\ci y)$. In that case we have $p\mid x$ and
$p\mid y$, which contradicts the fact that $\gcd(x,y,t)=1$.
As $t>0$, we get that $t$ may also be written as the sum
of two squares.
\par
If $\prod_{j=1}^4L_j(u,v)\neq 0$, we choose for $j\in\{1,2,3\}$
an element $z^+_j\in\ZZ[\ci]$ such that $L_j(u,v)/m_j=z_j^+\overline{z_j^+}$
and an element $z_0^+\in\ZZ[\ci]$ such that $t=z_0^+\overline{z_0^+}$.
Then we get the relation
\[L_4(u,v)/m_4=\left(\frac{x+\ci y}{\alpha_\mm (z_0^+)^2\prod_{j=1}^3z_j^+}
\right)\overline{\left(\frac{x+\ci y}{\alpha_\mm (z_0^+)^2\prod_{j=1}^3z_j^+}
\right)}\]
and we put $z_4^+=(x+\ci y)/(\alpha_\mm (z_0^+)^2\prod_{j=1}^3z_j^+)\in\QQ[\ci]$.
If $\prod_{j=1}^4L_j(u,v)=0$, we choose $z_1^+,z_2^+,z_3^+,z^+$
as above and $z_4^+\in\ZZ[\ci]$ such that $L_4(u,v)/m_4=z_4^+\overline{z_4^+}$.
In both cases, we put $z_j^-=\overline{z_j^+}$ for $j\in\onefour$
and $z_0^-=\overline{z_0^+}$.
\par
The family so constructed satisfy the 
relations~\eqref{equ.tors.xyz} and~\eqref{equ.lifting.first},
from which it follows that the corresponding family 
$(z_\delta)_{\delta\in\DDelta}$ is a solution to the 
systems~\eqref{equ.tors.closed} and~\eqref{equ.tors.open}.
Thus we obtain a point $R$ in $\torsQ_\mm(\QQ)$ such that
$\pi_\mm(R)=P$.
\par
Let $\mm'$ belong to $\Sigma$ and assume that the point $P$ belongs to
the set $\pi_{\mm'}(\torsQ_{\mm'}(\QQ))$ as well. Then 
by~\eqref{equ.lifting.first}, we have for any prime number $p$
\[v_p(m'_j)-v_p(m'_k)=v_p(L_j(u,v))-v_p(L_k(u,v))=v_p(m_j)-v_p(m_k)\]
for any $j,k$ in $\onefour$ such that $L_j(u,v)L_k(u,v)\neq 0$.
Similarly, denoting by $\sgn(m)$ the sign of an integer $m$,
we have
\[\sgn(m'_j)/\sgn(m'_k)=\sgn(m_j)/\sgn(m_k).\]
These relations between $\mm$ and $\mm'$ remain valid if
$L_j(u,v)L_k(u,v)=0$ since the products $\prod_{j=1}^4m_j$ 
and $\prod_{j=1}^4m_j'$ are squares.
But, by definition of $\Sigma$, we have
\[m'_1>0\quad\text{and}\quad\min_{1\leq j\leq 4}v_p(m'_j)=0\]
for any prime number $p$, and similarly for $\mm$.
We obtain that $\mm=\mm'$.
\end{proof}

\section{Jumping up}
\label{section.jumpingup}%
Having constructed the needed versal torsors explicitly, 
we now wish to lift our initial counting problem
to these torsors. In order to do this, we shall define
an adelic domain $\mathcal D_\mm$ in the adelic space
$\torsQ_\mm(\Adeles_\QQ)$ so that for any $P\in\pi_\mm(\torsQ_\mm(\QQ))$
the cardinality of $\pi_\mm^{-1}(P)\cap\mathcal D_\mm$ is
$\card\TNS(\QQ)_\tors$. 

\subsection{Idelic preliminaries}
We first need to gather a few facts about the adelic
space $\TNS(\Adeles_\QQ)$.
\begin{nota}
We consider the affine space
\[\Aff_{\DDelta,\ZZ}=
\Spec(\ZZ[X_{\delta},Y_\delta,\delta\in\DDelta_\QQ]).\]
Let~$A$ be a commutative ring. The group $\Ga$ acts
on the ring
\[\prod_{\delta\in\DDelta}A\otimes_{\ZZ}\ZZ[\ci]\]
and we may identify the $A$-points of $\Aff_\DDelta$ with
the elements of the invariant ring
\[A_\DDelta=\Bigl(\prod_{\delta\in\DDelta}A\otimes_{\ZZ}\ZZ[\ci]\Bigr)^\Ga.\]
\par
Let $\primes$ be the set of prime numbers.
Let $p\in\primes$. We put ${\Sp=\Spec(\QQ_p\otimes_\ZZ\ZZ[\ci])}$
which we may identify with the set of places of $\QQ[\ci]$ above
$p$. If $\boldsymbol a=(a_{\mathfrak p})_{\mathfrak p\in\Sp}$ 
and $\boldsymbol b=(b_{\mathfrak p})_{\mathfrak p\in\Sp}$
belong to $\ZZ^\Sp$, we write 
$\boldsymbol a\geq \boldsymbol b$ if $a_{\mathfrak p}\geq b_{\mathfrak p}$
for $\mathfrak p\in\Sp$ 
and $\min(\boldsymbol a,\boldsymbol b)
=(\min(a_{\mathfrak p},b_{\mathfrak p}))_{\mathfrak p\in\Sp}$.
The valuations induce a map
\[\widehat v_p:\QQ_p\otimes_\ZZ\ZZ[\ci]
\longrightarrow(\ZZ\cup\{+\infty\})^{\Sp}.\]
Thus we get a natural map
\[(\QQ_p\otimes_\ZZ\ZZ[\ci])^\DDelta\longrightarrow
(\ZZ\cup\{+\infty\})^{\Sp\times\DDelta}.\]
The action of $\Ga$ on $\Sp$ and $\DDelta$ induces
an action of $\Ga$ on the set on the right-hand side
so that the above map is $\Ga$ equivariant.
Denoting by $\overline\Gamma_p$ the set of invariants
in $(\ZZ\cup\{+\infty\})^{\Sp\times\DDelta}$ and by $\Gamma_p$
its intersection with $\ZZ^{\Sp\times\DDelta}$, we get
a map
\[\log_p:\Aff_\DDelta(\QQ_p)\longrightarrow\overline\Gamma_p\]
whose restriction to $T_\DDelta(\QQ_p)$ is a morphism from
this group to the group $\Gamma_p$ and $\log_p$ is compatible with
the action of $T_\DDelta(\QQ_p)$ on the left and the action
of $\Gamma_p$ on the right.
We denote by $\Xi_p$ the set of elements $(r_{\mathfrak p,\delta})$ 
of $\Gamma_p$ such that
$r_{\mathfrak p,\delta}\geq 0$ for any $\mathfrak p\in\Sp$ and
any $\delta\in\DDelta$.
\par
If $T$ is an algebraic torus over $\QQ$ which splits over $\QQ(\ci)$,
then $X^*(T)$ denotes the group of characters of $T$ over $\QQ(\ci)$
and $X_*(T)=\Hom(X^*(T),\ZZ)$ its dual, that is the group
of cocharacters of $T$. 
We denote by $\langle\cdot{,}\cdot\rangle$ the natural pairing $X^*(T)\times X_*(T)\to\ZZ$.
For any place $v$ of $\QQ$, we denote
by $X_*(T)_v$ the group of cocharacters of $T$ over $\QQ_v$,
which may be described as $X_*(T)^{\Gal(\overline\QQ_v/\QQ_v)}$.
We also consider the groups $X_*(T)_\QQ=X_*(T)^\Ga$
and $X^*(T)_\QQ=X^*(T)^\Ga$.
The group $\Gamma_p$ may then be seen as the group $X_*(T_\DDelta)_p$.
The restriction of $\log_p$ from $T_\DDelta(\QQ_p)$ to $\Gamma_p$
is then the natural morphism defined in \cite[\S2.1]{ono:tori}.
For any $(\rr_\delta)_{\delta\in\DDelta}\in\Gamma_p$, 
we put $\rr_j^{\pm}=
\rr_{D_j^{\pm}}$ for $j\in\onefour$
and $\rr_0^{\pm}=\rr_{E^\pm}$.
The group $X_*(\TNS)_p$ is then the subgroup of $\Gamma_p$
given by the equations
\begin{align*}
\rr_j^++\rr_j^-&=\rr_l^++\rr_l^-\\
\noalign{\noindent for $1\leq j<l\leq 4$ and}
\rr_0^++\rr_j^++\rr^+_l&=\rr_0^-+\rr^-_m+\rr^-_n
\end{align*}
if $\{j,l,m,n\}=\onefour$.
\end{nota}
\begin{rema}
If $p\equiv 3\mod 4$ or $p=2$ then there exists a unique
element $\mathfrak p$ in $\Sp$. Thus $\Gamma_p$ is canonically isomorphic
to $\ZZ^{\DDelta_\QQ}$. If $p\equiv 1\mod 4$, then choosing
an element $\mathfrak p\in\Sp$, we get an isomorphism
from $\ZZ^\DDelta$ to $\Gamma_p$.
\end{rema}
\begin{lemma}
\label{lemma.idelic}%
For any prime $p$ the morphism $\log_p$ induces
an isomorphism from the quotient $\TNS(\QQ_p)/\TNS(\ZZ_p)$
to $X_*(\TNS)_p$ and there is an exact sequence
\[1\longrightarrow \TNS(\QQ)_{\tors}\longrightarrow\TNS(\QQ)
\longrightarrow\bigoplus_{p\in\primes}X_*(\TNS)_p\longrightarrow 0.\]
\end{lemma}
\begin{proof}
By \cite[p.~449]{draxl:tori}, the kernel of the map $\log_p$ 
from $\TNS(\QQ_p)$ to $X_*(\TNS)_p$
coincides with $\TNS(\ZZ_p)$ for any prime $p$.
Let us prove that the map $\bigoplus_p\log_p$ from $\TNS(\QQ)$ to
$\bigoplus_pX_*(\TNS)_p$ is surjective.
We first assume that $p\neq 2$. If $p\equiv 1\mod 4$
we choose an element $\varpi\in\ZZ[i]$ such that $p=\varpi\overline\varpi$
and identify $\Sp$ with $\{\varpi,\overline\varpi\}$.
If $\rr\in\Gamma_p$, we then define 
\begin{equation*}
\exp_\varpi(\rr)=(\varpi^{r_{\varpi,\delta}}
\overline\varpi^{r_{\overline\varpi,\delta}})_{\delta\in\DDelta}.
\end{equation*}
If $p\equiv 3\mod 4$, then we put $\varpi=p$ and for $\rr\in\Gamma_p$,
we define $\exp_\varpi(\rr)$ to be $(\varpi^{r_{p,\delta}})_{\delta\in\DDelta}$.
By construction, $\exp_\varpi$ is a morphism from
$\Gamma_p$ to $T_\DDelta(\QQ)$ and satisfies
$\log_p\circ\exp_\varpi=\Id_{\Gamma_p}$ 
and $\log_\ell\circ\exp_\varpi=0$ for 
any prime $\ell\neq p$. Moreover we have
\begin{equation}
\label{equ.good.split}
\chi(\exp_\varpi(\rr))=p^{\langle\chi,\rr\rangle}
\end{equation}
for any $\chi\in X^*(T_\DDelta)_\QQ$ and any $\rr\in\Gamma_p$.
Therefore, if $\rr$ belongs to $X_*(\TNS)_p$,
then $\exp_\varpi(\rr)$ belongs to $\TNS(\QQ)$. It remains
to prove a similar result for $p=2$, although there is no
morphism which satisfies~\eqref{equ.good.split}.
Let $\rr$ belong to $X_*(\TNS)_2$.
Let us write $r_j=r_j^+=r_j^-$
for~$j$ in $\zerofour$. Since $\rr$ belong to $X_*(\TNS)_2$,
we have $r_1=r_2=r_3=r_4$.
We put $z^+_j=(1+\ci)^{r_j}$ for $j\in\{0,1,2,3\}$
and $z^+_4=(-\ci)^{r_0+2r_1}(1+\ci)^{r_0}$ and $z_j^-=\overline z_j^+$
for $j\in\zerofour$. Then $\log_2(\zz)=\rr$ and $\zz$ satisfies
equation~\eqref{equ.torus.fibre}. Moreover
if $\{j,k,l,m\}=\onefour$ one has
\[z_0^+z_j^+z_k^+/(z_0^-z_l^-z_m^-)=
\frac{(1+\ci)^{r_0+2r_1}}{(1-\ci)^{r_0+2r_1}}(-\ci)^{r_0+2r_1}=1\]
which proves that $\zz$ satisfies~\eqref{equ.torus.can}.
\par
If $\zz$ belongs to the kernel of the map $\bigoplus_p\log_p$
then its coordinates are invertible elements in $\ZZ[\ci]$.
Thus $\zz$ is a torsion element of $\TNS(\QQ)$.
\end{proof}

\subsection{Local domains}
To construct $\mathcal D_\mm$, for any prime $p$
and any $\mm\in\Sigma$
we shall define a fundamental domain in $\torsQ_\mm(\QQ_p)$
under the action of $\TNS(\QQ_p)$ modulo $\TNS(\ZZ_p)$.
In other words, we want to construct an open domain
$\mathcal D_{\mm,p}\subset\torsQ_\mm(\QQ_p)$ such that
\begin{conditions}
\item The open set $\mathcal D_{\mm,p}$ is stable
under the action of $\TNS(\ZZ_p)$;
\item For any $t$ in $\TNS(\QQ_p)\setminus\TNS(\ZZ_p)$,
one has $t.\mathcal D_{\mm,p}\cap\mathcal D_{\mm,p}=\emptyset$;
\item For any~$x$ in $\torsQ_\mm(\QQ_p)$, there exists
an element~$t$ in $\TNS(\QQ_p)$ such that~$x$ belongs to
$t.\mathcal D_{\mm,p}$.
\end{conditions}
\begin{lemma}
\label{lemma.lift.inter}%
For any prime number $p$, the domain $\torsiZ(\ZZ_p)$ is
a fundamental domain in $\torsiZ(\QQ_p)$ under the action
of $\Tspl(\QQ_p)$ modulo $\Tspl(\ZZ_p)$.
\end{lemma}
\begin{proof}
As in the proof of lemma~\ref{lemma.lift.split}, if $P$ belongs
to $S(\QQ_p)$, there exists a point $Q=(x,y,t,u,v)\in\torsiZ(\QQ_p)$
such that $\pi_\inter(Q)=P$ and
\begin{equation*}
\min(v_p(x),v_p(y),v_p(t))=\min(v_p(u),v_p(v))=0.
\end{equation*}
The last condition is equivalent to  $Q\in\torsiZ(\ZZ_p)$. The lemma then
follows from the facts that the action of $\Tspl(\QQ_p)$ 
on $\torsiZ(\QQ_p)$ is given by
\[((\lambda,\mu),(x,y,t,u,v))\mapsto
(\lambda x,\lambda y,\mu^{-2}\lambda t,\mu u,\mu v)\]
and that the $\Tspl(\QQ_p)$-orbits are the fibers of the
projection ${\pi_{\inter}:\torsiZ(\QQ_p)\to S(\QQ_p)}$. 
\end{proof}
\begin{nota}
Let $\nn=(n_1,n_2,n_3,n_4)$ belong to $(\ZZ\setminus\{0\})^4$. We then define
$\torscZ_\nn$ as the subscheme of $\Aff_{\DDelta,\ZZ}$ given by the
equations
\begin{equation}
\label{equ.torcz}
\Delta_{j,k}n_l(X^2_l+Y^2_l)+\Delta_{k,l}n_j(X_j^2+Y_j^2)
+\Delta_{l,j}n_k(X_k^2+Y_k^2)=0
\end{equation}
if $1\leq j<k<l\leq 4$.
The scheme $\torsZ_\nn$ is the open subset of $\torscZ_\nn$
given by the conditions~\eqref{equ.tors.open},
where we put $Z_{\delta^+}=X_\delta+\ci Y_\delta$ and
$Z_{\delta^-}=X_\delta-\ci Y_\delta$ for $\delta\in\DDelta_\QQ$.
\end{nota}
\begin{rems}
(i) Let $\mm$ be an element of $\Sigma$.
The scheme $\torsZ_\mm$ is a model of $\torsQ_\mm$
over $\Spec(\ZZ)$.
\par
(ii)
The variety $\torscZ_{\mm,\QQ}$ corresponds to the
restricted product of the versal torsor
by the affine toric variety associated
to the opposite of the effective cone which
has been introduced in \cite[prop.~4.2.2]{peyre:torseurs}.
\par
(iii)
We may note that an element $Q\in\torsQ_\mm(\QQ_p)$
belongs to $\torscZ_\mm(\ZZ_p)$ if and only if $\log_p(Q)$
belongs to $\Xi_p$.
\par
(iv)
The equations~\eqref{equ.torcz} define an intersection of two quadrics in 
$\Proj_\QQ^7$, 
upon which we will ultimately need to count integral points 
of bounded height.  As shown by Cook in \cite{cook:quadratic},
the Hardy--Littlewood circle method can be adapted to handle 
intersections of diagonal quadrics in at least $9$ variables 
provided that the associated singular locus is empty. 
Here we will need to deal with an intersection 
of diagonal quadrics in only $8$ variables.
For this we will call upon the alternative approach 
based on the geometry of numbers in~\cite{bretechebrowning:4linear}.
\end{rems}
\begin{lemma}
\label{lemma.orbit}%
Two elements of $\torsQ_\mm(\QQ_p)$ belong to the same orbit
under the action of $\TNS(\ZZ_p)$ if and only if they have the
same image by $\pi_\mm$ and $\log_p$.
\end{lemma}
\begin{proof}
According to proposition~\ref{prop.torsor},
two elements of $\torsQ_\mm(\QQ_p)$ belong to the same orbit
under the action of $\TNS(\QQ_p)$ if and
only if their image by $\pi_\mm$ coincide.
On the other hand, $\TNS(\ZZ_p)=\TNS(\QQ_p)\cap T_\DDelta(\ZZ_p)$
is the set of elements of $\Aff_\DDelta(\QQ_p)$ which
are sent to the origin of $\Gamma_p$ by $\log_p$.
Therefore if two elements of $\torsQ_\mm(\QQ_p)$
belong to the same orbit for $\TNS(\ZZ_p)$ their image
in $\overline\Gamma_p$ coincides. Conversely, let
$x$ and $y$ be elements of $\torsQ_\mm(\QQ_p)$ which have
the same image by $\pi_\mm$ and $\log_p$. Then there
exists an element $t\in\TNS(\QQ_p)$ such that $y=tx$.
Since $\log_p(x)=\log_p(y)$, if a coordinate $z_\delta$
of $x$ is different from $0$, the corresponding component
of $\log_p(t)$ is $0$. Taking into account the 
conditions~\eqref{equ.tors.open} and the equations~\eqref{equ.torus.fibre}
and~\eqref{equ.torus.can} which define $\TNS$, this implies
that $\log_p(t)$ is the unit element and thus $t\in\TNS(\ZZ_p)$.
\end{proof}
\begin{rema}
\label{rema.lifting}%
The idea behind the construction of $\mathcal D_{\mm,p}$
is first to consider
the intersection 
\[\widehat\pi_m^{-1}(\torsiZ(\ZZ_p))\cap\torscZ_\mm(\ZZ_p),\]
which is stable under the action of $\TNS(\ZZ_p)$.
For all primes $p$ for which there is good reduction,
this intersection coincides with $\torsZ_\mm(\ZZ_p)$.
More generally, if $p$ is good or if $p\not\equiv 1\mod 4$,
this intersection satisfies the conditions (i) to (iii) 
and yields the wanted domain.
On the other hand, if $p$ is a prime dividing one of the
$\Delta_{j,k}$ and such that $p\equiv 1\mod 4$, then
for any $Q\in\torsiZ(\ZZ_p)\cap\widehat\pi_\mm(\torsQ_\mm(\QQ_p))$ 
the intersection
\[\widehat\pi_\mm^{-1}(Q)\cap\torscZ_\mm(\ZZ_p)\]
is the union of a finite number of $\TNS(\ZZ_p)$-orbits.
We then select a total order on $\Gamma_p$ and choose
the minimal element in the image of the last intersection
by $\phi_p$. In that way, we construct the wanted domain.
\par
To better understand the construction,
let us first describe the conditions satisfied by $\log_p(R)$ for 
a lifting $R$ of a point $Q\in\torsiQ(\QQ_p)$. 
Let $R=(z_\delta)_{\delta\in\DDelta}\in\torsQ_\mm(\QQ_p)$ and let 
$Q=(x,y,t,u,v)=\widehat\pi_\mm(R)$. 
Let us denote by $(\rr_\delta)_{\delta\in\DDelta}\in
\overline\Gamma_p$ the image of $R$ by $\log_p$.
We also put $\boldsymbol n_j=\widehat v_p(L_j(u,v)/m_j)$ for $j\in\onefour$,
$\boldsymbol n_0=\widehat v_p(t)$
and $\boldsymbol n^{\pm}=\widehat v_p((x\pm\ci y)/\alpha_\mm)$. 
Then we have the relations
\begin{align}
\label{equ.rela.ni}%
\boldsymbol n_i&=\rr_i^++
\rr_i^-\\
\noalign{\noindent for $j\in\zerofour$, and}
\label{equ.rela.no}%
\boldsymbol n^{\pm}&=2
\rr_0^{\pm}+\sum_{j=1}^4\rr_j^{\pm}.
\end{align}
\end{rema}
\begin{lemma}
\label{lemma.coeff.inter}
Let $p$ be a prime number and let $\mm$ belong to $\Sigma$.
Let $Q$ belong to the intersection
$\torsiZ(\ZZ_p)\cap\pi_\mm(\torsQ_\mm(\QQ_p))$
and let $(\boldsymbol n_j)_{j\in\zerofour}$ and 
$\boldsymbol n^+,\boldsymbol n^-$
be the corresponding elements of $\ZZ^\Sp$
defined in remark~\ref{rema.lifting}.
\begin{assertions}
\item One has $\boldsymbol n_j\geq 0$ for $j\in\zerofour$, 
$\boldsymbol n^+\geq 0$ and $\boldsymbol n^-\geq 0$.
\item If $p\not\in\badp$, then $\min(\boldsymbol n_i,
\boldsymbol n_j)=0$
if $1\leq i<j\leq 4$.
\item If $p\not\equiv 1\mod 4$, then $\boldsymbol n_0=0$.
\item One has $\min(\boldsymbol n_0,\boldsymbol n^+,\boldsymbol n^-)=0$.
\item There exists a solution in $\Xi_p$ to the 
equations~\eqref{equ.rela.ni}
and \eqref{equ.rela.no}.
\item The number of such solutions is finite.
\item There exists a unique solution to these equations
in $\Xi_p$ if $p\not\in\badp$ or if $p\not\equiv 1\mod 4$.
\end{assertions}
\end{lemma}
\begin{proof}
We write $\mm=(m_1,\dots,m_4)$ and $Q=(x,y,t,u,v)$.
As $Q$ belongs to the set $\pi_\mm(\torsQ_\mm(\QQ_p))$,
one has that $p|m_i$ if and only if $p\equiv 3\mod 4$ and
$v_p(L_i(u,v))$ is odd.
If these conditions are verified, $v_p(\alpha_\mm)=1$
and $\alpha_\mm|L_i(u,v)$. Similarly, using the equation~\eqref{equ.torsiz},
we have that $\alpha_\mm|x\pm\ci y$ and this concludes the proof
of a).
\par
We now assume that $p \not\in\badp$. Let $i,j$ be such that
$1\leq i<j\leq 4$. Thus $p$ does not divide
$\Delta_{i,j}$. This implies that $\min(v_p(L_i(u,v)),v_p(L_j(u,v)))=0$
and so $\min(\boldsymbol n_i,\boldsymbol n_j)=0$.
\par
We now prove assertion c). If $p|t$ then by equation~\eqref{equ.torsiz},
it follows that $p^2|x^2+y^2$. If we assume that $p=2$ or $p\equiv 3\mod 4$
this implies that $p|x$ and $p|y$ which contradicts
the fact that $\min(v_p(x),v_p(y),v_p(t))=0$.
\par
Let $\mathfrak p\in\Sp$. If $\mathfrak p$ divides $x+\ci y$, $x-\ci y$
and $t$, then $p$ divides $x$, $y$ and $t$. This proves assertion d).
\par
Since $Q$ belongs to
$\pi_\mm(\torsQ(\QQ_p))$,
the equations~\eqref{equ.rela.ni}
and \eqref{equ.rela.no} have a solution in $\Gamma_p$. If $p\equiv 3\mod 4$
or $p=2$, then the integers $r_j^{\pm}\in\ZZ$ are
such that $r_j^+=r_j^-$ for $j\in\zerofour$.
Therefore the equations~\eqref{equ.rela.ni}
have a unique solution in $\Gamma_p$. By a) the coordinates
of this solution are positive. If $p\equiv 1\mod 4$, then
by choosing an element $\mathfrak p\in \Sp$ we are reduced
to solving the equations
\begin{align*}
 n_i&= r_i^++
 r_i^-\\
\noalign{\noindent for $j\in\zerofour$, and}
 n^{\pm}&=2 r_0^{\pm}+\sum_{j=1}^4 r_j^{\pm}.
\end{align*}
in $\ZZ^\DDelta$, where $n_j\geq 0$ for $j\in\zerofour$,
$n^+\geq 0$ and $n^-\geq 0$. Since we have the
relation $2n_0+\sum_{j=1}^4n_j=n^++n^-$, we may write
$n^+=2a_0^++\sum_{j=1}^4a_j^+$ where $0\leq a_j^+\leq n_j$
for $j\in\zerofour$. Then we put $a_j^-=n_j-a_j^+$
for $j\in\zerofour$ to get a solution with nonnegative coordinates.
\par
The assertion f) follows from the fact that there is only
a finite number of nonnegative 
integral solutions to an equation of the form
$n=k^++k^-$.
\par
If $p\equiv 3\mod 4$ or $p=2$ we have already seen that
the solution to the system of equations is unique.
If $p\not\in\badp$ and $p\equiv 1\mod 4$, 
then it follows from the assertions
b) and d) that $r_j^{\pm}=\min(n_j,n^{\pm})$,
which implies that the solution is unique.
\end{proof}
\begin{lemma}
\label{lemma.most.places}%
If $p$ is a prime number such that $p\equiv 1\mod 4$
or $p\not\in\badp$,
then for $\mm\in\Sigma$, the set 
$\torscZ_\mm(\ZZ_p)\cap\widehat\pi_\mm^{-1}(
\torsiZ(\ZZ_p))$
satisfies the conditions (i) to (iii) and defines
a fundamental domain in $\torsQ_\mm(\QQ_p)$ under the action
of $\TNS(\ZZ_p)$.
\end{lemma}
\begin{proof}
To prove the lemma it is sufficient to prove that
the intersection of any nonempty
fiber of $\pi_\mm$ with $\torsZ_\mm(\ZZ_p)$ is not empty and is an orbit
under the action of $\TNS(\ZZ_p)$.
Let $P$ belong to the set $\pi_\mm(\torsQ_\mm(\QQ_p))$.
By lemma~\ref{lemma.lift.inter} we may lift $P$ to a point
$Q$ which belongs to $\torsiZ(\ZZ_p)$.
According to lemma~\ref{lemma.coeff.inter}, e), we may find
an element $\rr\in\Xi_p$ which is a solution to the 
equations~\eqref{equ.rela.ni}
and \eqref{equ.rela.no}. Let $R'$ be any lifting of $P$ to
$\torsQ_\mm(\QQ_p)$ and let $\rr'=\log_p(R)$. The difference
$\rr'-\rr$ belongs to $X_*(\TNS)_p$. According to 
lemma~\ref{lemma.idelic}, there exists $t\in\TNS(\QQ_p)$ such
that $\log_p(t)=\rr-\rr'$. Then the point $R=t.R'\in\torsQ_\mm(\QQ_p)$
satisfies ${\log_p(R)=\rr}$
and $R$ belongs to
$\torscZ_\mm(\ZZ_p)\cap\widehat\pi_\mm^{-1}(
\torsiZ(\ZZ_p))$.
\par
It remains to prove that if two element $R$ and $R'$ 
of $\torsZ_\mm(\ZZ_p)$ are
in the same fibre for $\pi_\mm$ then they belong to the same
orbit under the action of $\TNS(\ZZ_p)$. Their images in
$\torsiZ(\QQ_p)$ belong to $\torsiZ(\ZZ_p)$ and therefore
are contained in the same orbit for the action of 
$\Tspl(\ZZ_p)$, which means that 
the equations described in remark~\ref{rema.lifting}
for $\log_p(R)$ and $\log_p(R')$
are exactly the same. We then apply assertion g) of 
lemma~\ref{lemma.coeff.inter} and
lemma~\ref{lemma.orbit}. 
\end{proof}
\begin{lemma}
If the prime number $p$ does not belong to $\badp$, 
then for $\mm\in\Sigma$, we have
\[\torsZ_\mm(\ZZ_p)=\torscZ_\mm(\ZZ_p)\cap\widehat\pi_\mm^{-1}(
\torsiZ(\ZZ_p)).\]
\end{lemma}
\begin{proof}
We keep the notation used in the proof of the previous lemma.
Using lemma~\ref{lemma.coeff.inter}, b)
and d), and the positivity of the coefficients in $\rr$, we get that 
$\min(\rr_{\delta_1},\rr_{\delta_2})=0$ whenever $\delta_1\cap\delta_2=\emptyset$,
which means that $R$ belongs to $\torsZ_\mm(\ZZ_p)$.
\end{proof}
\begin{defi}
Let $\mm$ belong to $\Sigma$.
If $p\not\in\badp$, we put $\mathcal D_{\mm,p}=\torsZ_\mm(\ZZ_p)$.
If $p\in\badp$ and $p\not\equiv 1\mod 4$, we put
\[\mathcal D_{\mm,p}=\torscZ_\mm(\ZZ_p)\cap\widehat\pi_\mm^{-1}(
\torsiZ(\ZZ_p)).\]
\end{defi}
It remains to define the domain for the primes $p\in\badp$
such that $p\equiv 1\mod 4$.
\begin{nota}
\label{nota.badp.one}
We put $\badp'=\{\,p\in\badp\mset  p\equiv 1\mod 4\,\}$. For any
$p\in\badp'$ we fix in the remainder of this text
a decomposition $p=\mathfrak \varpi_p\overline{\varpi_p}$
for an irreducible element $\varpi_p\in\ZZ[i]$.
We may then write $\Sp=\{\varpi_p,\overline{\varpi_p}\}$.
The group $\Gamma_p$ is isomorphic to $\ZZ^\DDelta$ through the
map $\phi_p$ which applies a family 
$(r_{\mathfrak p,\delta})_{(\mathfrak p,\delta)\in\Sp\times\DDelta}$ 
onto the family $(r_{\varpi_p,\delta})_{\delta\in\DDelta}$.
Let $j\neq k$ be two elements of $\onefour$ such that $p|\Delta_{j,k}$.
We then define
$\ff_{j,k}=(f_\delta)_{\delta\in\Delta}\in \ZZ^{\DDelta}$ by
\[f_\delta=\begin{cases}
1\text{ if $\delta\in\{D_j^-,D_k^+\}$},\\
0\text{ otherwise.}
\end{cases}\]
We put $\ee_{j,k}=\phi_p^{-1}(\ff_{j,k})$ and consider the set
\begin{equation}
\label{equ.def.lambdap}
\Lambda_p=\Xi_p\setminus\bigcup_{\{(j,k)\in\onefour\mid j<k
\text{ and }p\mid\Delta_{j,k}\}}
\ee_{j,k}+\Xi_p.
\end{equation}
\end{nota}
\begin{defi}
Let $\mm$ belong to $\Sigma$.
If $p\in\badp$ and $p\equiv 1\mod 4$, then we define $\mathcal D_{\mm,p}$
to be the set of $R\in\widehat\pi_\mm^{-1}(\torsiZ(\ZZ_p))$ such that
$\log_p(R)\in\Lambda_p$.
\end{defi}
\begin{rema}
In particular, one has $\mathcal D_{\mm,p}\subset
\torscZ_\mm(\ZZ_p)$ for any prime number $p$.
\end{rema}
\begin{lemma}
\label{lemma.bad.places}%
If $p\in\badp$ and $p\equiv 1\mod 4$,
then for $\mm\in\Sigma$, the set 
$\mathcal D_{\mm,p}$
satisfies the conditions (i) to (iii) and defines
a fundamental domain in $\torsQ_\mm(\QQ_p)$ under the action
of $\TNS(\ZZ_p)$.
\end{lemma}
\begin{proof}
According to lemma~\ref{lemma.orbit} and lemma~\ref{lemma.coeff.inter} e),
we have only to prove that for any 
$Q\in\torsiZ(\ZZ_p)\cap\widehat\pi_\mm(\torsQ_p)$, there
exist a unique solution of the equations~\eqref{equ.rela.ni}
and \eqref{equ.rela.no} which belongs to $\Lambda_p$.
Among the solutions in $\Xi_p$, there is a unique solution such that
if $\sss=\phi_p(\rr)$, the quadruple $(s_1^+,s_2^+,s_3^+,s_4^+)$
is maximal for the lexicographic order. It remains to prove that the solution
satisfies this last condition if and only if $\rr$ belongs to $\Lambda_p$.
Let $\rr$ be the solution for which the above quadruple is maximal
and $\widetilde\rr$ be any solution in $\Xi_p$ 
and $\widetilde\sss=\phi_p(\widetilde\rr)$.
If $\rr\neq\widetilde\rr$, then we consider the smallest $j\in\onefour$
such that $s_j^+>\widetilde s_j^{\,+}$. With the notation
of remark~\ref{rema.lifting}, this implies that
$\nn_j\neq 0$, $\nn^+\neq 0$ and $\nn^-\neq 0$. Therefore
$\nn_0=0$ and there exists $k>j$ such that $s_k^+<\widetilde s_k^{\,+}$. 
Since $s_j^-<\widetilde s_j^{\,-}$, we may conclude that
$\widetilde\rr\in \ee_{j,k}+\Xi_p$. 
Moreover $p\mid\Delta_{j,k}$.
Conversely if $\widetilde\rr$
belongs to $\ee_{j,k}+\Xi_p$, 
for some $j,k\in\onefour$ such that $j<k$,
then $\widetilde\rr-\ee_{j,k}+\ee_{k,j}$
is another solution to system of equations which gives
a bigger quadruple for the lexicographic order.
\end{proof}
\Subsection{Adelic domains and lifting of the points}
\begin{defi}
Let $\mm\in\Sigma$.
We define the open subset $\mathcal{D}_\mm$ of $\torsQ_\mm(\Adeles_\QQ)$
as the product $\torsQ_\mm(\RR)\times\prod_{p\in\primes}\mathcal D_{\mm,p}$.
\end{defi}
\begin{prop}
The set $\mathcal D_\mm$ is a fundamental domain
in $\torsQ_\mm(\Adeles_\QQ)$ under the action of
$\TNS(\QQ)$ modulo $\TNS(\QQ)_\tors$. In other words
\begin{conditions}
\item The open set $\mathcal D_\mm$ is stable
under the action of $\TNS(\QQ)_\tors$;
\item For any $t$ in $\TNS(\QQ)\setminus\TNS(\QQ)_\tors$,
one has $t.\mathcal D_\mm\cap\mathcal D_\mm=\emptyset$;
\item For any~$x$ in $\torsQ_\mm(\Adeles_\QQ)$, there exists
an element~$t$ in $\TNS(\QQ)$ such that~$x$ belongs to
$t.\mathcal D_\mm$.
\end{conditions}
\end{prop}
\begin{proof}
The assertion (i) follows from the fact that $\mathcal D_{\mm,p}$
is stable under $\TNS(\ZZ_p)$ for any prime number $p$.
If $t$ belongs to $\TNS(\QQ)\setminus\TNS(\QQ)_\tors$,
then, by lemma~\ref{lemma.idelic}, there exists a prime number $p$
such that $\log_p(t)\neq 0$. Thus  
$t.\mathcal D_{\mm,p}\cap\mathcal D_{\mm,p}=\emptyset$, which proves (ii).
Let~$x$ belong to $\torsQ_\mm(\Adeles_\QQ)$. For any prime number~$p$, 
there exists an element $t_p\in\TNS(\QQ_p)$ such that
$t_p.x\in\mathcal D_{\mm,p}$. By lemma~\ref{lemma.idelic},
there exists an element $t\in\TNS(\QQ)$ such that
$\log_p(t)=\log_p(t_p)$ for any prime number $p$
and $t.x\in\mathcal D_\mm$.
\end{proof}
\begin{cor}
\label{cor.snow}%
Let $P$ belong to $S(\QQ)$ and let $\mm$ be the unique
element of $\Sigma$ such that $P\in\pi_\mm(\torsQ_\mm(\QQ))$.
Then
\[\card(\pi_\mm^{-1}(P)\cap\mathcal D_\mm)=\card\TNS(\QQ)_\tors=2^{8}.\]
\end{cor}
\begin{proof}
This corollary follows from the last proposition and the
fact that $\pi_\mm^{-1}(x)$ is an orbit under the action of
$\TNS(\QQ)$.
\end{proof}
Let us now lift the heights to the versal torsors.
\begin{defi}
\label{defi.localheight}%
As in notation~\ref{nota.height} we put $C=\sqrt{\prod_{j=1}^4
|a_j|+|b_j|}$.
Let $w$ be a place of $\QQ$. We define a function $H_w$
on $\QQ_w^5$ by
\begin{equation*}
H_w(x,y,t,u,v)=
\begin{cases}
\max(\frac{|x|_w}C,\frac{|y|_w}C,\max(|u|_w,|v|_w)^2|t|_w)&
\text{if $w=\infty$,}\\
\max(|x|_w,|y|_w,\max(|u|_w,|v|_w)^2|t|_w)&\text{otherwise,}
\end{cases}
\end{equation*}
for any $(x,y,t,u,v)\in\QQ_w^5$. If $\mm\in\Sigma$,
we shall also denote
by $H_w:\torsQ_\mm(\QQ_w)\to\RR$ the composite function
$H_w\circ\widehat\pi_\mm$. We then define $H:\torsQ_\mm(\Adeles_\QQ)\to\RR$
by $H=\prod_{w\in\Val(\QQ)}H_w$.
\end{defi}
\begin{rems}
(i) The line bundle $\omega_S^{-1}$ defines a character
$\chi_\omega$ on the torus $\Tspl=\GmQ^2$ simply given by 
${(\lambda,\mu)\mapsto\lambda}$ and we have
the relation
\begin{equation}
\label{equ.localheight}
H_w(t.R)=|\chi_\omega(t)|_wH_w(R)
\end{equation}
for any $t\in\Tspl(\QQ_w)$ and any $R\in\Tspl(\QQ_w)$.
A similar assertion is true on $\torsQ_\mm$ for $\mm\in\Sigma$.
\par
(ii) 
As a point $Q=(x:y:t:u:v)$ in $\torsiZ(\RR)$
satisfies the equations~\eqref{equ.torsiz},
we have that
\[\max(\vert x\vert,\vert y\vert)^2\leq 
\prod_{j=1}^4(|a_j|+|b_j|)
\max(\vert u\vert,\vert v\vert)^4\vert t\vert^2.\]
and it follows that
\[H_\infty(Q)=\max(\vert u\vert,\vert v\vert)^2\vert t\vert.\]
\end{rems}
\begin{prop}
\label{prop.liftingheight}%
Let $\mm\in\Sigma$.
For any $R\in\torsQ_\mm(\QQ)$, one has
\[H(\pi_\mm(R))=H(R).\]
\end{prop}
\begin{proof}
We may define a map $\widehat\psi:\QQ^5\to\QQ^5$ by
$(x,y,t,u,v)\mapsto(v^2t:uvt:u^2t:x:y)$. The
restriction of the map~$\widehat\psi$ from
$\torsiZ$ to $\Aff_\QQ^5\setminus\{0\}$
is a lifting of the map $\psi:S\to S'$.
On $S'$ the height $H_4$ is given by
\[H_4(x_0:\cdots:x_4)=
\max\left(\vert x_0\vert_\infty,\vert x_1\vert_\infty,\vert x_2\vert_\infty,
\frac{\vert x_3\vert_\infty}C,\frac{\vert x_4\vert_\infty}C\right)
\times\prod_{p\in\primes}\max_{0\leq j\leq 4}(|x_j|_p)\]
for any $(x_0,\dots,x_4)\in\QQ^5$. This formula implies the statement
of the lemma.
\end{proof}
\begin{cor}
For any real number $B$, we have
\[N(B)=\frac 1{\card\TNS(\QQ)_\tors}
\sum_{\mm\in\Sigma}\card\{\,R\in\torsQ_\mm(\QQ)\cap\mathcal D_\mm
\mset H(R)\leq B\,\}\]
\end{cor}
\begin{proof}
This corollary follows from propositions~\ref{prop.uniquem},
\ref{prop.torsor},
and~\ref{prop.liftingheight} and corollary \ref{cor.snow}.
\end{proof}
\begin{rema}
For any prime number~$p$ and any $\mm\in\Sigma$, we have 
$\mathcal D_{\mm,p}\subset\widehat\pi_\mm^{-1}(\torsiZ(\ZZ_p))$.
Therefore, for any $R=(R_w)_{w\in\Val(\QQ)}$ belonging
to $\mathcal D_\mm$, we have $H(R)=H_\infty(R_\infty)$. 
\end{rema}
\begin{nota}
For any real number~$B$, and any $\mm\in\Sigma$,
we denote by $\mathcal D_{\mm,\infty}(B)$ the set
of $R\in\torsQ_\mm(\RR)$ such that the point
$Q=(x,y,t,u,v)=\widehat\pi_\mm(R)$
satisfies the conditions
\begin{equation}
\label{equ.domain.infty}
H_\infty(Q)\leq B\quad\text{and}\quad H_\infty(Q)\geq\max(|u|,|v|)^2\geq 1.
\end{equation}
We define $\mathcal D_\mm(B)$ as the product
$\mathcal D_{\mm,\infty}(B)\times\prod_{p\in\primes}\mathcal D_{\mm,p}$.
\end{nota}
\begin{rema}
Let $F$ be a fiber of the morphism $\pi:S\to\Proj^1_\QQ$.
Then the Picard group of $S$ is a free $\ZZ$-module
with a basis given by the pair $([F],[\omega_S^{-1}])$.
According to the formula~\eqref{equ.localheight},
the function $H_\infty$ corresponds to $[\omega_S^{-1}]$.
In a similar way the map applying $(x,y,t,u,v)$ to $\max(|u|,|v|)$ 
corresponds to $[F]$. On the other hand, the cone of effective
divisors in $\Pic(S)$ is the cone generated by $[F]$ and 
$[E^+]+[E^-]=[\omega_S^{-1}]-2[F]$. But, by the preceding remark, the function
\[Q=(x,y,t,u,v)\longmapsto\frac{H_\infty(Q)}{\max(|u|,|v|)^2}\]
corresponds to $[E^+]+[E^-]$. Thus the lower bounds
imposed in the definition of $\mathcal D_{\mm,\infty}(B)$
corresponds to the condition (3.9)
of~\cite[p.~268]{peyre:cercle}.
\par
These lower bounds are automatically
satisfied by any point~$R$ in $\mathcal D_\mm\cap\torsQ_\mm(\QQ)$.
Indeed $Q=\widehat\pi_\mm(R)$ belongs to $\torsiZ(\ZZ)$ and
writing $Q=(x,y,t,u,v)$ we get that $\max(|u|,|v|)\geq 1$.
Since $(x,y,t)\neq 0$, by equation~\eqref{equ.torsiz},
we also have that $t\neq 0$ and therefore $|t|\geq 1$ which yields
the second inequality.
\end{rema}
\begin{cor}
For any real number $B$, we have
\[N(B)=\frac 1{\card\TNS(\QQ)_\tors}
\sum_{\mm\in\Sigma}\card(\torsQ_\mm(\QQ)\cap\mathcal D_\mm(B)).\]
\end{cor}
\begin{proof}
This follows from the last remark and the preceding corollary.
\end{proof}
\subsection{Moebius inversion formula and change of variables}
As is usual with these type of problems, we now wish
to use a Moebius inversion formula to replace the
primality conditions by divisibility conditions.
In fact we shall perform three inversions corresponding
to the various primality conditions.
\par
We shall simultaneously parametrize the sets thus introduced
to reduce our problem to the study of a series which
may be handled with techniques of analytic number theory.
\subsubsection{First inversion}
The first inversion corresponds to the conditions
imposed at the places $p\in\badp$ with $p\equiv 1\mod 4$.
\begin{nota}
Let $\N(\fa) =\#(\ZZ[\ci]/\fa)$ denote the norm of
an ideal $\fa$ of the ring of Gaussian integers $\ZZ[\ci]$.
We define
\begin{equation*}
\widehat{\mathfrak{D}}=\{\mathfrak b\subset \ZZ[\ci]\mset
\N(\mathfrak b)\in \mathfrak{D}\},
\end{equation*}
where
\begin{equation}\lab{18-D}
\mathfrak{D}=\{d\in\ZZp\mset p\mid d \Rightarrow p\equiv 1 \bmod{4} \}.
\end{equation}
\par
Let $A$ be a commutative ring.
Let $\bb=(\mathfrak b_\delta)_{\delta\in\DDelta}$
be a family of ideals of $A\otimes_\ZZ\ZZ[\ci]$ such that
$\mathfrak b_{\overline \delta}=\overline{\mathfrak b_\delta}$ for
any $\delta\in\DDelta$. Then $(\prod_{\delta\in\DDelta}\mathfrak b_\delta)^\Ga$
is an ideal of $A_\DDelta$ and for any $\nn\in\ZZ^4$, we define
\[\torscZ_\nn(\bb)=\torscZ_\nn(A)\cap
\Bigl(\prod_{\delta\in\DDelta}\mathfrak b_\delta\Bigr)^\Ga.\]
We define $\mathcal I_\DDelta(A)$ as the set of such families of ideals.
For any $p$, the map $\log_p$ induces a map from
$\mathcal I_\DDelta(\ZZ)$ to $\overline\Gamma_p$.
If $\log_2(\faa)=0$, then we define
\[\llambda(\faa)=\prod_{p\in\primes\setminus\{2\}}\exp_{\varpi_p}(\log_p(\faa)).\]
For any $\faa\in\mathcal I_\DDelta(\ZZ)$, we also put
$\N(\faa)=(\N(\mathfrak a^+_j))_{1\leq j\leq 4}\in\NN^4$.
\par
If $\llambda=(\lambda_\delta)_{\delta\in\DDelta}$
belongs to $T_\DDelta(\QQ)\cap \ZZ_\DDelta$,
then we put
$\N(\llambda)=(\lambda^+_j\lambda^-_j)_{1\leq j\leq 4}\in\ZZp^4$ 
and define a morphism
$m_\llambda:\torscZ_{\N(\llambda)\nn}\to\torscZ_\nn$ using the action
of the torus $T_\DDelta$ on $\Aff_\DDelta$.
For any commutative ring $A$, we may define an element $\llambda A_\DDelta\in
\mathcal I_\Delta(A)$ by taking the family of ideals 
$(\lambda_\delta A)_{\delta\in\DDelta}$. If $\faa\in\mathcal I_\DDelta(\ZZ)$
satisfies $\log_2(\faa)=0$, then $\faa=\llambda(\faa)\ZZ_\DDelta$.
For any 
$\faa\in\mathcal I_\DDelta(\ZZ)$, we similarly define
$\faa A_\DDelta$ as 
$(\mathfrak a_\delta A)_{\delta\in\DDelta}\in\mathcal I_\DDelta(A)$.
\par
Let $\mm\in\Sigma$ and let $\faa=(\fa_j)_{1\leq j\leq 4}\in
\widehat{\mathfrak D}^4$. We may see $\faa$ as an element 
of $\mathcal I_\DDelta(\ZZ)$ by putting $\fa_j^+=\fa_j$ and $\fa_j^-=\overline
\fa_j$ for $j\in\onefour$ and $\fa_0^+=\fa_0^-=\ZZ[\ci]$.
Let $\nn=\mm\N(\faa)=(m_j\N(\fa_j))_{1\leq j\leq 4}$.
Recall that $\alpha_\mm$ is the positive square root of
$\prod_{j=1}^4m_j$. 
We put
\[\alpha_{\mm,\faa}=\alpha_\mm\times\prod_{j=1}^4\lambda(\faa)_j^+.\]
Note that $\prod_{j=1}^4n_j=N(\alpha_{\mm,\faa})$.
We then define a map $\widehat\pi_{\mm,\faa}:\torscZ_\nn\to\Aff_\ZZ^5$
as follows: thanks to equations~\eqref{equ.torcz}
and the fact that, by~\eqref{equ.pthreefour},
the family $(a_j,b_j)_{1\leq j\leq 4}$ generates $\ZZ^2$,
the system of equations
\begin{equation}
L_j(U,V)=n_j(X_j^2+Y_j^2)
\end{equation}
in the variables $U$ and $V$ has a unique solution
in the ring of functions on $\torscZ_\nn$.
We also define $T=X_0^2+Y_0^2$ and define $X$ and $Y$ by
the relation
\[X+\ci Y=\alpha_{\mm,\faa}(X_0+\ci Y_0)^2\prod_{j=1}^4(X_j+\ci Y_j).\]
The morphism $\widehat\pi_{\mm,\faa}$ is then defined
by the family of functions $(X,Y,T,U,V)$.
Since these functions satisfy the relation
\[X^2+Y^2=T^2\prod_{j=1}^4L_j(U,V),\]
the image of $\widehat\pi_{\mm,\faa}$ is contained in the
Zariski closure $\torsicZ$ of $\torsiZ$ in $\Aff_\ZZ^5$.
\par
Let $\mm\in\Sigma$ and $\faa\in\widehat\fD^4$. For any prime number $p$
we define $\mathcal D_{\mm,\faa,p}^1$ as 
$\torscZ_\nn(\ZZ_p)\cap\widehat\pi^{-1}_{\mm,\faa}(\torsiZ(\ZZ_p))$
where $\nn=\mm\N(\faa)$.
For any real number $B$, we also define 
$\mathcal D^1_{\mm,\faa,\infty}(B)$ as the set of $R\in\torscZ_\nn(\RR)$
such that $\widehat\pi_{\mm,\faa}(R)$ satisfies the 
conditions~\eqref{equ.domain.infty}.
We then put $\mathcal D^1_{\mm,\faa}(B)=\mathcal D^1_{\mm,\faa,\infty}(B)
\times\prod_{p\in\primes}\mathcal D^1_{\mm,\faa,p}$.
When $\faa_j=\ZZ[\ci]$ for $j\in\onefour$, we shall forget $\faa$ in
the notation.
\par
Let $\badp'$ be the set of
$p\in\badp$ such that $p\equiv 1\mod 4$.
For any $p\in\badp'$,
we consider the set $\mathcal E_p$ of subsets $I$
of $\Delta\setminus\{E^+,E^-\}$ such that
\begin{conditions}
\item if $\delta_j^+\in I$ then there exists  $k<j$ such that
$\delta_k^-\in I$;
\item if $\delta_k^-\in I$ then there exists  $j>k$ such that
$\delta_j^+\in I$;
\item if $\delta_j^+\in I$ and $\delta_k^-\in I$ with $j\neq k$
then $p\mid\Delta_{j,k}$.
\end{conditions}
For any $I\in\mathcal E_p$ we define 
$\ff_I=(f_\delta)_{\delta\in\DDelta}\in\ZZ^\DDelta$
by
\[f_\delta=\begin{cases}
1&\text{if $\delta\in I$,}\\
0&\text{otherwise.}
\end{cases}\]
Using notation~\ref{nota.badp.one}, we then consider
$\ee_I=\varphi_p^{-1}(\ff_I)$ and
$\Sigma'_p=\bigl\{\,\exp_{\varpi_p}(\ee_I),\  I\in
\mathcal E_p\,\bigr\}$.
We define $\Sigma'$ as the subset of $\mathcal I_\DDelta(\ZZ)$
defined by
\[\Sigma'=\biggl\{\,\biggl(\prod_{p\in\badp'}\llambda_p
\biggr)\ZZ_\DDelta,\ (\llambda_p)_{p\in\badp'}\in
\prod_{p\in\badp'}\Sigma'_p\,\biggr\}\]
An element $\faa\in\Sigma'$ is determined
by the quadruple $(\fa_j^+)_{1\leq j\leq 4}$ and we shall
also consider $\Sigma'$ as a subset of $\widehat\fD^4$.
For $p\in\badp'$ we define a map $\mu_p:\mathcal E_p\to\ZZ$
by the conditions
\[\mu_p(\emptyset)=1\quad \text{and}\quad
\sum_{J\subset I}\mu_p(J)=0\text{ if $I\neq\emptyset$.}\]
The map $\mu:\Sigma'\to\ZZ$ is defined by 
$\mu(\faa)=\prod_{p\in\badp'}\mu_p(I_p(\faa))$.
\par
We shall denote by $\Adeles_{f,\infty}$ the ring
$\RR\times\prod_{p\in\primes}\ZZ_p$.
\end{nota}
\begin{rems}
\label{rems.notas.ideals}
(i) Let $\llambda=(\lambda_\delta)_{\delta\in\DDelta}\in T_\DDelta(\QQ)\cap
\ZZ_\DDelta$.
Let $A$ be a commutative ring.
Then $m_\llambda$ is a bijection from the set 
$\torscZ_{\N(\llambda)\nn}(A)$ to the set 
$\torscZ_\nn(\llambda A_\DDelta)$.
\par
(ii) With the same notation, for the ring $A=\ZZ_p$, the set
$\torscZ_\nn(\boldsymbol{\mathfrak d})$ is the inverse 
image by $\log_p$ of the set $\log_p(\llambda)+\Xi_p$.
\end{rems}
\begin{lemma}
\label{lemma.firstinversion}%
Let $p\in\badp'$. For any subset $K$ of $\Gamma_p$,
we denote by $\oone_K$ its characteristic function.
Then
\[\oone_{\Lambda_p}=\sum_{I\in\mathcal E_p}\mu_p(I)\oone_{\ee_I+\Xi_p}.\]
\end{lemma}
\begin{proof}
For any $j,k$ in $\onefour$ such that $j<k$ and $p\mid\Delta_{j,k}$,
we put $I_{j,k}=\{\delta^-_j,\delta^+_k\}$. 
Let $K$ be a subset of $\{\,(j,k)\in\onefour^2\mset
j<k\text{ and }p\mid\Delta_{j,k}\,\}$. 
Let $I=\bigcup_{(j,k)\in K}I_{j,k}$. Then
we have
\[\bigcap_{(j,k)\in K}(\ee_{j,k}+\Xi_p)=\ee_I+\Xi_p.\]
On the other hand, a subset $I$ of $\DDelta$ belongs
to $\mathcal E_p$ if and only if it is the union of subsets
$I_{j,k}$ with $j<k$ and $p\mid\Delta_{j,k}$.
The lemma then follows from equation~\eqref{equ.def.lambdap}
which defines $\Lambda_p$
and the fact that the map $I\mapsto\ee_I+\Xi_p$ reverses the
inclusions.
\end{proof}
\begin{lemma}
\label{lemma.multlambda}
Let $\faa\in\Sigma'$ and let $B$ be a positive real number.
The multiplication by $\llambda(\faa)\in T_\DDelta(\QQ)$
maps  $\mathcal D^1_{\mm,\faa}(B)$ onto 
$\mathcal D^1_{\mm}(B)\cap\torscZ_\mm(\faa(\Adeles_{f,\infty})_\DDelta)$.
\end{lemma}
\begin{proof} 
By remark~\ref{rems.notas.ideals} (i),
the map $m_{\llambda(\faa)}$ is a bijection
from the set $\torscZ_{\N(\faa)\mm}(\Adeles_{f,\infty})$
onto the set $\torscZ_\mm(\faa(\Adeles_{f,\infty})_\DDelta)$.
Let us now compare the maps $\widehat\pi_\mm\circ m_{\llambda(\faa)}$
and $\widehat\pi_{\mm,\faa}$.
The map $\widehat\pi_{\mm,\faa}$ is given by the relations
\[\begin{cases}
L_j(U,V)=\N(\mathfrak a^+_j)m_i(X_j^2+Y_j^2)\text{ for }j\in\onefour,\\
T=X_0^2+Y_0^2,\\
X+\ci Y=\alpha_{\mm,\faa}(X_0+\ci Y_0)^2\prod_{j=1}^4(X_j+\ci Y_j),
\end{cases}\]
whereas $\widehat\pi_\mm\circ m_{\llambda(\faa)}$ is given by
\[\begin{cases}
L_j(U,V)=\lambda(\faa)^+_j\lambda(\faa)^-_j
m_i(X_j^2+Y_j^2)\text{ for }j\in\onefour,\\
T=X_0^2+Y_0^2,\\
X+\ci Y=\alpha_{\mm}\biggl(\prod_{j=1}^4\lambda(\faa)_j^+\biggr)
(X_0+\ci Y_0)^2\prod_{j=1}^4(X_j+\ci Y_j).
\end{cases}\]
Therefore $\widehat\pi_\mm\circ m_{\llambda(\faa)}$
coincides with $\widehat\pi_{\mm,\faa}$.
This proves that for any prime number $p$, the
map $m_{\lambda(\faa)}$ maps $\widehat\pi_{\mm,\faa}^{-1}(\ZZ_p)$
onto $\widehat\pi_\mm^{-1}(\ZZ_p)$.
Moreover $m_{\llambda(\faa)}$ sends the set $\mathcal D^1_{\mm,\faa,\infty}(B)$ 
onto $\mathcal D^1_{\mm,\infty}(B)$.
\end{proof}
\begin{prop}
For any real number $B$, we have
\[N(B)=\frac 1{\card\TNS(\QQ)_\tors}
\sum_{\mm\in\Sigma}\sum_{\faa\in\Sigma'}
\mu(\faa)\card(\torsQ_{\N(\faa)\mm}(\QQ)\cap
\mathcal D^1_{\mm,\faa}(B)).\]
\end{prop}
\begin{proof}
This follows from lemma~\ref{lemma.firstinversion},
the definition of $\mathcal D_\mm(B)$ and lemma~\ref{lemma.multlambda}.
\end{proof}

\subsubsection{Second inversion}
The inversion we shall now perform corresponds
to the condition $\gcd(x,y,t)=1$.
\begin{nota}
The map $\mu:\widehat{\mathfrak D}\to \ZZ$ is the multiplicative function
such that
\[\mu(\mathfrak p^k)=\begin{cases}
1&\text{if $k=0$,}\\
-1&\text{if $k=1$,}\\
0&\text{otherwise.}
\end{cases}\]
for any prime ideal $\mathfrak p$ in $\widehat{\mathfrak D}$ and
any integer $k\geq 0$.
\par
Let $\mm\in\Sigma$ and $\faa\in\Sigma'\subset\widehat\fD^4$. Let 
$\fbb=(\fb_j)_{j\in\onefour}\in\widehat{\mathfrak{D}}^4$. 
We put $\nn=\N(\faa\fbb)\mm$ and
$\mu(\fbb)=\prod_{j=1}^4\mu(\fb_j)$.
Let $B$ be a real number. Let $p$ be a prime number.
If $R$ belongs to $\torscZ_{\nn}(\ZZ_p)$, we denote
by $X,Y,T,U$ and~$V$ the functions on $\torscZ_{\nn}$
which define $\widehat\pi_{\mm,\faa\fbb}$.
The local domain $\mathcal D^2_{\mm,\faa,\fbb,p}$ is then defined
as follows:
\begin{itemize}
\item 
If $p\equiv 3\mod 4$ or $p=2$, then $\mathcal D^2_{\mm,\faa,\fbb,p}$
is the set of $R\in\torscZ_{\nn}(\ZZ_p)$
such that $T(R)\in\ZZ_p^*$ and $\min(v_p(U(R)),v_p(V(R)))=0$;
\item
If $p\equiv 1\mod 4$ then $\mathcal D^2_{\mm,\faa,\fbb,p}$
is the set of $R=(z_\delta)_{\delta\in\DDelta}\in
\torscZ_{\nn}(\ZZ_p)$
such that $z_0^-$ belongs to $\bigcap_{j=1}^4\fb_j$, such that
$\min\bigl(v_p(T(R)),v_p\bigl(\prod_{j=1}^4\N(\faa_j)\bigr)\bigr)=0$
and such that
$\min(v_p(U(R)),v_p(V(R)))=0$.
\end{itemize}
We also put $\mathcal D^2_{\mm,\faa,\fbb,\infty}(B)=\mathcal D^1_{\mm,\faa,\infty}(B)$
and
\[\mathcal D^2_{\mm,\faa,\fbb}(B)=\mathcal D^2_{\mm,\faa,\fbb,\infty}(B)
\times\prod_{p\in\primes}\mathcal D^2_{\mm,\faa,\fbb,p}.\]
\end{nota}
\begin{prop}
\label{prop.secondinversion}%
For any real number $B$, we have the relation
\[N(B)=\frac 1{\card\TNS(\QQ)_\tors}
\sum_{\mm\in\Sigma}\sum_{\faa\in\Sigma'}\sum_{\fbb\in\widehat
{\mathfrak D}^4}
\mu(\faa)\mu(\fbb)
\card(\torsQ_{\N(\faa)\N(\fbb)\mm}(\QQ)\cap
\mathcal D^2_{\mm,\faa,\fbb}(B)).\]
\end{prop}
\begin{proof}
Let $\mm\in\Sigma$, let $\faa\in\Sigma'$ and let $p$ be a prime
number.
\par 
Let us first assume that $p\not\equiv 1\mod 4$.
By lemma~\ref{lemma.coeff.inter} c), we have $v_p(t)=0$
for any $(x,y,t,u,v)\in\torsiZ(\ZZ_p)$. Conversely,
let $R$ belong to ${\torscZ_{\mm\N(\faa)}(\ZZ_p)}$. If $v_p(T(R))=0$,
then ${\min(v_p(X(R)),v_p(Y(R)),v_p(T(R)))=0}$.
\par
We now assume that $p\equiv 1\mod 4$.
For any $R=(z_\delta)_{\delta\in\DDelta}\in\torscZ_{\mm\N(\faa)}(\QQ_p)$
we have the relations
\[T(R)=z_0^+z_0^-\quad\text{and}\quad
X(R)+\ci Y(R)=\alpha_{\mm,\faa}(z_0^+)^2\prod_{j=1}^4z_j^+.\]
Note that if $\varpi_p|\alpha_{\mm,\faa}$ for any prime $p\equiv 1\mod 4$,
then $p|\alpha_{\mm,\faa}$.
Therefore we have the relation $\gcd(X(R),Y(R),T(R))=1$ in $\ZZ_p$
if and only if $R$ satisfies the following two conditions:
\begin{conditions}
\item One has $\min(v_p(T(R)),v_p(\N(\prod_{j=1}^4\fa_j)))=0$;
\item There is no $j\in\onefour$ and no $\varpi\in\Sp$ such that
$z_j^+\in\varpi$ and $z_0^+\in\overline\varpi$.
\end{conditions}
We denote by $\widehat\fbb$ the unique element
of $\mathcal I_\DDelta(\ZZ)$ such that
$\widehat\fb_j^+=\fb_j$ for $j\in\onefour$
and $\widehat\fb_0^-=\bigcap_{j=1}^4\fb_j$.
A classical Moebius inversion yields that the characteristic
function of the set of the elements $R$ in $\torscZ_{\mm\N(\faa)}(\ZZ_p)$
which satisfy condition (ii)
is equal to
\[\sum_{\fbb\in\widehat{\mathfrak D}^4}
\mu(\fbb)\oone_{\torscZ_{\mm\N(\faa)}\bigl(\widehat\fbb(\ZZ_p)_\DDelta\bigr)}.
\]

By remark~\ref{rems.notas.ideals}~(i), the multiplication
map $m_{\llambda(\fbb)}$ maps 
$\torscZ_{\mm\N(\faa)}\bigl(\widehat\fbb(\ZZ_p)_\DDelta\bigr)$ 
onto the set of $(z_\delta)_{\delta\in\DDelta}$ in
$\torscZ_{\mm\N(\faa\fbb)}(\ZZ_p)$
such that $z_0^-$ belongs to $\bigcap_{j=1}^4\fb_j$.
The rest of the proof is similar to the proof of lemma~\ref{lemma.multlambda}.
\end{proof}

\subsubsection{Third inversion}
The last inversion  corresponds to
the condition $\gcd(u,v)=1$, in which it will prove
nonetheless useful to retain the fact that $u,v$ cannot both be even.
\begin{nota}
Let $\mm\in\Sigma$ and $\faa\in\Sigma'$. 
Let $\fbb=(\fb_j)_{j\in\onefour}\in\widehat{\mathfrak{D}}^4$.
We put $\nn=\N(\faa)\N(\fbb)\mm$. 
Let~$\ell$ be an odd integer.
Let~$p$ be a prime number. The local domain $\mathcal D^3_{\mm,\faa,\fbb,\ell,p}$
is then defined as follows:
\begin{itemize}
\item 
If $p=2$, then $\mathcal D^3_{\mm,\faa,\fbb,\ell,p}$
is the set of $R\in\torscZ_{\nn}(\ZZ_p)$
such that $T(R)\in\ZZ_p^*$ and $\min(v_p(U(R)),v_p(V(R)))=0$;
\item
If $p\equiv 3\mod 4$, then $\mathcal D^3_{\mm,\faa,\fbb,\ell,p}$
is the set of $R\in\torscZ_{\nn}(\ZZ_p)$
such that $T(R)\in\ZZ_p^*$ and $\ell$ divides $U(R)$ and $V(R)$.
\item
If $p\equiv 1\mod 4$ then $\mathcal D^3_{\mm,\faa,\fbb,\ell,p}$
is the set of $R=(z_\delta)_{\delta\in\DDelta}\in
\torscZ_{\nn}(\ZZ_p)$
such that $z_0^-$ belongs to $\bigcap_{j=1}^4\fb_j$, such that
$\min\bigl(v_p(T(R)),v_p\bigl(\prod_{j=1}^4\N(\faa_j)\bigr)\bigr)=0$
and $\ell$ divides $U(R)$ and $V(R)$.
\end{itemize}
We define $\mathcal D^3_{\mm,\faa,\fbb,\ell,\infty}(B)
=\mathcal D^2_{\mm,\faa,\fbb,\infty}(B)$
and 
\[\mathcal D^3_{\mm,\faa,\fbb,\ell}(B)
=\mathcal D^3_{\mm,\faa,\fbb,\ell,\infty}(B)
\times\prod_{p\in\primes}\mathcal D^3_{\mm,\faa,\fbb,\ell,p}.\]
\end{nota}
\begin{prop}
\label{prop.aftermoebius}%
For any positive real number $B$,
we have that $N(B)$ is equal to
\[\frac 1{\card\TNS(\QQ)_\tors}
\sum_{\mm\in\Sigma}\sum_{\faa\in\Sigma'}\sum_{\fbb\in\widehat
{\mathfrak D}^4}
\sum_{\substack{\ell=1\\ 2\nmid \ell}}^\infty
\mu(\faa)\mu(\fbb)\mu(\ell)
\card(\torsQ_{\N(\faa)\N(\fbb)\mm}(\QQ)\cap
\mathcal D^3_{\mm,\faa,\fbb,\ell}(B)).\]
\end{prop}

\section{Formulation of the counting problem}

We are now ready to begin the analytic part of the proof 
of theorem \ref{theo.main}. Let us recall that
the linear forms that we are working with take the shape
\[\begin{array}{ll}
L_1(U,V)=U, & 
L_2(U,V)=V, \\
L_3(U,V)=a_3U+b_3V, &
L_4(U,V)=a_4U+b_4V,
\end{array}\]
with integers $a_3,b_3,a_4,b_4$ such that 
$\gcd(a_3,b_3)=\gcd(a_4,b_4)=1$ and 
\begin{equation}
  \label{eq:resultant}
\D=a_3b_3a_4b_4(a_3b_4-a_4b_3)\neq 0.
\end{equation}
It is clear that the forms involved are all pairwise non-proportional. 
In this section we will further reduce our counting problem
using the familiar multiplicative
arithmetic function
\[r(n)=\card\{(x,y)\in \ZZ^2\mset x^2+y^2=n\}=
4 \sum_{d\mid n }\chi(d),\]
where $\chi$ is the real non-principal character modulo $4$. It is to
this expression that we will be able to direct the full force of analytic
number theory.

In what follows we will allow the implied constant in any estimate to
depend arbitrarily upon the coefficients of the linear forms
involved. Furthermore, we will henceforth 
reserve $j$ for an arbitrary index from the set $\{1,2,3,4\}$.
Finally, many of our estimates will involve a small parameter
$\ve>0$ and it will ease notation if we also allow the implied constants to
depend on the choice of $\ve$.  We will follow common practice and allow $\ve$ to take
different values at different parts of the argument.

Recall the definitions of $\Sigma, \Sigma'$ from section
\ref{section.torsors} and section
\ref{section.jumpingup} respectively. 
In particular we have $m_jN(\fa^+_j)=O(1)$ whenever $\mm\in
\Sigma$ and $\faa\in\Sigma'$.

\begin{prop}\label{prop:6.1}
For $B\geq 1$, we have
\begin{align*}
N(B)&=\frac 1{\card\TNS(\QQ)_\tors}
\sum_{\substack{\mm\in \Sigma\\ \faa\in \Sigma'}}\mu(\faa)
\sum_{\substack{\ell=1\\ 2\nmid \ell}}^\infty \mu(\ell)
\sum_{\fbb \in \widehat{\fD}^4} \mu(\fbb) 
\sum_{\substack{t \in \fD\\ \gcd(t,\N(\faa))=1\\ \N(\bigcap
    \fb_j)\mid t}}
r\Big(\frac{t}{\N(\bigcap \fb_j)}\Big)
\mathcal{U}\Big(\frac{B}{t}\Big),
\end{align*}
where 
\begin{align*}
\mathcal{U}(T)=
\sum_{\substack{(u,v)\in \ZZ^2\cap \sqrt{T}\mcal{R_\mm}\\
\ell\mid u,v\\
2\nmid \gcd(u,v)\\
m_j\N(\fa^+_j\fb_j) \mid L_j(u,v)
}}\prod_{j=1}^4r\Big(\frac{L_j(u,v)}{m_j\N(\fa^+_j\fb_j) }  \Big)
\end{align*}
and
\begin{equation}
  \label{eq:RR}
\mcal{R_\mm}=\Big\{(u,v)\in \RR^2\mset
0<|u|,|v|\leq 1, ~m_jL_j(u,v)>0\text{ for }j\in\onefour\Big\}.
\end{equation}
\end{prop}
\begin{proof}
We apply proposition~\ref{prop.aftermoebius}.
Let $\mm\in\Sigma$, $\faa\in\Sigma'$ and $\fbb\in\widehat{\mathcal D}^4$.
We wish to express $\card(\torsQ_{\N(\faa)\N(\fbb)\mm}(\QQ)\cap
\mathcal D^3_{\mm,\faa,\fbb,\ell}(B))$ in terms of the function $r$.
But given $(t,u,v)\in\ZZ^3$, the number of elements $R$ in that
intersection such that $(T(R),U(R),V(R))=(t,u,v)$
is $0$ if $(t,u,v)$ does not satisfy the conditions
\[\gcd(t,\N(\faa))=1,\quad
N(\bigcap\fb_j)|t,\quad
\ell|u,v,\quad 2\nmid t\gcd(u,v)
\text{ and }m_j\N(\fa_j^+\fb_j) \mid L_j(u,v)\]
and is equal to
\[r\left(\frac t{\N(\bigcap\fb_j)}\right)
\prod_{j=1}^4r\left(\frac{L_j(u,v)}{m_j\N(\fa_j^+\fb_j)}\right)\]
otherwise.
\end{proof}

Let us set
\begin{equation}
  \label{eq:thed}
  d_j=m_j\N(\fa^+_j)\N(\fb_j), \quad D_j=\begin{cases}
[d_j,\ell], & \mbox{if $j=1$ or $2$},\\
d_j, & \mbox{if $j=3$ or $4$},
\end{cases}
\end{equation}
where $[d_j,\ell]$ is the least common multiple of $d_j,\ell$.
Then $d_j,D_j$ are odd positive integers such that $d_j\mid D_j$. 
We may then write 
\begin{equation}\label{eq:UUT}
\mathcal{U}(T)=
\sum_{\substack{(u,v)\in \sfg_{\ma{D}}\cap \sqrt{T}\mcal{R_\mm}\\
2\nmid \gcd(u,v)\\
}}\prod_{j=1}^4r\Big(\frac{L_j(u,v)}{d_j}  \Big),
\end{equation}
where 
\begin{equation}\lab{eq:lattice}
\sfg_{\ma{D}}=\{(u,v)\in\ZZ^2\mset D_j\mid L_j(u,v)\}.
\end{equation}

Before passing to a detailed analysis of the sum $\UU(T)$ and its
effect on the behaviour of the counting function $N(B)$, we will first
corral together some of the technical tools that
will prove useful to us.

\subsection{Geometric series}
\label{geom-series}

Given a vector $\ma{n}=(n_1,n_2,n_3,n_4)\in \ZZ_{\geq 0}^4$, let
$$
m(\ma{n})=\max_{i\neq j}\{n_i+n_j\}.
$$ 
It will be useful to note that
$m(n_1+\la,\ldots,n_4+\la)=m(\ma{n})+2\la$, for any $\la\in\ZZ$,
whence in particular $m(\ma{n})-2=m(n_1-1,n_2-1,n_3-1,n_4-1)$.  

For
$\ve\in\{-1,+1\}$ we will need to calculate the geometric series
\begin{equation}\lab{30-S0pm}
S_0^{\varepsilon}(z)=\sum_{\ma{n}\in\ZZ_{\geq 0}^4} \varepsilon^{n_1+n_2+n_3+n_4}
z^{m(\ma{n})},
\end{equation}
for $|z|<1$.  To do so we will break up the sum according to the values of
$\min\{n_1,n_2\}$ and $\min\{n_3,n_4\}$.  
Let $S_{0,0}^\ve(z)$ denote the contribution to $S_0^\ve(z)$ from
$\ma{n}$ such that $\min\{n_1,n_2\}=\min\{n_3,n_4\}=0$, and 
let $S_{0,1}^\ve(z)$ denote the corresponding contribution from
$\ma{n}$ such that $\min\{n_1,n_2\}\geq 1$ and $\min\{n_3,n_4\}=0$.
Now it is rather easy to see that
\begin{equation}\lab{31-0,0}
\begin{split}
S_{0,0}^\varepsilon(z)
&=
\Big(\sum_{\min\{n_1,n_2\}=0}(\varepsilon z)^{n_1+n_2}\Big)^2 =
\Big(\frac{1+\varepsilon z}{1-\varepsilon z}\Big)^2.
\end{split}
\end{equation}
since $m(\ma{n})=n_1+n_2+n_3+n_4$ in this setting. Next we claim that
\begin{equation}\lab{31-1,0}
S_{0,1}^\varepsilon(z)= \frac{ (1+2\varepsilon +2z+\varepsilon z^2)z^2}{(1-\varepsilon
z)^2(1-\varepsilon z^2)}.
\end{equation}
To see this we note that
\begin{align*}
S_{0,1}^\varepsilon(z)&=\Big(2\sum_{n_1,n_2,n_3\geq 1, n_4=0}+ \sum_{n_1,n_2\geq 1,
n_3=n_4=0} \Big) \varepsilon^{n_1+n_2+n_3+n_4}z^{m(\ma{n})}.
\end{align*}
Now the second summation is clearly  $\big(\sum_{a\geq
1}(\varepsilon z)^a\big)^2=z^2/(1-\varepsilon z)^2$. Similarly, the first summation is
\begin{align*}
&= 2
\sum_{n_1,n_2,n_3\geq 1} (\varepsilon z)^{n_1+n_2+n_3}z^{-\min\{n_j\}}\\
&=2\sum_{k\geq 1}z^{-k} \sum_{\min\{n_j\}=k} (\varepsilon z)^{n_1+n_2+n_3}\\
&=2\sum_{k\geq 1}z^{-k} \Big(\sum_{n_1,n_2,n_3\geq k}
(\varepsilon z)^{n_1+n_2+n_3}- \sum_{n_1,n_2,n_3,\geq k+1}
(\varepsilon z)^{n_1+n_2+n_3}\Big)\\  
&=2\sum_{k\geq 1}z^{-k}
\Big(\frac{(\varepsilon z)^{3k}}{(1-\ve z)^3}-\frac{(\varepsilon
  z)^{3k+3}}{(1-\ve z)^3}
\Big)
=  2\varepsilon\frac{(1+\varepsilon z+z^2)z^2}{(1-\varepsilon
z)^2(1-\varepsilon z^2)}.
\end{align*}
Combining these two equalities completes the proof of \eqref{31-1,0}.
We may now establish the following result.

\begin{lemma}\lab{31-s0-}
Let $|z|<1$.  Then we have
$$
S_0^-(z)=\frac{(1-z)^2}{(1+z)^2(1+z^2)}
$$
and
$$
S_0^+(z)=\frac{1+2z+6z^2+2z^3+z^4}{(1-z)^4(1+z)^2}.
$$
\end{lemma}

\begin{proof}
The proof of lemma \ref{31-s0-} is based on the simple observation that
$$
S_0^\varepsilon(z)=S_{0,0}^\varepsilon(z)+2 S_{0,1}^\varepsilon(z) + z^2
S_0^\varepsilon(z),
$$
from which it follows that
$$
S_0^\varepsilon(z)=(1-z^{2})^{-1}\big(S_{0,0}^\varepsilon(z)+2
S_{0,1}^\varepsilon(z)\big).
$$
We complete the proof of the lemma  by inserting
\eqref{31-0,0} and \eqref{31-1,0} into this equality.
\end{proof}

\subsection{Geometry of numbers}

It will be useful to collect together some elementary facts concerning the
set $\sfg_{\ma{D}}$ that was defined in \eqref{eq:lattice}. 
For the moment we allow $\ma{D}\in \ZZp^4$ to be arbitrary. 
It is clear that $\sfg_{\ma{D}}$ 
defines a sublattice of $\ZZ^2$ of rank $2$, since it is closed
under addition and contains the vector $D_1D_2D_3D_4(u,v)$ for any $(u,v)\in \ZZ^2$.

Let us write 
\begin{equation}
  \label{eq:def-rho}
\rho(\ma{D})=\det \sfg_{\ma{D}},
\end{equation}
for the determinant.
It follows from the Chinese remainder theorem that there is a multiplicativity property
$$
\rho(g_1 h_1,\ldots,g_4 h_4)=\rho(g_1,\ldots,g_4)\rho(h_1,\ldots,h_4),
$$
whenever $\gcd(g_1 g_2 g_3 g_4,h_1 h_2 h_3 h_4)=1$.
Recall the definition \eqref{eq:resultant} of $\Delta$.  Then
\cite[Eqn. (3.12)]{heathbrown:2003} shows that 
\begin{equation}
  \label{eq:tenn''}
\rho(p^{e_1},\ldots,p^{e_4})=p^{\max_{i< j}\{e_{i}+e_{j}\}},
\end{equation}
for any prime $p\nmid \Delta$.  Likewise, 
when $p\mid \D$ one has 
\begin{equation}
  \label{eq:tenn-}
\rho(p^{e_1},\ldots,p^{e_4})\asymp p^{\max_{i< j}\{e_{i}+e_{j}\}},
\end{equation}
where the symbol $\asymp$ indicates that the two
quantities involved have the same order of magnitude. 
It follows from the properties that we have recorded here that
\begin{equation}
  \label{eq:251.1}
  \rho(\ma{D})\asymp [D_1D_2,D_1D_3,D_1D_4,D_2D_3,D_2D_4,D_3D_4].
\end{equation}

We can also say something about the size of the smallest successive
minimum, $s_1$ say,  of
$\sfg_{\ma{D}}$.  Thus we have
\begin{equation}\lab{sm}
s_1 \geq \min\{D_1,D_2\}.
\end{equation}
For this we note that
$\sfg_{\ma{D}}\subseteq \sfl=\{(u,v) \in \ZZ^2\mset ~D_1 \mid u, ~D_2 \mid v\}$.
Now $\sfl\subseteq \ZZ^2$ is a sublattice of rank $2$, with smallest
successive minimum $\min\{D_1,D_2\}$. 
The desired inequality is now obvious.
\section{Estimating \texorpdfstring{$\UU(T)$}{U(T)}: an upper bound}
\label{section.bound}

Our goal in this section is to provide an upper bound for 
$\UU (T)$, which is uniform in the
various parameters.
This will allow us to reduce the range of
summation for the various parameters appearing in our expression for $N(B)$.
Our main tool will be previous work of the first two
authors \cite{bretechebrowning:sums}, 
which is concerned with the average order of
arithmetic functions ranging over the values taken by binary forms.

Throughout this section we continue to adhere to the convention that
all of our implied constants are allowed to depend upon the
coefficients of the forms $L_j$. Recall the expression for $\UU(T)$
given in \eqref{eq:UUT}, with $d_j, D_j$ given by \eqref{eq:thed}.
With these in mind we have the following result.

\begin{lemma}\label{lem:UU-upper}
Let $\ve>0$ and let $T\geq 1$.
Then we have
$$
\UU(T) \ll (d\ell)^{\ve} \Big(
\frac{T}{[D_1D_2,\ldots, D_3D_4]}+\frac{T^{1/2+\ve}}{\ell}\Big),
$$
where  $d=d_1d_2d_3d_4$.
\end{lemma}

\begin{proof}
Since we are only concerned with providing an upper bound for
$\UU(T)$, we may drop any of the conditions in the
summation over $(u,v)$ that we care to choose.  Thus it follows that
$$
 \UU(T)\leq 
\sum_{(u,v)\in \sfg_{\ma{D}}\cap (0,\sqrt{T}]^2}
\prod_{j=1}^4
r\Big(\frac{|L_j(u,v)|}{d_{j}}\Big),
$$
where $\sfg_{\ma{D}}$ is the lattice defined in \eqref{eq:lattice}.

Let $\ma{e}_1,\ma{e}_2$ be a minimal basis for
$\sfg_{\ma{D}}$.
This is constructed by taking 
$\ma{e}_1 \in \sfg_{\ma{D}}$ to be any non-zero vector for which
$|\ma{e}_1|$ is least, and then choosing $\ma{e}_2 \in \sfg_{\ma{D}}$
to be any vector not proportional to $\ma{e}_1$, for which $|\ma{e}_2|$ is
least. The  successive minima of $\sfg_{\ma{D}}$ are
the numbers $s_i=|\ma{e}_i|,$ for $i=1,2$. They satisfy the inequalities
\begin{equation}\lab{rog6}
\ell \leq s_1\leq s_2, \quad s_1s_2 \ll \rho(\ma{D})  \leq s_1s_2,
\end{equation}
where $\rho$ is defined in \eqref{eq:def-rho} and the lower bound for
$s_1$ follows from \eqref{sm} and the definition \eqref{eq:thed} of $D_1,D_2$.
Write $M_j(X,Y)$ for the linear form obtained from $d_j^{-1}L_j(U,V)$
via the change of variables $(U,V)\mapsto X \ma{e}_1+Y\ma{e}_2$.
Each $M_j$ has integer coefficients of size $O(\rho(\ma{D}))$.
Furthermore, it follows from work of Davenport \cite[lemma 5]{davenport:cubic}
that $x\ll  \max\{|u|,|v|\}/s_1$ and $y\ll  \max\{|u|,|v|\}/s_2$ whenever one writes
$(u,v)\in\sfg_{\ma{D}}$ as
$
(u,v)=x \ma{e}_1+y\ma{e}_2,
$
with $x,y\in \ZZ$.  Let 
$$
T_1=s_1^{-1}\sqrt{T}, \quad T_2=s_2^{-1}\sqrt{T},
$$
so that in particular $T_1\geq T_2>0$. Then we may deduce that
$$
\UU(T)
\leq
\sum_{x \ll T_1, y\ll T_2}\prod_{j=1}^4r(|M_j(x,y)|).
$$

Suppose that $M_j(X,Y)=a_{j1}X+a_{j2}Y$, 
with integer coefficients $a_{ji}=O(\rho(\ma{D}))$. 
We proceed to introduce a multiplicative function $r_1(n)$, via
$$
r_1(p^\nu)=\left\{
\begin{array}{ll}
1+\chi(p), & \mbox{$\nu=1$ and $p \nmid 6d\ell \prod a_{ji}$,}\\
(1+\nu)^4, & \mbox{otherwise},
\end{array}
\right.
$$
where $d=d_1d_2d_3d_4$.
Then $r(n_1)r(n_2)r(n_3)r(n_4)\leq 2^8 r_1(n_1n_2n_3n_4),$
and it is not hard to see that $r_1$ belongs to the class of 
non-negative arithmetic functions considered previously by the first
two authors \cite{bretechebrowning:sums}. An 
application of \cite[corollary 1]{bretechebrowning:sums} now reveals that 
\begin{align*}
\UU(T)
\ll
(d\ell)^\ve (T_1T_2 + T_1^{1+\ve})
&\ll (d \ell)^\ve   \Big(\frac{T}{s_1s_2} + \frac{T^{1/2+\ve}}{s_1}\Big),
\end{align*}
for any $\ve>0$. 
Combining \eqref{rog6} with \eqref{eq:251.1} we therefore conclude the
proof of the lemma.
\end{proof}

The main purpose of lemma \ref{lem:UU-upper} is to reduce the range of
summation of the various parameters appearing in proposition \ref{prop:6.1}.
Let us write $E_0(B)$ for the overall contribution to the summation from
values of $\fb_j, \ell$ such that
\begin{equation}
  \label{eq:bnf}
\max \N(\fb_j)>\log (B)^{D} \quad \mbox{or} \quad \ell> \log (B)^L,
\end{equation}
for parameters $D,L>0$ to be selected in due course. 
We will denote by $N_1(B)$ the remaining contribution, so that 
\begin{equation}
  \label{eq:N1E0}
  N(B)=N_1(B)+E_0(B).
\end{equation}
Henceforth, the
implied constants in our
estimates will be allowed to depend on $D$ and $L$, in addition to 
the coefficients of the linear forms $L_j$.
We proceed to establish the following result.

\begin{lemma}\label{lem:M(B)}
We have $E_0(B)\ll B\log (B)^{1-\min\{D/4,L/2\}+\ve}$, for any $\ve>0$.
\end{lemma}

\begin{proof}
We begin observing that  $\UU(B/t)=0$ in $E_0(B)$, unless 
$D_j\leq \sqrt{B/t}$,
in the notation of \eqref{eq:thed}. But then it follows that we must have
$$
t \leq \frac{B}{\sqrt{D_1D_2D_3D_4}}\leq \frac{B
  \sqrt{\gcd(\N(\fb_1),\ell)\gcd(\N(\fb_2),\ell)}}{\ell\sqrt{\N(\fb_1)\cdots
    \N(\fb_4)}} =B_0,
$$
say, in the summation over $t$. 
Here we have used the fact that $m_jN(\fa^+_j)=O(1)$ whenever $\mm\in 
\Sigma$ and $\faa\in\Sigma'$.

We now apply lemma  \ref{lem:UU-upper} to bound $\UU(B/t)$, giving  
\begin{align*}
E_0(B)\ll
\sum_{\substack{\mm\in \Sigma\\ \faa\in \Sigma'}}
&\sum_{\ell} \ell^\ve
\sum_{\fb_1,\ldots,\fb_4} (\N(\fb_1)\cdots \N(\fb_4))^\ve \\
&\times \sum_{\substack{t\leq B_0\\ \N(\bigcap
    \fb_j)\mid t}}
r\Big(\frac{t}{\N(\bigcap \fb_j)}\Big)
\Big(
\frac{B}{t[D_1D_2,\ldots, D_3D_4]}+\frac{B^{1/2+\ve}}{t^{1/2+\ve}\ell}\Big),
\end{align*}
for any $\ve>0$, where the summations over $\ell$ and $\fb_j$ are
subject to \eqref{eq:bnf}.  In view of the elementary estimates
\begin{equation}
  \label{eq:familiar}
\sum_{n\leq x} \frac{r(n)}{n^\theta} \ll \begin{cases}
\log (2x) & \mbox{if $\theta\geq 1$,}\\
x^{1-\theta} & \mbox{if $0\leq \theta<1$,}
\end{cases}
\end{equation}
we easily conclude that
\begin{align*}
E_0(B)\ll
\sum_{\substack{\mm\in \Sigma\\ \faa\in \Sigma'}}
&\sum_{\ell} \ell^\ve
\sum_{\fb_1,\ldots,\fb_4} (\N(\fb_1)\cdots \N(\fb_4))^\ve \\
&\times  \frac{1}{\N(\bigcap \fb_j)}\Big(
\frac{B \log (B)}{[D_1D_2,\ldots,
  D_3D_4]}+\frac{B^{1/2+\ve}B_0^{1/2-\ve}}{\ell}\Big).
\end{align*}
The second term in the inner bracket is 
$$
\frac{B^{1/2+\ve}B_0^{1/2-\ve}}{\ell}\ll 
B \cdot \frac{\gcd(\N(\fb_1),\ell)^{1/4}\gcd(\N(\fb_2),\ell)^{1/4}} 
{\ell^{3/2-\ve}\N(\fb_1)^{1/4-\ve}\cdots \N(\fb_4)^{1/4-\ve}}.
$$
Similarly, a rapid consultation with \eqref{eq:thed} reveals that the first term is 
\begin{align*}
  \frac{B \log (B)}{[D_1D_2,\ldots,
  D_3D_4]}&\ll  
  \frac{B \log (B)}{(D_1D_2)^{3/4}(D_3D_4)^{1/4}}\\
&\ll  
B\log (B)\cdot \frac{\gcd(\N(\fb_1),\ell)^{1/4}\gcd(\N(\fb_2),\ell)^{1/4}} 
{\ell^{3/2}\N(\fb_1)^{1/4}\cdots
    \N(\fb_4)^{1/4}}.
\end{align*}
Bringing these estimates together we may now conclude that 
\begin{align*}
E_0(B)\ll B\log (B)
\sum_{\ell} 
\sum_{\fb_1,\ldots,\fb_4} 
\frac{1}{\N(\bigcap \fb_j)}\cdot \frac{\gcd(\N(\fb_1),\ell)^{1/4}\gcd(\N(\fb_2),\ell)^{1/4}} 
{\ell^{3/2-\ve}\N(\fb_1)^{1/4-\ve}\cdots
    \N(\fb_4)^{1/4-\ve}},
\end{align*}
where the sums are over $\ell \in \ZZp$ and $\fb_1,\ldots,\fb_4 \subseteq
\widehat{\mathfrak{D}}$ such that \eqref{eq:bnf} holds.

For fixed $\ell\in \ZZp$
and $\ve>0$ we proceed to estimate the sum
$$
S_{\ell}(T)=\sum_{\substack{\fb_1,\ldots,\fb_4\subseteq \ZZ[i]\\ \max
    \N(\fb_j)\geq T}} 
\frac{\gcd(\N(\fb_1),\ell)^{1/4}\gcd(\N(\fb_2),\ell)^{1/4}} 
{\N(\bigcap \fb_j)\N(\fb_1)^{1/4-\ve}\cdots
    \N(\fb_4)^{1/4-\ve}}.
$$
This is readily achieved via Rankin's trick and the observation
that $\N(\mathfrak{a})\mid \N(\mathfrak{a}\cap\mathfrak{b})$ for
any $\mathfrak{a}, \mathfrak{b}\subseteq \ZZ[i]$.  Thus it follows that 
$\N(\bigcap \fb_j)\geq [\N(\fb_1),\ldots,\N(\fb_4)]$, whence
\begin{align*}
S_{\ell}(T)
&\leq \frac{1}{T^\delta}\sum_{\fb_1,\ldots,\fb_4\subseteq\ZZ[i]}
\frac{\gcd(\N(\fb_1),\ell)^{1/4}\gcd(\N(\fb_2),\ell)^{1/4}} 
{[\N(\fb_1),\ldots,\N(\fb_4)]^{1-\delta}\N(\fb_1)^{1/4-\ve}\cdots
    \N(\fb_4)^{1/4-\ve}}\\
&\ll \frac{1}{T^\delta}\sum_{b_1,\ldots,b_4=1}^\infty 
\frac{\gcd(b_1,\ell)^{1/4}\gcd(b_2,\ell)^{1/4}} 
{[b_1,\ldots,b_4]^{1-\delta}b_1^{1/4-\ve}\cdots
    b_4^{1/4-\ve}}\\
&\ll \frac{1}{T^\delta}\sum_{[k_1,k_2]\mid \ell} (k_1k_2)^\ve
\sum_{b_1,\ldots,b_4=1}^\infty 
\frac{1}{[b_1,\ldots,b_4]^{1-\delta}b_1^{1/4-\ve}\cdots
    b_4^{1/4-\ve}}\\
&\ll_{\delta} \ell^\ve T^{-\delta},
\end{align*}
provided that $\delta<1/4,$  
as can be seen by considering the corresponding Euler product. 

Armed with this we see that the overall contribution to the above
estimate for $E_0(B)$ arising from 
$\ell, \fb_1,\ldots,\fb_4 $ for which $\ell>\log (B)^L$
is 
\begin{align*}
&\ll B\log (B) \sum_{\ell> \log (B)^L} \ell^{-3/2+\ve} 
S_{\ell}(1) \ll B\log (B)^{1-L/2+\ve},
\end{align*}
which is satisfactory.  In a similar fashion we 
see that the overall contribution to $E_0(B)$ arising from 
$\ell, \fb_1,\ldots,\fb_4 $ for which $\max \N(\fb_j)>\log (B)^D$ is 
\begin{align*}
&\ll 
B\log (B)
\sum_{\ell} \ell^{-3/2+\ve}
S_{\ell}(\log (B)^D) 
\ll B\log (B)^{1-D/4+\ve},
\end{align*}
which is also satisfactory. The statement of lemma \ref{lem:M(B)} is
now obvious. 
\end{proof}
\section{Estimating \texorpdfstring{$\UU(T)$}{U(T)}: an asymptotic formula}
\label{section.estimate}

In view of our work in the previous section it remains to estimate 
$N_1(B)$, which we have defined 
as the contribution to $N(B)$ from values of $\fb_j, \ell$ for which
\eqref{eq:bnf} fails.  Thus 
\begin{align*}
N_1(B)&=\frac 1{\card\TNS(\QQ)_\tors}
\sum_{\substack{\mm\in \Sigma\\ \faa\in \Sigma'}}\mu(\faa)
\hspace{-0.2cm}
\sum_{\substack{\ell \leq \log (B)^L \\ 2\nmid \ell}} 
\hspace{-0.2cm}
\mu(\ell)
\hspace{-0.5cm}
\sum_{\substack{\fb_1,\ldots,\fb_4 \in \widehat{\fD}\\
\N(\fb_j)\leq \log (B)^{D}}} 
\hspace{-0.2cm}
\prod_{j=1}^4 \mu(\fb_j) 
\hspace{-0.3cm}
\sum_{\substack{t \in \fD\cap[1,B]\\ \gcd(t,\N(\faa))=1\\ \N(\bigcap
    \fb_j)\mid t}}
\hspace{-0.3cm}
r\Big(\frac{t}{\N(\bigcap \fb_j)}\Big)
\mathcal{U}\Big(\frac{B}{t}\Big).
\end{align*}
Here we have inserted the condition $t\leq B$ in the summation over
$t$, since the innermost summand is visibly zero otherwise. 
Whereas the previous section was primarily concerned with a uniform
upper bound for the sum $\UU(T)$ defined in \eqref{eq:UUT}, our work in
the present section will revolve around a uniform asymptotic formula
for $\UU(T)$. The error term that arises in our analysis will involve
the real number
\begin{equation}\label{defeta} 
\eta=1-\frac{1+\log(\log (2))}{\log (2)},
\end{equation}
which has numerical value $0.086071\ldots$.

Before revealing our result for $\UU(T)$, we must first 
introduce some notation for certain local densities that emerge in the
asymptotic formula. 
In fact estimating $\UU(T)$ boils down to counting integer points on the affine variety
\begin{equation}
  \label{eq:torsor}
L_j(U,V)=d_j(S_j^2+T_j^2), \quad (1\leq j\leq 4),
\end{equation}
in $\Aff_{\QQ}^{10}$, with $U,V$ restricted to lie in a lattice
depending on $\mathbf{D}$. 
Thus the expected leading constant admits an interpretation as a product of local
densities. 
Given a prime $p>2$ and 
$\bd,\mathbf{D}$ as in~\eqref{eq:thed}, let 
$$
N_{\bd,\mathbf{D}}(p^n)=\card\Big\{(u,v,\ma{s},\ma{t})\in
(\ZZ/p^n\ZZ)^{10}\bmset\begin{array}{l}
L_j(u,v)\equiv d_j(s_j^2+t_j^2) \bmod{p^n}\\
D_j \mid L_j(u,v)
\end{array}
\Big\}.
$$
The $p$-adic density on  \eqref{eq:torsor}
is defined to be  
\begin{equation}
  \label{eq:def-sigp}
\omega_{\bd,\mathbf{D}}(p)=\lim_{n\rightarrow
  \infty}p^{-6n-\lambda_1-\cdots-\lambda_4}N_{\bd,\mathbf{D}}(p^n), 
\end{equation}
when $p>2, $ where
\begin{equation}\label{eq:lm}
\mal=\big(v_p(d_{1}),\ldots,v_p(d_{4})\big),
\quad \mmu=\big(v_p(D_{1}),\ldots,v_p(D_{4})\big).
\end{equation}
When $\ma{d},\ma{D}$ are as in \eqref{eq:thed} and $p>2$, we will set 
\begin{equation}
  \label{eq:sig>2}
\sigma_p(\ma{d},\ma{D})=\omega_{\bd,\mathbf{D}}(p).
\end{equation}
Turning to the case $p=2$, we define 
\begin{equation}
  \label{eq:sig=2}
\sigma_2(\ma{d},\ma{D})=
\lim_{n\rightarrow \infty}2^{-6n}N_{\bd,\mathbf{D}}(2^n)
\end{equation}
where
$$
N_{\bd,\mathbf{D}}(2^n) =
\card\Big\{(u,v,\ma{s},\ma{t})\in
(\ZZ/2^n\ZZ)^{10}\bmset \begin{array}{l}
L_j(u,v)\equiv d_j
(s_j^2+t_j^2) \bmod{2^n}\\
2\nmid \gcd(u,v)
\end{array}
\Big\}.
$$
Finally, we let $\omega_{\mcal{R}_\mm}(\infty)$ denote the usual 
archimedean density of solutions  to the system of equations
\eqref{eq:torsor}, with  $(u,v,\ma{s},\ma{t})\in\mcal{R}_\mm\times
\RR^{8}$ and where $\mcal{R}_\mm$  is defined in \eqref{eq:RR}. 
We are now ready to record our main estimate for $\UU(T)$.

\begin{lemma}\label{lem:Tfinal}
Recall the definitions of $\ma{d},\ma{D}$ from \eqref{eq:thed}.
Then for any $\ve>0$ and $T>1$ we have 
$$
\UU(T)=c_{\bd,\ma{D},\mcal{R}_\mm} T
+O\Big(\frac{(d_1d_2d_3d_4 \ell)^\ve T}{\log (T)^{
    \eta-\varepsilon}}\Big), 
$$
where
\begin{equation}
  \label{eq:cfinal}
c_{\bd,\ma{D},\mcal{R}_\mm}=\omega_{\mcal{R}_\mm}(\infty)
\prod_{p\in\primes}\sigma_p(\ma{d},\ma{D}).
\end{equation}
\end{lemma}

\begin{proof}
Our primary tool in estimating $\UU(T)$ asymptotically is the subject
of allied work of the first two authors \cite{bretechebrowning:4linear}. We begin by bringing our 
expression for $\UU(T)$ into
a form that can be tackled by the main
results there. According to \eqref{eq:resultant} 
we may assume that the binary linear forms $L_j$ are 
pairwise non-proportional and primitive. Furthermore, it is clear that the region
$\mcal{R}_\mm\subset \RR^2$ 
defined in \eqref{eq:RR} is open, bounded and convex, with a piecewise continuously
differentiable boundary such that $m_jL_j(u,v)>0$ for each $(u,v)\in \mcal{R}_\mm$.

A key step in applying the work of \cite{bretechebrowning:4linear} consists in
checking that the ``normalisation hypothesis''  
\textsf{NH}$_2(\ma{d})$
is satisfied  in the present context.
In fact it is easy to see that 
$L_j,\mcal{R}_\mm$ will satisfy \textsf{NH}$_2(\ma{d})$ 
provided that   
$$
L_1(U,V)\equiv d_1U \pmod{4}, \quad L_2(U,V)\equiv V \pmod{4}.
$$
The second congruence is automatic since $L_2(U,V)=V$. 
Recalling that $L_1(U,V)=U$, we therefore conclude that 
\textsf{NH}$_2(\ma{d})$  holds if $d_1\equiv 1
\bmod{4}$. Alternatively, if $d_1\equiv 3
\bmod{4}$, we make the unimodular change of variables
$(U,V)\mapsto (-U,V)$ to place ourselves in the setting of 
\textsf{NH}$_2(\ma{d})$.
We leave the reader to check that this ultimately leads to an
identical estimate in the ensuing argument. Thus, for the purposes of
our exposition here, we may freely assume that 
$L_j,\mcal{R}_\mm$ satisfy \textsf{NH}$_2(\ma{d})$ in $\UU(T)$.

We proceed by writing
\begin{equation}
  \label{eq:0501}
\UU(T)=U_1(T)+U_2(T)+U_3(T),
\end{equation}
where $U_1(T)$ denotes the contribution to $\UU(T)$ from $(u,v)$ such that
$2\nmid uv$, 
$U_2(T)$ denotes the contribution from $(u,v)$ such that
$2\nmid u$ and $2\mid v$, and finally  
$U_3(T)$ is the contribution from $(u,v)$ such that
$2\mid u$ and $2\nmid v$.  
Beginning with an estimate for $U_1(T)$,
we observe that 
$$
U_1(T)=S_1(\sqrt{T},\ma{d},\sfg_{\ma{D}}),
$$
in the notation of \cite[eq. (1.9)]{bretechebrowning:4linear}, with $\ma{d},\ma{D}$
given by \eqref{eq:thed}. An application of \cite[theorems 3 and
4]{bretechebrowning:4linear} with $(j,k)=(1,2)$ therefore reveals that 
there exists a constant $c_1$ such that
$$
U_1(T)
=c_1T +
O\Big(\frac{(d\ell)^{\ve} T}{\log(T)^{ \eta-\varepsilon}}\Big),
$$
where $d=d_1d_2d_3d_4$.
The value of the constant is given by 
$$
  c_1=\omega_{\mcal{R}_\mm}(\infty)\omega_{1,\ma{d}}(2)\prod_{p>2}
\omega_{\bd,\mathbf{D}}(p).
$$
Here $\omega_{\bd,\mathbf{D}}(p)$ is given by 
\eqref{eq:def-sigp}
and $\omega_{\mcal{R}_\mm}(\infty)$ is defined prior to the
statement of the lemma. 
Finally, if 
$$
N_{i,\bd}'(2^n)=\card\Big\{(u,v,\ma{s},\ma{t})\in
(\ZZ/2^n\ZZ)^{10}\bmset\begin{array}{l}
L_j(u,v)\equiv d_j(s_j^2+t_j^2) \bmod{2^n}\\
u \equiv 1 \bmod{4}, ~v\equiv i \bmod{2}
\end{array}
\Big\},
$$
for any $i \in \{0,1\}$, then
the corresponding $2$-adic density is given by 
\begin{equation*}
\omega_{i,\bd}(2)=\lim_{n\rightarrow \infty}2^{-6n}N_{i,\bd}'(2^n).
\end{equation*}
Note that the notation introduced in \cite{bretechebrowning:4linear} involves an
additional subscript in $\omega_{i,\ma{d}}(2)$ whose presence indicates
which of the various normalisation hypotheses the $L_j,\mcal{R}_\mm$ are
assumed to satisfy. Since we have placed ourselves in the context of 
\textsf{NH}$_2(\ma{d})$ in each case, we have found it reasonable to
suppress mentioning this here.

Let us now shift to a consideration of the sum $U_2(T)$ in
\eqref{eq:0501}, for which one finds that
$$
U_2(T)=S_0(\sqrt{T},\ma{d},\sfg_{\ma{D}}).
$$
Applying \cite[theorems 3 and
4]{bretechebrowning:4linear} with $(j,k)=(0,2)$ therefore yields
$$
U_2(T)
=c_2T +
O\Big(\frac{(d\ell)^{\ve} T}{\log (T)^{ \eta-\varepsilon}}\Big),
$$
where now 
$$
c_2=\omega_{\mcal{R}_\mm}(\infty)\omega_{0,\ma{d}}(2)\prod_{p>2}
\omega_{\bd,\mathbf{D}}(p),
$$
with notation as above.

Finally we turn to the sum $U_3(T)$ in \eqref{eq:0501}. 
Making the unimodular change of variables
$(U,V)\mapsto (V,U)$, one now sees that 
$$
U_3(T)=S_0(\sqrt{T};\ma{d},\sfg_{\ma{D}}^{\flat}),
$$
where now the underlying region is $\mcal{R}_\mm^\flat=
\{(u,v) \in \RR^2\mset (v, u)\in \mcal{R}_\mm\}$
and $\sfg_{\ma{D}}^\flat$ is defined as for $\sfg_{\ma{D}}$, but with 
the linear forms $L_j(U,V)$ replaced by $L_j(V,U)$.
Thus an application of \cite[theorems 3 and
4]{bretechebrowning:4linear} with $(j,k)=(0,2)$ produces 
$$
U_3(T)
=c_3T +
O\Big(\frac{(d\ell)^{\ve} T}{\log (T)^{ \eta-\varepsilon}}\Big),
$$
with 
$$
c_3=\omega_{\mcal{R}_\mm^\flat}(\infty)\omega_{0,\ma{d}}^\flat(2)\prod_{p>2}
\omega_{\bd,\mathbf{D}}^\flat(p)=
\omega_{\mcal{R}_\mm}(\infty)\omega_{0,\ma{d}}^{\flat}(2)\prod_{p>2}
\omega_{\bd,\mathbf{D}}(p),
$$
where the superscripts $\flat$ indicate that the local densities 
are taken with respect to the linear forms $L_j(V,U)$. 

We are now ready to bring together our various estimates for 
$U_1(T), U_2(T)$ and $U_3(T)$
in \eqref{eq:0501}. This leads to the asymptotic formula in the
statement of the lemma, with leading constant
$$
c_{\bd,\ma{D},\mcal{R}_\mm}=
\omega_{\mcal{R}_\mm}(\infty)\big(\omega_{1,\ma{d}}(2)+\omega_{0,\ma{d}}(2)+
\omega_{0,\ma{d}}^{\flat}(2)\big) \prod_{p>2} \omega_{\bd,\mathbf{D}}(p). 
$$
The statement of the lemma easily follows with recourse to the
definitions 
\eqref{eq:sig>2}, \eqref{eq:sig=2} of the local densities 
$\sigma_p(\ma{d},\ma{D})$.
\end{proof}

We will need to consider the effect of the error term  
in lemma \ref{lem:Tfinal} on the quantity $N_1(B)$ that was described
at the start of the section.  Accordingly, let us write
\begin{equation}
  \label{eq:vas}
 N_1(B)=N_2(B)+E_1(B),
\end{equation}
where $N_2(B)$ denotes the overall contribution from the main term in 
lemma \ref{lem:Tfinal} and $E_1(B)$ denotes the contribution from the
error term.  

\begin{lemma}\label{lem:M(B)e1}
We have $E_1(B)\ll B\log (B)^{1+L-\eta+\ve}$, for any $\ve>0$.
\end{lemma}

\begin{proof}
Inserting the error term in lemma \ref{lem:Tfinal} into our expression
for $N_1(B)$, we obtain
\begin{align*}
E_1(B)
&\ll
B \log (B)^\ve
\sum_{\substack{\ell \leq \log (B)^L}} 
\sum_{\substack{\fb_1,\ldots,\fb_4 \in \widehat{\fD}\\
\N(\fb_j)\leq \log (B)^{D}}} 
\sum_{\substack{t\leq B \\ \N(\bigcap
    \fb_j)\mid t}}
r\Big(\frac{t}{\N(\bigcap \fb_j)}\Big)\cdot \frac{1}{t\log (2B/t)^{
    \eta}}\\
&\ll
B \log (B)^{L+\ve}
\sum_{\substack{\fb_1,\ldots,\fb_4 \in \widehat{\fD}\\
\N(\fb_j)\leq \log (B)^{D}}} \frac{1}{\N(\bigcap \fb_j)}
\sum_{\substack{t\leq B_1}} \frac{r(t)}{t\log (2B_1/t)^{\eta}}, 
\end{align*}
where we have written $B_1=B/\N(\bigcap \fb_j)$, for ease of notation. 
Combining the familiar \eqref{eq:familiar} with partial summation, we
therefore conclude that
\begin{align*}
E_1(B)
&\ll
B \log (B)^{1+L-\eta+\ve}
\sum_{\substack{\fb_1,\ldots,\fb_4 \in \widehat{\fD}\\
\N(\fb_j)\leq \log (B)^{D}}} \frac{1}{\N(\bigcap \fb_j)}\\
&\ll
B \log (B)^{1+L-\eta+\ve}
\sum_{b_1,\ldots,b_4=1}^\infty
\frac{1}{[b_1,\ldots,b_4](b_1b_2b_3b_4)^\ve}\\
&\ll
B \log (B)^{1+L-\eta+\ve}.
\end{align*}
This concludes the proof of the lemma. 
\end{proof}

To be useful we will also need a uniform upper bound for the constant 
\eqref{eq:cfinal} appearing in lemma \ref{lem:Tfinal}. 
This is achieved in the following result.

\begin{lemma}\label{lem:c-upper}
Let $\ve>0$. Then we have
$$
c_{\bd,\ma{D},\mcal{R}_\mm}\ll \frac{(D_1D_2D_3D_4)^\ve}{[D_1D_2,\ldots, D_3D_4]},
$$
where $\bd,\ma{D}$ are given by \eqref{eq:thed}.
\end{lemma}

\begin{proof}
Now it follows from \cite[theorem 4]{bretechebrowning:4linear} that
$\omega_{\mcal{R}_\mm}(\infty)=\pi^4 \Vol(\mcal{R}_\mm)\ll 1$. Similarly, it
is easy to see that 
$\sigma_2(\ma{d},\ma{D})\leq 2^4$, since for any $A\in \ZZ$ there are 
at most $2^{n+1}$  solutions of the congruence $s^2+t^2 \equiv A \bmod{2^n}$ by
\cite[eq. (2.5)]{bretechebrowning:4linear}.
Thus we have
$$
c_{\bd,\ma{D},\mcal{R}_\mm}\ll \prod_{p>2}|\sigma_p(\ma{d},\ma{D})|,
$$
where $\sigma_p(\ma{d},\ma{D})$ is given by \eqref{eq:sig>2}. 
Assume that $p>2$.  A further application of \cite[theorem~4]{bretechebrowning:4linear}
now yields 
$$
\sigma_p(\ma{d},\ma{D})=
\Big(1-\frac{\chi(p)}{p}\Big)^4
\sum_{\nu_1,\ldots,\nu_4=0}^{\infty}
\frac{\chi(p)^{{\nu_1}+{\nu_2}+{\nu_3}+{\nu_4}}}{\rho(p^{\max\{\mu_1,\lambda_1+\nu_1\}},
\ldots, p^{\max\{\mu_4,\lambda_4+\nu_4\}})},
$$
where $\rho$ is the determinant given in \eqref{eq:def-rho} and 
$\mal, \mmu$ are given by \eqref{eq:lm}. Using the multiplicativity of
$\rho$ we may clearly write 
$$
\prod_{p>2}|\sigma_p(\ma{d},\ma{D})|=\frac{1}{\rho(\ma{D})}
\prod_{p>2}|\sigma_p'(\ma{d},\ma{D})|,
$$
where now
$$
\sigma_p'(\ma{d},\ma{D})=
\Big(1-\frac{\chi(p)}{p}\Big)^4
\sum_{\nu_1,\ldots,\nu_4=0}^{\infty}
\frac{\chi(p)^{{\nu_1}+{\nu_2}+{\nu_3}+{\nu_4}}\rho(p^{\mu_1},
\ldots, p^{\mu_4})}{\rho(p^{\max\{\mu_1,\lambda_1+\nu_1\}},
\ldots, p^{\max\{\mu_4,\lambda_4+\nu_4\}})}.
$$
In view of \eqref{eq:251.1}, it will suffice to
show that 
\begin{equation}
  \label{eq:suffice}
\prod_{p>2}|\sigma_p'(\ma{d},\ma{D})|\ll (D_1D_2D_3D_4)^\ve,
\end{equation}
in order to complete the proof of the lemma.

Recall the definition \eqref{eq:resultant} of  $\D$ and write $D=D_1D_2D_3D_4$. Then
for $p\nmid \D D$ 
it follows from \eqref{eq:tenn''} that
$$
\sigma_p'(\ma{d},\ma{D})=
\Big(1-\frac{\chi(p)}{p}\Big)^4
\sum_{\nu_1,\ldots,\nu_4=0}^{\infty}
\frac{\chi(p)^{{\nu_1}+{\nu_2}+{\nu_3}+{\nu_4}}}{p^{m(\mnu)}},
$$
where $m(\mnu)$ is defined in section \ref{geom-series}.
On refamiliarising oneself with the notation $S_0^\ve(z)$ introduced
in \eqref{30-S0pm}, lemma \ref{31-s0-} therefore yields
$$
\sigma_p'(\ma{d},\ma{D})=
\Big(1-\frac{1}{p}\Big)^4
S_0^{+}(1/p)=
\frac{1+2/p+6/p^2+2/p^3+1/p^4}{(1+1/p)^2},
$$
if $p\equiv 1 \bmod 4$, and 
$$
\sigma_p'(\ma{d},\ma{D})=
\Big(1+\frac{1}{p}\Big)^4
S_0^{-}(1/p)=\frac{(1-1/p^2)^2}{(1+1/p^2)},
$$
if $p\equiv 3 \bmod 4$. Thus 
$\sigma_p'(\ma{d},\ma{D})=1+O(1/p^2)$ for $p\nmid \D D$.

Suppose now that $p\mid \D D$. Then \eqref{eq:tenn-} implies that 
$$
\sigma_p'(\ma{d},\ma{D})\ll 
\sum_{\nu_1,\ldots,\nu_4=0}^{\infty}
\frac{1}{p^{m(\ma{n})-m(\mmu)}}\ll 1
$$
where $\ma{n}=
(\max\{\mu_1,\lambda_1+\nu_1\},\ldots, \max\{\mu_4,\lambda_4+\nu_4\})$.
Putting this together with our treatment of the factors  corresponding
to $p\nmid \D D$, we are easily led to the desired upper bound in \eqref{eq:suffice}.
This therefore concludes the proof of the lemma. 
\end{proof}

\section{The d\'enouement}

Take $D=4$ and $L=2\eta/3$ in lemmas \ref{lem:M(B)} 
and lemma \ref{lem:M(B)e1}, and let $\ve>0$ be given. We therefore deduce that
$$
N(B)=N_2(B)+O\big(B\log (B)^{1-\eta/3+\ve} \big)
$$
via \eqref{eq:N1E0} and \eqref{eq:vas}, where $\card\TNS(\QQ)_\tors
\times N_2(B)$ is equal to
\begin{align*}
B\sum_{\substack{\mm\in \Sigma\\ \faa\in \Sigma'}}\mu(\faa)
\hspace{-0.1cm}
\sum_{\substack{\ell \leq \log (B)^{2\eta/3} \\ 2\nmid \ell}} 
\hspace{-0.1cm}
\mu(\ell)
\hspace{-0.1cm}
\sum_{\substack{\fb_1,\ldots,\fb_4 \in \widehat{\fD}\\
\N(\fb_j)\leq \log (B)^{4}}} 
\prod_{j=1}^4 \mu(\fb_j) 
c_{\bd,\ma{D},\mcal{R}_\mm}
\hspace{-0.1cm}
\sum_{\substack{t \in \fD\cap[1,B]\\ \gcd(t,\N(\faa))=1\\ \N(\bigcap
    \fb_j)\mid t}}
\hspace{-0.1cm}
\frac{r(t/\N(\bigcap \fb_j))}{t}.
\end{align*}
Here $ c_{\bd,\ma{D},\mcal{R}_\mm}$ is given by 
\eqref{eq:cfinal}, with $\bd,\ma{D}$ being given by \eqref{eq:thed}
and $\mcal{R}_\mm$ given by \eqref{eq:RR}. 
The following simple result allows us to carry out the inner summation over
$t$.

\begin{lemma}\label{lem:r}
Let $m\in \ZZp$ and let $T\geq 1$. Then for any $\ve>0$ we have 
$$
\sum_{\substack{t \in \fD\cap[1,T]\\ \gcd(t,m)=1}}
\frac{r(t)}{t}= C_m \log (T)
+O(m^\ve),
$$
where
$$
C_m=2L(1,\chi) 
\prod_{p\equiv 3
  \bmod{4}}\Big(1-\frac{1}{p^{2}}\Big)
\prod_{\substack{p\mid m\\ p\equiv 1\bmod{4}}} \Big(1-\frac{1}{p}\Big)^2.
$$
\end{lemma}

\begin{proof}
Recall the definition \eqref{18-D} of the set $\fD$.
We consider the Dirichlet series
$$
F_m(s)=\sum_{\substack{t\in \fD\\ \gcd(t,m)=1}} 
\frac{r(t)}{t^s}
=4\prod_{\substack{p\nmid m\\ p\equiv 1 \bmod{4}}}
\sum_{k\geq 0}\frac{k+1}{p^{ks}}
=4\prod_{\substack{p\nmid m\\ p\equiv 1 \bmod{4}}}
\Big(1-\frac{1}{p^s}\Big)^{-2},
$$
for $\Re e(s)>1$. 
Thus we may write $F_m(s)=F_1(s)H(s)$, with 
$$
H(s)=\prod_{\substack{p\mid m\\ p\equiv 1 \bmod{4}}}
\Big(1-\frac{1}{p^s}\Big)^{2}=\sum_{d=1}^\infty \frac{h(d)}{d^s},
$$
say, for an appropriate arithmetic function $h$. One calculates
\begin{align*}
F_1(s)=4\zeta(s)L(s,\chi) \Big(1-\frac{1}{2^{s}}\Big)\prod_{p\equiv 3
  \bmod{4}} \Big(1-\frac{1}{p^{2s}}\Big), 
\end{align*}
whence an application of Perron's formula yields
$$
\sum_{\substack{t \in \fD\cap[1,T]}}
\frac{r(t)}{t}= C_1\log (T) +O(1),
$$
with $C_1$ defined as in the statement of the lemma. 

We may complete the proof of the lemma using an argument based on
Dirichlet convolution. Thus it follows that 
\begin{align*}
\sum_{\substack{t \in \fD\cap[1,T]\\ \gcd(t,m)=1}}
\frac{r(t)}{t}
&= \sum_{\substack{d \mid m^2\\ d \in \fD\cap [1,T]}}
\frac{h(d)}{d}\Big(C_1 \log\Big(\frac{T}{d}\Big) +O(1)\Big)\\
&= \prod_{\substack{p\mid m\\ p\equiv 1\bmod{4}}} \Big(1-\frac{1}{p}\Big)^2
C_1 \log (T) +O\Big(\sum_{d\mid m^2} \frac{|h(d)|\log (2d)}{d}\Big).
\end{align*}
The main term confirms the prediction in the statement of the lemma
and the error term is easily seen to be $O(m^\ve)$ for any $\ve>0$, 
which is satisfactory. 
\end{proof}

Making the obvious change of variables it now follows from lemma
\ref{lem:r} that 
\begin{align*}
\sum_{\substack{t \in \fD\cap[1,B]\\ \gcd(t,\N(\faa))=1\\ \N(\bigcap
    \fb_j)\mid t}}
\frac{r(t/\N(\bigcap \fb_j))}{t}
&= \frac{c_{\faa,\fbb}\log (B/\N(\bigcap
    \fb_j))}{\N(\bigcap
    \fb_j)}
+O(1)\\
&= \frac{c_{\faa,\fbb}\log (B)}{\N(\bigcap
    \fb_j)}
+O(1),
\end{align*}
where 
$$
c_{\faa,\fbb}=
\begin{cases}
C_{\N(\faa)} & \mbox{if $\gcd(\N(\bigcap \fb_j),\N(\faa))=1$,}\\
0 & \mbox{otherwise}.
\end{cases}
$$
In particular it is clear that 
$c_{\faa,\fbb}=O(1)$.
Applying lemma \ref{lem:c-upper} it is easy to conclude that 
the overall contribution to $N_2(B)$ from the error term in this estimate is
\begin{align*}
&\ll 
B\sum_{\ell \leq \log (B)^{2\eta/3}}
\ell^\ve
\sum_{\N(\fb_j)\leq \log (B)^{4}} 
\frac{(\N(\fb_1)\cdots \N(\fb_4))^\ve}{[
\N(\fb_1)\N(\fb_2), \ldots, \N(\fb_3)\N(\fb_4)]}\\
&\ll 
B\log (B)^{2\eta/3+\ve} \sum_{b_1,\ldots,b_4\leq \log (B)^{4}}  
\frac{1}{[b_1b_2, \ldots, b_3b_4]}\\
&
\leq B\log (B)^{2\eta/3+\ve} \prod_{p\leq \log (B)^4} S_0^+(1/p),
\end{align*}
in the notation of \eqref{30-S0pm}. This is therefore seen to be 
$O(B\log (B)^{2\eta/3+\ve})$ via lemma \ref{31-s0-}.

In conclusion, we may  write
$$
N(B)=N_3(B)+O\big(B\log (B)^{1-\eta/3+\ve} \big),
$$
where now
\begin{align*}
N_3(B)=
\frac{B\log (B)}{\card\TNS(\QQ)_\tors}
\sum_{\substack{\mm\in \Sigma\\ \faa\in \Sigma'}}\mu(\faa)
\sum_{\substack{\ell \leq \log (B)^{2\eta/3} \\ 2\nmid \ell}} 
\mu(\ell)
\sum_{\substack{\fb_1,\ldots,\fb_4 \in \widehat{\fD}\\
\N(\fb_j)\leq \log (B)^{4}}} \frac{c_{\faa,\fbb}c_{\bd,\ma{D},\mcal{R}_\mm}}{\N(\bigcap
    \fb_j)} \prod_{j=1}^4 \mu(\fb_j) .
\end{align*}
Here we have used \eqref{defeta} to observe that $1-\eta/3>2\eta/3$. 
Finally, through a further application of lemma \ref{lem:c-upper}, it
is now a trivial matter to re-apply the proof of lemma \ref{lem:M(B)}
to show that the summations over $\ell$ and 
$\fb_j$ can be extended to infinity with error 
$O(B\log (B)^{1-\eta/3+\ve} )$. 
This therefore leads to the final outcome that 
$$
N(B)=cB\log (B)+O\big(B\log (B)^{1-\eta/3+\ve} \big),
$$
for any $\ve>0$, where
\begin{equation}
\label{equ.final.constant}
c=\frac 1{\card\TNS(\QQ)_\tors}
\sum_{\substack{\mm\in \Sigma\\ \faa\in \Sigma'}}\mu(\faa)
\sum_{\substack{\ell=1\\ 2\nmid \ell}}^\infty
\mu(\ell)
\sum_{\substack{\fb_1,\ldots,\fb_4 \in \widehat{\fD}}}
\frac{c_{\faa,\fbb}c_{\bd,\ma{D},\mcal{R}_\mm}}{\N(\bigcap
    \fb_j)} \prod_{j=1}^4 \mu(\fb_j) .
\end{equation}
Here $c_{\bd,\ma{D},\mcal{R}_\mm}$ is given by 
\eqref{eq:cfinal}, with $\bd,\ma{D}$ being given by \eqref{eq:thed}. 
\section{Jumping down}
We shall now relate the constant $c$
defined by equation~\eqref{equ.final.constant}
with the one expected, as required to complete
the proof of theorem~\ref{theo.main}.

\subsection{Expression in terms of volumes}
Let us first recall that the adelic set $\torsQ_\nn(\Adeles_\QQ)$
comes with a canonical measure which is defined as follows.
The canonical line bundle on $\omega_{\torsQ_\nn}$ is trivial
\cite[lemme 3.1.12]{peyre:cercle}
and the invertible functions on $\torsQ_\nn$ are constant.
Therefore up to multiplication by a constant there exists
a unique section $\formon{\torsQ_\nn}$ of $\omega_{\torsQ_\nn}$
which does not vanish.
By \cite[\S2]{weil:adeles}, this form defines a measure 
$\measure_{\torsQ_\nn,v}$ on $\torsQ_\nn(\QQ_v)$ for any
place $v$ of $\QQ$. According to \cite[lemme 3.1.14]{peyre:cercle},
the product $\prod_v\measure_{\torsQ_\nn,v}$ converges
and defines a measure on $\torsQ_\nn(\Adeles_\QQ)$. By the product
formula, this measure does not depend on the choice of the section
$\formon{\torsQ_\nn}$.
Let us now describe explicitly how to construct such
a section $\formon{\torsQ_\nn}$. 
\begin{nota}
Let $\mathcal X_\nn$ be the
subscheme of $\Aff^8_\ZZ=\Spec(\ZZ[X_j,Y_j,1\leq j\leq 4])$
defined by the equations~\eqref{equ.torcz}. Then $\torscZ_\nn$ is
the product $\mathcal X_\nn\times\Aff^2_\ZZ$.
We denote by $\mathcal X^\circ_\nn$ the complement of the origin
in $\mathcal X_\nn$.
For three distinct elements $j,k,l$
of $\onefour$, let us denote by $P_{j,k,l}$ the quadratic form
\[\Delta_{j,k}n_l(X^2_l+Y^2_l)+\Delta_{k,l}n_j(X_j^2+Y_j^2)
+\Delta_{l,j}n_k(X_k^2+Y_k^2).\]
Then we have the relations
\begin{align*}
a_jP_{k,l,m}+a_kP_{l,m,j}+a_lP_{m,j,k}+a_mP_{j,k,l}&=0\\
b_jP_{k,l,m}+b_kP_{l,m,j}+b_lP_{m,j,k}+b_mP_{j,k,l}&=0
\end{align*}
whenever $\{j,k,l,m\}=\onefour$.
Since $\Delta_{1,2}=1$, the scheme $\mathcal X^\circ_\nn$ is the complete
intersection in $\Aff^6_\ZZ\setminus\{0\}$
of the quadrics defined by $P_{1,2,3}$ and $P_{1,2,4}$.
Therefore the corresponding Leray form is a nonzero
section of the canonical line bundle $\omega_{\mathcal X^\circ_{\nn,\QQ}}$.
On $\Aff^2_\ZZ$, we may take the natural form 
$\frac\partial{\partial X_0}\wedge\frac\partial{\partial Y_0}$.
The exterior product of these forms 
gives a form on an open subset of $\mathcal Y_\nn$,
and by restriction a form $\formon{\torsQ_\nn}$
on $\torsQ_\nn$ which does not vanish.
We denote by $\measure_{\nn,v}$ the corresponding measure on
$\mathcal Y_n(\QQ_v)$ for $v\in\Val(\QQ)$.
\end{nota}
\begin{lemma}
\label{lemma.volumeprime}%
Let $\mm\in\Sigma$ and $\faa\in\Sigma'$. 
Let $\fbb=(\fb_j)_{j\in\onefour}$ belong to $\widehat{\mathfrak{D}}^4$.
Let~$\ell$ be an odd integer.
Let $d_j$ and $D_j$ be defined by formula~\eqref{eq:thed}.
Then for any prime number $p$ we have
\[\measure_{\nn,p}(\mathcal D^3_{\mm,\faa,\fbb,\ell,p})
=\beta_pp^{-v_p\bigl(\N\bigl(\bigcap_j\fb_j\bigr)\bigr)}
\lim_{n\to+\infty}p^{-6n}N_{\bd,\mathbf{D}}(p^n),\]
where
\[\beta_p=\begin{cases}
\frac 12&\text{if $p=2$},\\
1-\frac 1{p^2}&\text{if $p\equiv 3\mod 4$},\\
\left(1-\frac 1p\right)^2&\text{if $p\mid \prod_j\N(\fa^+_j)$ 
and $p\equiv 1\mod 4$},\\
0&\text{if $p\mid\prod_j\N(\fa^+_j)$ and $p\mid\prod_j\N(\fb_j)$},\\
1&\text{otherwise.}
\end{cases}\]
\end{lemma}
\begin{proof}
In the product $\mathcal X_{N(\faa\fbb)\mm}\times\Aff^2_\ZZ$,
the domain $\mathcal D^3_{\mm,\faa,\fbb,\ell,p}$ decomposes as
a product.
The projection on the eight coordinates $X_j,Y_j$, where $j\in\onefour$,
gives an isomorphism from
the complete intersection in $\Aff^{10}_\ZZ-\{0\}$
given by the equations
\[L_j(U,V)=n_j(X_j^2+Y_j^2)\]
for $j\in\onefour$ to the scheme $\mathcal X^\circ_\nn$.
Moreover this isomorphism map is compatible with the respective Leray forms.
Since the measure defined by the Leray measure coincides
with the counting measure (see, for example, 
\cite[proposition 1.14]{lachaud:waring}),
the volume of the first component is equal to
$\lim_{n\to+\infty}p^{-6n}N_{\bd,\mathbf{D}}(p^n)$.
The measure on $\Aff^2_\ZZ$ is the standard Haar measure.
On the other hand, the image of the domain in $\ZZ_p^2$ may be described
as follows:
\begin{itemize}
\item It is $\ZZ[\ci]_{1+\ci}\setminus(1+\ci)\ZZ[\ci]_{1+\ci}$ if $p=2$;
\item It is $\ZZ_p^2\setminus p\ZZ_p^2$ if $p\equiv 3\mod 4$;
\item It is the set of $(x,y)\in\ZZ_p^2$
such that $p$ does not divide $\N(x+\ci y)$ if $p\mid\prod_j\N(\fa^+_j)$,
the prime~$p$ does not divide $\N(\bigcap_j\fb_j)$
and $p\equiv 1\mod 4$;
\item It is empty if $p\mid\prod_j\N(\fa^+_j)$ and $p\mid\prod_j\N(\fb_j)$;
\item It is $(\bigcap_j\fb_j)\ZZ_p[\ci]$ otherwise.
\end{itemize}
Therefore $\beta_pp^{-v_p\bigl(\N\bigl(\bigcap_j\fb_j\bigr)\bigr)}$
is the volume of this component.
\end{proof}
\begin{lemma}
\label{lemma.volumeinfinite}%
Let $\mm\in\Sigma$ and $\faa\in\Sigma'$. 
Let $\fbb=(\fb_j)_{j\in\onefour}$ belong to $\widehat{\mathfrak{D}}^4$.
We put $\nn=\N(\faa\fbb)\mm$.
Let~$\ell$ be an odd integer. For any real number $B$, we have
\[\measure_{\nn,\infty}(\mathcal D^3_{\mm,\faa,\fbb,\ell,\infty}(B))=
\frac{4L(1,\chi)
\pi^4}{\prod_{j=1}^4n_j}\Vol(\mathcal R_\mm)f(B),\]
where $f(B)=\int_0^{\log(B)}ue^u \rd u=B\log(B)-B+1$.
\end{lemma}
\begin{proof}
The functions $U$ and $V$ on $\torscZ_\nn=\mathcal X_\nn\times\Aff^2$ 
are induced by functions on $\mathcal X_\nn$ which we shall
also denote by $U$ and $V$. Let $H_{F,\infty}:\mathcal X_\nn(\RR)\to\RR$
and $H_{E,\infty}:\RR^2\to\RR$ be defined by
\[H_{F,\infty}(R)=\max(|U(R)|,|V(R)|)\qquad\text{and}\qquad
H_{E,\infty}(x_0,y_0)=x_0^2+y_0^2.\]
Then the domain $\mathcal D^3_{\mm,\faa,\fbb,\ell,\infty}(B)$ is the set
of $(R,(x_0,y_0))\in\mathcal X_\nn(\RR)\times\RR^2$ such that
\[H_{F,\infty}(R)\geq 1,\qquad H_{E,\infty}(x_0,y_0)\geq 1,\qquad\text{and}\qquad
H_{F,\infty}(R)^2H_{E,\infty}(x_0,y_0)\leq B.\]
Let us denot by $v_{\nn,1}(t)$ (resp. $v_2(t)$) the
volume of the set of $R\in\mathcal X_\nn(\RR)$ (resp. $(x_0,y_0)\in\RR^2$)
such that $H_{F,\infty}(R)\leq t$ (resp. $H_{E,\infty}(x_0,y_0)\leq t$).
Then the functions $v_{\nn,1}$ and $v_2$ are monomials of respective
degrees $2$ and $1$. Therefore the volume
of the domain $\mathcal D^3_{\mm,\faa,\fbb,\ell,\infty}(B)$
is given by
\[v_{\nn,1}(1)v_2(1)\int_{\substack{t\geq 1, u\geq 1\\t^2u\leq B}} 2t\rd u\rd t
=v_{\nn,1}(1)v_2(1)f(B).\]
To compute the value of $v_{\nn,1}(1)$, we may use the change
of variables $x_j'=\sqrt{|n_j|}x_j$ and $y'_j=\sqrt{|n_j|}y_j$.
Since the Leray form may be locally described as
\[\left|\begin{matrix}\frac{\partial P_{1,2,3}}{\partial X_1}
&\frac{\partial P_{1,2,3}}{\partial X_2}\\
\frac{\partial P_{1,2,4}}{\partial X_1}
&\frac{\partial P_{1,2,4}}{\partial X_2}\\\end{matrix}\right|^{-1}
\rd X_3 \rd X_4\prod_{j=1}^4\rd Y_j=(4\Delta_{3,4}X_1X_2)^{-1}
\rd X_3\rd X_4\prod_{j=1}^4\rd Y_j\]
we get that $v_{\nn,1}(1)=v_{\boldsymbol\varepsilon,1}(1)\prod_{j=1}^4n_j^{-1}$,
where $\varepsilon_j=\sgn(n_j)=\sgn(m_j)$. It follows
that $v_{\nn,1}(1)=(\prod_{j=1}^4n_j)^{-1}\pi^4\Vol(\mathcal R_\mm)$.
We conclude the proof with the equalities $v_2(1)=\pi=4L(1,\chi)$.
\end{proof}
\begin{prop}
Let $\mm\in\Sigma$ and $\faa\in\Sigma'$. 
Let $\fbb=(\fb_j)_{j\in\onefour}$ belong to $\widehat{\mathfrak{D}}^4$.
Let~$\ell$ be an odd integer. Then
\[\frac{c_{\faa,\fbb}c_{\bd,\ma{D},\mcal{R}}}{\N(\bigcap
    \fb_j)}f(B)=\Vol(\mathcal D^3_{\mm,\faa,\fbb,\ell}(B)),\]
where $f(B)=B\log(B)-B+1$.
\end{prop}
\begin{proof}
This follows from lemmata~\ref{lemma.volumeprime}
and \ref{lemma.volumeinfinite}: indeed, by
\cite[(2.8)]{bretechebrowning:4linear}, 
we have $\omega_{\mathcal R_\mm}(\infty)=
\pi^4\Vol(\mathcal R_\mm)$ and
\[\prod_{p\in\primes}\sigma_p(\bd,\mathbf D)=\frac 1{\prod_{j=1}^4n_j}
\prod_{p\in\primes}\lim_{k\to+\infty}p^{-6k}N_{\bd,\mathbf{D}}(p^k)\]
where $\nn=N(\faa\fbb)\mm$.
\end{proof}
\Subsection{Moebius reversion}

\begin{prop}
Let $B$ be a real number and $\mm$ belong to $\Sigma$.
Then 
\[\Vol(\mathcal D_{\mm}(B))=\sum_{\faa\in\Sigma'}\sum_{\fbb\in\widehat
{\mathfrak D}^4}
\sum_{\ell\operatorname{ odd}}
\mu(\faa)\mu(\fbb)\mu(\ell)
\Vol(\mathcal D^3_{\mm,\faa,\fbb,\ell}(B)).\]
\end{prop}
\begin{proof}
For any $\llambda\in T_\DDelta(\QQ)\cap\ZZ_\DDelta$,
and any $\nn\in\ZZ^4$,
the multiplication by $\llambda$ defines an isomorphism
from $\torscZ_{N(\llambda)\nn}$ to $\torscZ_{\nn}$. Therefore
it sends the canonical form on the adelic set
$\torscZ_{N(\llambda)\nn}(\Adeles_\QQ)$ onto the
canonical form on $\torscZ_{\nn}(\Adeles_\QQ)$.
Therefore the volume of $\mathcal D^3_{\mm,\faa,\fbb,\ell}(B)$
coincides with the volume of its image in $\torscZ_{\mm}(\Adeles_\QQ)$.
The formula then follows from lemma~\ref{lemma.firstinversion}
and the proofs of propositions~\ref{prop.secondinversion}
and~\ref{prop.aftermoebius}.
\end{proof}

\Subsection{The constant}

\begin{prop}
We have
\[C_H(S)B\log(B)=\frac 1{\card\TNS(\QQ)_\tors}
\sum_{\mm\in\Sigma}\Vol(\mathcal D_{\mm}(B))+O(B).\]
\end{prop}
\begin{proof}
The following proof is based upon the ideas 
of Per Salberger~\cite{salberger:tamagawa}
as des\-cri\-bed in~\cite[\S5.3]{peyre:cercle}.

We may identify $\omega_S^{-1}$ with $\mathcal O_{S'}(1)$
(see lemma~\ref{lemma.linearomega}). This enables us
to define an adelic metric on $\omega_S^{-1}$ by
\[\Vert y\Vert_v=\begin{cases}
\min\left(\left|\frac{y}{X_0(x)}\right|,
\left|\frac{y}{X_1(x)}\right|,
\left|\frac{y}{X_2(x)}\right|,
C\left|\frac{y}{X_3(x)}\right|,
C\left|\frac{y}{X_4(x)}\right|\right)&\text{if $v=\infty$,}\\
\min_{0\leq i\leq 4}\left(\left|\frac{y}{X_i(x)}\right|_v
\right)&\text{otherwise.}
\end{cases}\]
for $x\in S'(\QQ_v)$ and $y$ in the corresponding fiber
$\mathcal O_{S'}(1)_x\otimes\QQ_v$, with the constant $C$ defined
in notation~\ref{nota.height}. This adelic metric defines the height
used throughout the text.
Let $v$ be a place of $\QQ$. We denote by $\meas_{H,v}$
the measure on $S(\QQ_v)$ corresponding
to the adelic metric on $\omega_S^{-1}$ (see \cite[\S2]{peyre:fano}).
Let us recall that on a split torus $\Gm^n$,
the form $\bigwedge_{j=1}^n\xi_j^{-1}d\xi_j$, where
$(\xi_j)_{1\leq j\leq n}$ is a basis of $X^*(\Gm^n)$,
up to sign does not depend on the choice of the basis.
Therefore there is a canonical Haar measure on $\TNS(\QQ_v)$
which we shall denote by $\meas_{\TNS,v}$.
Let $\mm$ be an element of $\Sigma$.
The functions $H_w$ defined in definition~\ref{defi.localheight} 
may been seen as the composite of the metrics on $\omega_S^{-1}$
with the natural morphism from the universal torsor
$\torsQ_\mm$ to the line bundle $\omega_S^{-1}$.
Let $U\neq\emptyset$ be an open subset of $\pi_\mm(\torsQ_\mm(\QQ_v))$.
According to \cite[lemme 3.1.14]{peyre:cercle} and
\cite[\S4.4]{peyre:torseurs}, if $s:U\to\torsQ_\mm(\QQ_v)$ is a continuous
section of $\pi_\mm$, then the measure $\meas_{\mm,v}$ is characterised by
the relation
\begin{equation}
\label{equ.meas.torsor}
\int_{\pi_\mm^{-1}(U)}f(y)\meas_{\mm,v}(y)=\int_U\int_{\TNS(\QQ_v)}f(t.s(x))
H_{v}(t.s(x))\meas_{\TNS,v}(t)\meas_{H,v}(x)
\end{equation}
for any continuous function $f$ on $\pi_\mm^{-1}(U)$ with compact support.

By lemmata~\ref{lemma.most.places} and~\ref{lemma.bad.places},
for any prime number $p$,
$\mathcal D_{\mm,p}$ is a fundamental domain in $\torsQ_\mm(\QQ_p)$
under the action of $\TNS(\QQ_p)$ modulo $\TNS(\ZZ_p)$.
Moreover, by definition, we have that $\mathcal D_{\mm,p}$ is
contained in $\widehat\pi_\mm^{-1}(\torsiZ(\ZZ_p))$ and thus 
$H_p$ is equal to $1$ on $\mathcal D_{\mm,p}$. 
Using~\eqref{equ.meas.torsor}, we get that
\[\meas_{\mm,p}(\pi_\mm^{-1}(U)
\cap\mathcal D_{\mm,p})=\meas_{\TNS,p}(\TNS(\ZZ_p))\meas_{H,v}(U)\]
for any open subset $U$ of $\pi_\mm(\mathcal D_{\mm,p})$.

The maps $\log\circ H_F$ and $\log\circ H_E$ define
a map $\log_\infty:\torsQ_\mm(\RR)\to\Pic(S)\dual\otimes_\ZZ\RR$
and using $\log_\infty\times \pi_\mm$ we get a homeomorphism
\[\torsQ_\mm(\RR)\to\Pic(S)\dual\otimes_\ZZ\RR\times\pi_\mm(\torsQ_\mm(\RR)).\]
Let
\[\TNS^1(\RR)=\{\,t\in\TNS(\RR)\mset \forall\chi\in\Pic(S),
\vert\chi(t)\vert=1\,\}.\]
Then for any real number $B$ and any open 
subset $U$ of $\pi_m(\mathcal D_{\mm,\infty}(B)$,
we get
\begin{align*}
&\meas_{\mm,\infty}(\pi_\mm^{-1}(U)\cap
\mathcal D_{\mm,\infty}(B))\\
&=\int_{\{\,y\in\Ceff(S)\dual\mset \langle\omega_S^{-1}
,y\rangle\leq \log(B)\,\}}e^{\langle\omega_S^{-1},y\rangle}\rd y\times
\omega_{\TNS}(\TNS^1(\RR))\,\omega_{H,\infty}(U)\\
&=\alpha(S)\meas_{\TNS,\infty}(\TNS^1(\RR))\,\meas_{H,\infty}(U)
f(B),
\end{align*}
where $\Ceff(S)\dual$ is the dual to the closed cone in $\Pic(S)\otimes_\ZZ\RR$
generated by the effective divisors.

Taking the product over all places of $\QQ$, we get the formula
\begin{equation}
\label{equ.global.volume}
\begin{aligned}
\meas_{\mm}(\mathcal D_{\mm}(B))&=
\alpha(S)\meas_{\TNS,\infty}(\TNS^1(\RR))
\meas_{H,\infty}(\pi_\mm(\torsQ_\mm(\RR)))\int_0^{\log(B)}ue^u\rd u\\
&\quad\times\left(\prod_{p\in\primes}
L_p(1,\Pic(\overline S))\meas_{\TNS,p}(\TNS(\ZZ_p))\right)\\
&\quad\times
\left(\prod_{p\in\primes}L_p(1,\Pic(\overline S))^{-1}
\meas_{H,p}(\pi_\mm(\torsQ_\mm(\QQ_p)))\right).
\end{aligned}
\end{equation}

By lemma~\ref{lemma.idelic}, the map from $\TNS(\QQ)$ 
to $\bigoplus_{p\in\primes}X_*(\TNS)_p$ is surjective. It follows
that
\[\TNS^1(\Adeles_\QQ)=(\TNS^1(\RR)\times\prod_{p\in\primes}\TNS(\ZZ_p))
.\TNS(\QQ)\]
and we get an exact sequence
\[1\longrightarrow\TNS(\QQ)_\tors\longrightarrow
\TNS^1(\RR)\times\prod_{p\in\primes}\TNS(\ZZ_p)
\longrightarrow \TNS^1(\Adeles_\QQ)/\TNS(\QQ)\longrightarrow 1.\]
Combining this with formula~\eqref{equ.global.volume} and the definitions
of the adelic measures, we get the formula
\[\meas_{\mm}(\mathcal D_{\mm}(B))=
\card\TNS(\QQ)_\tors\alpha(S)\tau(\TNS)\,\meas_H(\pi_m(\torsQ_\mm(\Adeles_\QQ)))
\int_0^{\log(B)}ue^u\rd u,\]
where $\tau(\TNS)$ denotes the Tamagawa number of $\TNS$.
By Ono's main theorem~\cite[\S 5]{ono:tamagawa}, $\tau(\TNS)$ is equal to 
$\card H^1(\QQ,\Pic(\overline S)
/\card\sha^1(\QQ,\TNS)$ and using Salberger's 
argument \cite[proof of lemma 6.17]{salberger:tamagawa}
and prop.~\ref{prop.uniquem},
any point in $S(\Adeles_\QQ)^{\Br}$ belongs to exactly $\card\sha^1(\QQ,\TNS)$
sets of the form $\pi_\mm(\torsQ_\mm(\Adeles_\QQ))$.
This concludes the proof of the proposition.
\end{proof}

\let\bold\mathbf
\ifx\undefined\bysame
\newcommand{\bysame}{\leavevmode\hbox to3em{\hrulefill}\,}
\fi
\ifx\undefined\numero
\newcommand{\numero}{$\hbox{n}^\circ$}
\fi
\ifx\undefined\andname
\newcommand{\andname}{and }
\fi
\ifx\undefined\comma
\newcommand{\comma}{,}
\fi

\end{document}